\pdfoutput=1
\RequirePackage{fix-cm}
\documentclass[smallextended,envcountsect,envcountsame]{svjour3}
\smartqed

\usepackage{amssymb}
\usepackage{algorithm}
\usepackage{algpseudocode}
\usepackage{bm}
\usepackage{booktabs}
\usepackage{capt-of}
\usepackage{color,colortbl}
\usepackage{float}
\usepackage{hyphenat}
\usepackage{makecell}
\usepackage{multirow}
\usepackage{mathtools}
\usepackage{pifont}
\usepackage{siunitx}
\usepackage{tabularx,array}
\usepackage[sort,compress]{cite}
\usepackage{amsfonts}
\usepackage{graphicx}
\usepackage{tikz}
\usetikzlibrary{decorations.pathreplacing,tikzmark,calc,positioning,arrows.meta}
\usepackage{wrapfig}
\usepackage{epstopdf}
\usepackage{enumitem}
\usepackage[hidelinks]{hyperref}
\usepackage[capitalise,noabbrev,nameinlink]{cleveref}

\makeatletter
\providecommand{\theHALG@line}{\thealgorithm.\arabic{ALG@line}}
\makeatother

\ifpdf
  \DeclareGraphicsExtensions{.eps,.pdf,.png,.jpg}
\else
  \DeclareGraphicsExtensions{.eps}
\fi

\setlist[enumerate]{leftmargin=.5in}
\setlist[itemize]{leftmargin=.5in}

\numberwithin{equation}{section}
\numberwithin{algorithm}{section}

\Crefname{subsection}{Section}{Sections}

\newcommand{\creflem}[1]{\hyperref[#1]{Lemma~\ref*{#1}}}
\newcommand{\crefdef}[1]{\hyperref[#1]{Definition~\ref*{#1}}}
\newcommand{\crefrem}[1]{\hyperref[#1]{Remark~\ref*{#1}}}
\crefname{algorithm}{Algorithm}{Algorithms}
\Crefname{algorithm}{Algorithm}{Algorithms}



\usepackage{amsopn}
\DeclareMathOperator{\diag}{diag}
\DeclareMathOperator{\offdiag}{off\text{-}diag}

\journalname{}
\makeatletter\def\makeheadbox{}\makeatother
\title{Simultaneous Matrix-Free HSS Compression and Factorization with Adaptive Ranks}
\titlerunning{Matrix-Free HSS Factorization with Adaptive Ranks}
\author{Elly Do \and Chao Chen\thanks{Corresponding author: Chao Chen. Submitted to the
Journal of Scientific Computing.}}
\authorrunning{E. Do and C. Chen}
\institute{E. Do \and C. Chen \at
  Department of Mathematics, North Carolina State University,
  Raleigh, NC, USA\\
  \email{ntdo2@ncsu.edu; cchen49@ncsu.edu}}
\date{}

\ifpdf
\hypersetup{
  pdftitle={Simultaneous Matrix-Free HSS Compression and Factorization with Adaptive Ranks},
  pdfauthor={Elly Do and Chao Chen}
}
\fi

\newcommand{\A}{\mathcal{L}}

\let\epsilon\varepsilon
\newcommand{\x}{\mathsf{x}}

\newcommand{\gray}[1]{{\textcolor{gray}{#1}}}

\definecolor{Gray}{gray}{0.9}
\newcolumntype{g}{>{\columncolor{Gray}}c}

\newcolumntype{H}{>{\setbox0=\hbox\bgroup}c<{\egroup}@{}}

\usepackage[normalem]{ulem}

\usepackage{pdfpages}
\graphicspath{{figs/}{}}

\begin{document}
\maketitle

\begin{abstract}
Many dense matrices in scientific computing have hierarchically low-rank
off-diagonal structure but are accessible only through routines that apply the matrix and, in general, its transpose or adjoint to vectors, that is, through matrix--vector products (matvecs). We present a fully matrix-free algorithm that compresses such a matrix into
Hierarchically Semi-Separable (HSS) form and factorizes it simultaneously with the off-diagonal ranks chosen
adaptively from a prescribed tolerance. The method uses $\mathcal{O}(k)$ matvecs with each operator,
where $k$ is the largest off-diagonal rank found. Building on fixed-rank
block nullification, it samples the matrix in small random batches and enlarges each local basis until a randomized error estimate falls below the tolerance. Unitary
nullspace updates give
$\mathcal{O}(k^3)$ local dense work, the same asymptotic cost as when the
ranks are known in advance. Because compression and elimination proceed together, no separate factorization step is required. Excluding the
matvecs, the factorization uses $\mathcal{O}(N)$ arithmetic
and storage when the off-diagonal ranks are bounded, and its inverse is applied
in $\mathcal{O}(N)$ time. Experiments on Poisson and Helmholtz Schur
complements, Mat\'{e}rn covariance matrices, and a complex
Lippmann--Schwinger equation demonstrate the factorization as both a direct
solver and a preconditioner on problems with up to millions of unknowns.

\end{abstract}

\section{Introduction}
Many scientific applications produce matrices that are dense but
\emph{data sparse}: their off-diagonal blocks admit accurate low-rank
approximations. Hierarchical rank-structured formats use this property
to reduce storage and the cost of matrix operations from quadratic or
cubic to linear or nearly linear complexity. Common formats include
$\mathcal{H}$-matrices~\cite{bebendorf2008hierarchical,
borm2003introduction,hackbusch1999asparse,hackbusch2015hierarchical},
$\mathcal{H}^2$-matrices~\cite{bebendorf2008hierarchical,borm2010efficient,
borm2003introduction,hackbusch2015hierarchical}, Hierarchically
Semi-Separable (HSS) matrices~\cite{chandrasekaran2006afast,
chandrasekaran2007afast,xia2010superfast}, and Hierarchically
Off-Diagonal Low-Rank (HODLR) matrices~\cite{ambikasaran2013a,
martinsoson2009afast}. Flat block and tile low-rank formats are also used in
high-performance dense and multifrontal solvers
\cite{amestoy2017complexity,boukaram2024cholesky}. We focus on the HSS format. Its weak
admissibility gives it a simpler block structure than $\mathcal{H}^2$, and its
nested bases permit $\mathcal{O}(N)$ factorization when the ranks remain
bounded, whereas the complexity of a HODLR factorization is superlinear.
In applications, however, these ranks are generally unknown in advance and
may vary among off-diagonal blocks, so a practical method must discover them as it runs.

A second challenge is access to the matrix. Many HSS construction and
factorization methods require selected entries or subblocks of the dense
matrix $A$. We instead assume only black-box matrix--vector products, or \emph{matvecs}, $x\mapsto Ax$ and $x\mapsto A^\top x$ for a real matrix; for a general complex matrix, $A^*$
replaces $A^\top$. Such \emph{matrix-free} access arises, for example, for
separator Schur complements applied through sparse solves~\cite{davis2016survey},
kernel matrices applied by FFTs or fast multipole methods
\cite{greengard1987fast,askham2021fmm2d}, and Hessians applied through
forward and adjoint solves~\cite{ambartsumyan2020hierarchical,
hartland2023hierarchical}. Although these matvecs can be fast, a global sketch mixes diagonal and off-diagonal contributions. Hierarchical compression from matvecs alone is therefore nontrivial. A one-time
factorization is valuable in each of these settings because the same
operator, or one that changes only through its right-hand side, recurs
across many solves: a sparse direct solver's Schur complement reappears
at every time step of an implicit scheme, a kernel matrix is solved for
many query points or realizations, and a Hessian is solved repeatedly
inside an optimization or inverse-problem iteration. A factorization
built once amortizes over all of these, where a matrix-free iterative
solve alone would repeat its work for every new right-hand side.

We address these coupled challenges with an algorithm that compresses an HSS
approximation and factorizes it simultaneously, in one upward pass, from
matrix-free access alone and with the ranks chosen adaptively. It uses $\mathcal{O}(k)$ matvecs with each of $A$ and $A^\top$ and no matrix-entry
queries, where $k$ is the largest off-diagonal rank found. The algorithm
selects the local ranks from a prescribed tolerance and continues to use every
sample generated as the ranks grow. We first present the compression and factorization framework with the ranks assumed known, then develop its adaptive extension. Our contributions are detailed in
\Cref{ss:contributions}.

\subsection{Existing Work}\label{s:intro-related}

Related methods fall broadly into bottom-up and top-down (peeling)
approaches. \Cref{t:related-work} summarizes their main features.

\begin{table}[t]
    \centering
    \caption{Comparison of related hierarchical-matrix construction methods.
    Here \emph{matrix-free} means that the construction uses matvecs with the matrix and, when needed, its transpose or adjoint, but no entry evaluations.
    Methods are grouped as bottom-up and top-down (peeling or
    recursive-projection) approaches. In the last column, $k$ is the maximum off-diagonal rank.}
    \label{t:related-work}
    \footnotesize
    \setlength{\tabcolsep}{4pt}
    \begin{tabular}{@{}llccc@{}}
    \toprule
    \textbf{Approach} & \textbf{Method} & \textbf{Matrix-free} & \textbf{Adaptive} & \textbf{Matvecs} \\
    \midrule
     \multirow{7}{*}{Bottom-up}
       & This work & \checkmark & \checkmark & $\mathcal{O}(k)$ \\
       & Martinsson~\cite{martinsoson2011afast}  & $\times$ & $\times$ & $\mathcal{O}(k)$ \\
       & Gorman~\cite{gorman2019robust} & $\times$ & \checkmark & $\mathcal{O}(k)$ \\
       & Boukaram~\cite{boukaram2025adaptive}  & $\times$ & \checkmark & $\mathcal{O}(k)$ \\
       & Levitt~\cite{levitt2024linear}  & \checkmark & $\times$ & $\mathcal{O}(k)$ \\
       & Yesypenko~\cite{yesypenko2026randomized}  & \checkmark & $\times$ & $\mathcal{O}(k)$ \\
       & Amsel et al.~\cite{amsel2026quasioptimal} & \checkmark & $\times$ & $\mathcal{O}(k\log(N/k))$ \\
     \midrule
     \multirow{4}{*}{Top-down}
       & Lin~\cite{lin2011fast} & \checkmark & $\times$ & $\mathcal{O}(k\log N)$ \\
       & Martinsson~\cite{martinsoson2016compressing} & \checkmark & \checkmark & $\mathcal{O}(k\log N)$ \\
       & Boukaram~\cite{boukaram2019randomized} & \checkmark & \checkmark & $\mathcal{O}(k\log N)$ \\
       & Halikias--Townsend~\cite{halikias2024structured} & \checkmark & $\times$ & $\mathcal{O}(k)$ \\
    \bottomrule
    \end{tabular}
\end{table}

\paragraph{Bottom-up approaches.}
Martinsson's construction~\cite{martinsoson2011afast} uses
$\mathcal{O}(k)$ matvecs but also queries a small number of matrix entries.
Gorman et al.~\cite{gorman2019robust} make this partially matrix-free HSS
construction adaptive by enlarging Gaussian sketches in blocks and estimating
residual Frobenius norms. Once their stopping test passes, they fix the
final rank by a column-pivoted QR factorization of the entire accumulated
sample, at a cost that grows with the square of the accumulated rank. The
pruning step of \Cref{s:rank-refinement} instead examines only the most
recently added block of columns, at a cost that depends on the block size
but not on the rank. Boukaram et al.~\cite{boukaram2025adaptive} extend
the adaptive framework to $\mathcal{H}^2$ matrices, likewise with entrywise
access. Levitt and Martinsson~\cite{levitt2024linear} introduce block
nullification for fully matrix-free HSS compression using $\mathcal{O}(k)$
matvecs, with the target rank prescribed. Yesypenko and
Martinsson~\cite{yesypenko2026randomized} combine fully matrix-free sampling
with randomized strong recursive skeletonization (RSRS) to compress and factorize
$\mathcal{H}^2$ matrices, again with prescribed ranks. Amsel et
al.~\cite{amsel2026quasioptimal} analyze a modification of the
Levitt--Martinsson construction and establish a quasi-optimal HSS
approximation guarantee using $\mathcal{O}(k\log(N/k))$ matvecs.

\paragraph{Top-down approaches.}
Lin et al.~\cite{lin2011fast} introduce peeling algorithms for fixed-rank
$\mathcal{H}$ and $\mathcal{H}^2$ matrices. Martinsson
\cite{martinsoson2016compressing} makes peeling adaptive for HODLR and HSS
construction, and Boukaram et al.~\cite{boukaram2019randomized} develop an
adaptive $\mathcal{H}^2$ construction for GPUs. These methods are matrix-free
but require $\mathcal{O}(k\log N)$ matvecs. For HSS matrices, Halikias and
Townsend~\cite{halikias2024structured} replace recursive peeling by recursive
projection and require $\mathcal{O}(k)$ matvecs with prescribed rank. Their
recovery procedure, however, solves a large least-squares system, for which
they give an $\mathcal{O}(N^3k^2)$ arithmetic bound.

Thus, existing methods provide different subsets of four desirable properties:
fully matrix-free access, tolerance-driven local rank selection, $\mathcal{O}(k)$ matvecs, and simultaneous compression and factorization. The closest fully matrix-free factorization is the RSRS
method of Yesypenko and Martinsson~\cite{yesypenko2026randomized}, which
requires the ranks in advance. The closest adaptive constructions use either
entry access~\cite{gorman2019robust,boukaram2025adaptive} or
$\mathcal{O}(k\log N)$ matvecs~\cite{martinsoson2016compressing,
boukaram2019randomized}. To our knowledge, no previous HSS method combines all
four properties.

\subsection{Contributions}\label{ss:contributions}
Our main contributions are the following.

\paragraph{Simultaneous matrix-free compression and factorization.}
We construct an HSS approximation directly in factorized form from the
Gaussian sketches $Y=A\Omega$ and $Z=A^\top\Phi$, using $\mathcal{O}(k)$ matvecs with each operator and without querying any entry of
$A$. Block nullification~\cite{levitt2024linear} removes each unknown
diagonal contribution from the global sketches, and randomized range
finding~\cite{martinsson2020randomized} captures the resulting off-diagonal
spaces. Orthogonal changes of basis expose redundant local variables. We
recover their diagonal interactions from the transformed samples, eliminate
them by structured block-triangular transformations, and prove that the
remaining Schur complement has the same HSS structure. The construction can therefore recurse up the tree, so compression and factorization proceed simultaneously, in one upward pass. \Cref{fig:actions-on-A} summarizes these transformations.

\paragraph{Adaptive rank selection at fixed-rank asymptotic cost.}
We extend the fixed-rank block-nullification framework of
Levitt and Martinsson~\cite{levitt2024linear} so that the HSS ranks are
determined during the factorization from a prescribed tolerance, with no
rank known in advance. The construction draws Gaussian test columns in
small batches and reuses every sample and every accepted basis vector
across batches. At each node, a randomized Frobenius-norm test on the
reused samples decides when to stop. A freeze margin keeps sampling
fluctuations from growing a node that has converged, the nodes of a level
are completed to a common rank, and an energy-based pruning step uses the
sample of the passing batch to remove the rank overshoot caused by the
block size, examining only the most recently added block. Fresh test
columns drawn at coarser levels are carried through the stored factors, so
sampling and factorization stay interleaved in one upward pass. The kernel
that keeps this affordable is a Householder row update that extends each
nullspace basis by one block without recomputing its existing columns.
Because the update is unitary, it preserves orthogonality to working
precision. Over $J$ rank trials, it reduces the local dense work from
$\mathcal{O}(Jk^3)$ for restarted recomputation to $\mathcal{O}(k^3)$,
matching the asymptotic cost when $k$ is known in advance.

\paragraph{Complexity and numerical validation.}
Excluding the matvecs, the adaptive factorization has
$\mathcal{O}(N)$ arithmetic and storage for bounded HSS ranks. For the
standard rank-growth model in two dimensions, the factorization uses
$\mathcal{O}(N^{1.5}\log N)$ arithmetic and $\mathcal{O}(N\log N)$ storage.
In both cases,
the cost of one solve with the computed factorization matches the storage. Experiments on
Poisson and Helmholtz Schur complements and Mat\'{e}rn covariance matrices
test the factorization as both a direct solver and a preconditioner. A complex
non-Hermitian Lippmann--Schwinger equation (\Cref{app:lippmann}) tests the
method outside the real symmetric setting.

\begin{figure}[H]
\centering
\includegraphics[width=0.9\textwidth]{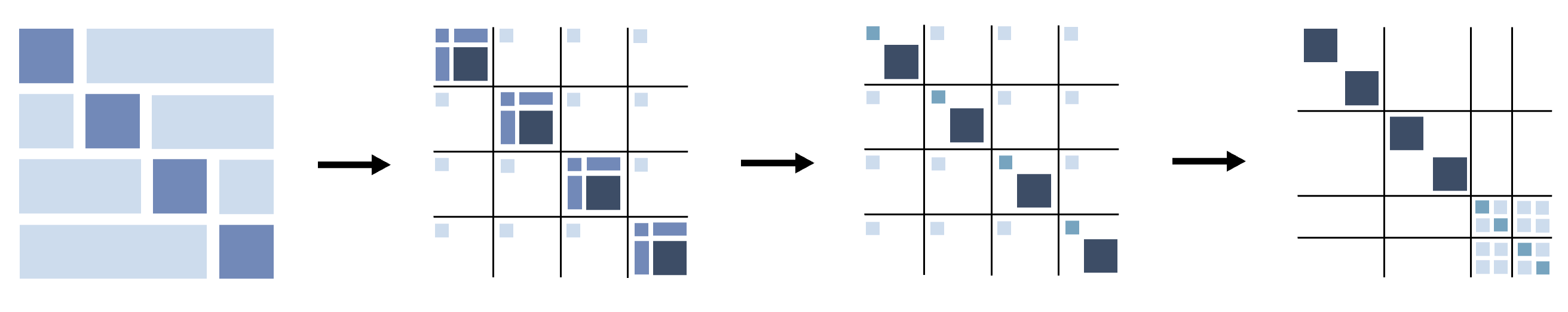}

\vspace{0.5em}
\begin{tabular}{@{}>{\centering\arraybackslash}p{0.2\textwidth}
                >{\centering\arraybackslash}p{0.22\textwidth}
                >{\centering\arraybackslash}p{0.22\textwidth}
                >{\centering\arraybackslash}p{0.22\textwidth}@{}}
$A$ & compression & elimination & permutation
\end{tabular}

\caption{Schematic overview of the algorithm. Orthogonal basis changes
compress the off-diagonal blocks, structured elimination decouples redundant
variables, and a permutation exposes a smaller Schur complement with the same
structure as $A$. Recursing on this Schur complement yields a block-diagonal factorization through
transformations that can be inverted efficiently. The partition is the
running example defined by the three-level HSS tree of
\Cref{fig:hss-tree-diagram}.}
\label{fig:actions-on-A}
\end{figure}

\subsection{Organization and Notation}\label{ss:organization-notation}
\Cref{s:background} reviews HSS structure, randomized range finding, and
randomized Frobenius-norm estimation. \Cref{s:fixed-rank} develops the
compression and factorization framework with the ranks assumed known.
\Cref{s:adaptive-rank} removes that assumption and determines the ranks from a
prescribed tolerance. \Cref{s:result}
presents the numerical experiments, and \Cref{s:end} concludes the paper.

\paragraph{Notation.}
Matrices and vectors may be real or complex. We write $M^*$ for the conjugate
transpose of $M$, $I$ for the identity matrix (of a size clear from context),
$\|\cdot\|_{\mathrm{F}}$ for the Frobenius norm, and $\mathbb{E}[\cdot]$ for
expectation. A \emph{standard Gaussian} matrix has independent
$\mathcal N(0,1)$ entries in the real case and independent standard circular
complex Gaussian entries in the complex case. Equivalently, each complex
entry is $(g_1+\mathrm{i}g_2)/\sqrt{2}$ for independent
$g_1,g_2\sim\mathcal N(0,1)$. For a matrix $M$ and index sets
$\mathcal S$ and $\mathcal T$, $M(\mathcal S,:)$ and $M(:,\mathcal T)$
denote the rows in $\mathcal S$ and the columns in $\mathcal T$, and
$M(j,:)$ and $M(:,j)$ a single row or column, as in MATLAB. All remaining
notation is introduced where it is first used.
\section{Background}\label{s:background}
This section sets the notation and collects the three ingredients the
algorithm uses: the hierarchical low-rank structure of an HSS matrix,
the randomized range finder, and the Frobenius-norm estimator for a matrix
that is globally low rank. \Cref{s:fixed-rank,s:adaptive-rank} build on
all three.

\subsection{HSS Tree and Matrix Structure}

To define an HSS matrix $A \in \mathbb{C}^{N \times N}$, we need to refer to its off-diagonal blocks.
We therefore first introduce a tree that partitions the row and column indices in the same way.
We use MATLAB-style indexing throughout: for index sets $\mathcal{I}$ and $\mathcal{J}$,
$A(\mathcal{I},\mathcal{J})$ is the submatrix of $A$ on rows $\mathcal{I}$ and
columns $\mathcal{J}$, and a colon selects all rows or columns (e.g.,
$A(\mathcal{I},:)$).

\paragraph{Running example.}
Before the formal definitions, we introduce a running example
used throughout the paper: a rank-$12$ HSS matrix of size
$96\times96$. \Cref{fig:hss-tree-diagram} partitions the index set
$[1,2,\ldots,96]$ using a binary tree, and \Cref{f:hss-structure} shows
the corresponding block partitioning and low-rank structure. The same
example is used for the schematic overview in \Cref{fig:actions-on-A} and the
worked steps in \Cref{s:fixed-rank}.

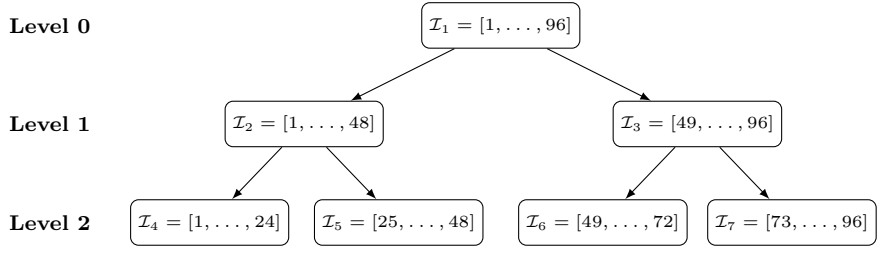
\begin{figure}
    \centering
    \begin{tikzpicture}[
        level distance=1.3cm,
        level 1/.style={sibling distance=5.2cm},
        level 2/.style={sibling distance=2.5cm},
        edge from parent/.style={draw, -latex, thin, black},
        block/.style={draw, rectangle, rounded corners=3pt, align=center,
          fill=white, minimum width=1.8cm, minimum height=0.6cm,
          font=\fontsize{6.4}{7}\selectfont},
        leaf/.style={draw, rectangle, rounded corners=3pt, align=center,
          fill=white, minimum width=1.8cm, minimum height=0.6cm,
          font=\fontsize{6.4}{7}\selectfont},
        level label/.style={font=\bfseries\footnotesize, align=left,
          text width=2.0cm}
    ]
    \node[block] (root) {$\mathcal{I}_1 = \lbrack 1, \dots, 96 \rbrack$}
        child {
            node[block] (n1) {$\mathcal{I}_2 = \lbrack 1, \dots, 48 \rbrack$}
            child { node[leaf] (l1) {$\mathcal{I}_4 = \lbrack 1, \dots, 24 \rbrack$} }
            child { node[leaf] (l2) {$\mathcal{I}_5 = \lbrack 25, \dots, 48 \rbrack$} }
        }
        child {
            node[block] (n2) {$\mathcal{I}_3 = \lbrack 49, \dots, 96 \rbrack$}
            child { node[leaf] (l3) {$\mathcal{I}_6 = \lbrack 49, \dots, 72 \rbrack$} }
            child { node[leaf] (l4) {$\mathcal{I}_7 = \lbrack 73, \dots, 96 \rbrack$} }
        };

    \node[level label, xshift=-5.5cm] at (root) {Level 0};
    \node[level label, xshift=-5.5cm] at (root |- n1) {Level 1};
    \node[level label, xshift=-5.5cm] at (root |- l1) {Level 2};

    \end{tikzpicture}
    \caption{A contiguous partitioning of the index set $\lbrack 1, \dots, 96 \rbrack$ using a perfectly balanced binary tree.
    This tree defines the running example used throughout the paper.}
    \label{fig:hss-tree-diagram}
\end{figure}

\begin{figure}
    \centering
    \begin{tikzpicture}[
        scale=0.6,
        dimline/.style={draw, thin, black, |-|},
        dimlabel/.style={font=\fontsize{9}{10}\selectfont, color=black}
    ]
        \definecolor{diagcolor}{RGB}{110, 130, 180}
        \definecolor{offdiagcolor}{RGB}{200, 220, 240}
        \definecolor{orangecolor}{RGB}{230, 140, 100}
        \def\gap{0.04}

        \begin{scope}[shift={(1,0)}]
            \foreach \x in {0, 1, 2, 3} {
                \foreach \y in {0, 1, 2, 3} {
                     \ifnum\x=\y\else
                        \fill[offdiagcolor] (\x+\gap, 3-\y+\gap) rectangle (\x+1-\gap, 4-\y-\gap);
                     \fi
                }
            }
            \foreach \k in {0, 1, 2, 3} {
                \fill[diagcolor] (\k+\gap, 3-\k+\gap) rectangle (\k+1-\gap, 4-\k-\gap);
            }
            \draw[orangecolor, ultra thick] (1+\gap, 3+\gap) rectangle (4-\gap, 4-\gap);
            \draw[orangecolor, ultra thick] (0+\gap, 0+\gap) rectangle (1-\gap, 3-\gap);
            \draw[dimline] (0, 4.2) -- (1, 4.2) node[midway, above, dimlabel] {$\mathcal{I}_4$};
            \draw[dimline] (1, 4.2) -- (2, 4.2) node[midway, above, dimlabel] {$\mathcal{I}_5$};
            \draw[dimline] (2, 4.2) -- (3, 4.2) node[midway, above, dimlabel] {$\mathcal{I}_6$};
            \draw[dimline] (3, 4.2) -- (4, 4.2) node[midway, above, dimlabel] {$\mathcal{I}_7$};
            \draw[dimline] (-0.2, 4) -- (-0.2, 3) node[midway, left, dimlabel] {$\mathcal{I}_4$};
            \draw[dimline] (-0.2, 3) -- (-0.2, 2) node[midway, left, dimlabel] {$\mathcal{I}_5$};
            \draw[dimline] (-0.2, 2) -- (-0.2, 1) node[midway, left, dimlabel] {$\mathcal{I}_6$};
            \draw[dimline] (-0.2, 1) -- (-0.2, 0) node[midway, left, dimlabel] {$\mathcal{I}_7$};
        \end{scope}

        \begin{scope}[shift={(2.5,0)}]
            \begin{scope}[shift={(4.7, 2.0)}]
                \draw[orangecolor, ultra thick] (0, 1.6) rectangle (0.6, 2.2);
                \fill[offdiagcolor] (0, 0.8) rectangle (0.6, 1.4);
                \draw[decorate, decoration={brace, amplitude=3pt}, thick]
                    (0.8, 2.2) -- (0.8, 0.8)
                    node[midway, right=3pt, font=\fontsize{9}{10}\selectfont] {Low rank (rank=$k$)};
                \fill[diagcolor] (0, 0) rectangle (0.6, 0.6);
                \node[right, font=\fontsize{9}{10}\selectfont] at (0.7, 0.3) {Full rank};
            \end{scope}
        \end{scope}
    \end{tikzpicture}
    \caption{
    A rank-$k$ HSS matrix of size $96\times96$ with $k=12$ and leaf size
    $2k=24$. The block partitioning corresponds to the leaf level of the
    tree in \Cref{fig:hss-tree-diagram}. In this example, all off-diagonal
    blocks (light blue) have rank
    exactly $k$, and so do all block rows $A_{ii^c}\equiv A(\mathcal I_i,\mathcal I_i^c)$
    and block columns $A_{i^ci}\equiv A(\mathcal I_i^c,\mathcal I_i)$, where
    $\mathcal I_i^c$ is the complement of $\mathcal I_i$ (e.g., $A_{4\,4^c}$
    and $A_{4^c\,4}$ in orange); see \Cref{eq:uv}.
    This running example illustrates the algorithm in \Cref{s:fixed-rank}.}
    \label{f:hss-structure}
\end{figure}
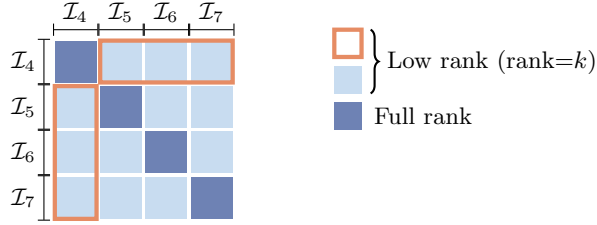

\begin{definition}[HSS Tree]\label{hss-tree-definition}
Consider the index set $\mathcal{I} = [1, 2, \dots, N]$ (row
and column indices of an $N$-by-$N$ matrix). We construct a perfectly balanced binary tree
$\mathcal{T}$ with depth $L = \mathcal{O}(\log_2(N))$.
Every node $i \in \mathcal{T}$ is
assigned a contiguous range of indices $\mathcal{I}_{i} \subseteq \mathcal{I}$ as follows:
\begin{enumerate}
    \item \textbf{Root node (level 0):} The root (node 1) contains the
    full index set $\mathcal{I}_1 = \mathcal{I}$.

    \item \textbf{Parent-child relationship:}
    For every node $i$ at level $0 \le \ell \le L-1$,  its index set $\mathcal{I}_{i}$ is partitioned evenly
    and contiguously into two disjoint subsets for its two children $\alpha = 2i$ and $\beta=2i+1$ at
    level $\ell+1$, which satisfy
    \[
    \mathcal{I}_{i} = \mathcal{I}_{\alpha} \cup \mathcal{I}_{\beta},
    \quad
    \text{and}
    \quad
    \mathcal{I}_{\alpha} \cap \mathcal{I}_{\beta} = \emptyset.
    \]

    \item \textbf{Leaf nodes (level $L$):} their
    index sets are disjoint, and their union is the full set $\mathcal{I}$.
\end{enumerate}
\end{definition}

The HSS matrix definition that follows describes the off-diagonal structure
through the complement of a node's index set. For a node $i$, let
$\mathcal I_i^c \equiv \mathcal I\setminus\mathcal I_i$. Write
$A_{ii}\equiv A(\mathcal I_i,\mathcal I_i)$ for its diagonal block, and
$A_{ii^c}\equiv A(\mathcal I_i,\mathcal I_i^c)$ and
$A_{i^ci}\equiv A(\mathcal I_i^c,\mathcal I_i)$ for the block row and block
column of $A$ outside it.

\begin{definition}[rank-$k$ HSS Matrix]\label{hss-matrix-definition}
A matrix $A \in \mathbb{C}^{N \times N}$ is a rank-$k$ Hierarchically
Semi-Separable (HSS) matrix with respect to an HSS
tree $\mathcal{T}$ from \crefdef{hss-tree-definition} if the following
two conditions hold:

\begin{enumerate}
    \item \textbf{Leaf level:}
    For each pair of distinct leaf nodes $i \ne j$, the off-diagonal block
    \[
    	A_{ij} \equiv A(\mathcal{I}_i, \mathcal{I}_j)
	\]
    admits a low-rank factorization
    \begin{equation}\label{eq:uv}
    \begin{tikzpicture}[baseline=(Aij.base)]
      \node[anchor=base] (Aij)  at (0.00,0) {$A_{ij}$};
      \node[anchor=base] (eq)   at (1.30,0) {$=$};
      \node[anchor=base] (Ui)   at (2.20,0) {$\tilde{U}_i$};
      \node[anchor=base] (Atij) at (3.10,0) {$B_{ij}$};
      \node[anchor=base] (Vj)   at (4.20,0) {$\left( \tilde{V}_{j} \right)^*$};

      \draw[thick] (-0.60,-0.35) rectangle (0.60,-1.55);   %
      \draw[thick] ( 2.00,-0.35) rectangle (2.40,-1.55);   %
      \draw[thick] ( 2.90,-0.35) rectangle (3.30,-0.75);   %
      \draw[thick] ( 3.60,-0.35) rectangle (4.80,-0.75);   %

      \node[font=\scriptsize, anchor=north] at (0.00,-1.62) {$|\mathcal{I}_i| \times |\mathcal{I}_j|$};
      \node[font=\scriptsize, anchor=north] at (2.20,-1.62) {$|\mathcal{I}_i| \times k$};
      \node[font=\scriptsize, anchor=north] at (3.10,-0.82) {$k \times k$};
      \node[font=\scriptsize, anchor=north] at (4.20,-0.82) {$k \times |\mathcal{I}_j|$};
    \end{tikzpicture}
    \end{equation}
    where the basis matrices $\tilde{U}_i\in\mathbb{C}^{|\mathcal I_i|\times k}$
    and $\tilde{V}_j\in\mathbb{C}^{|\mathcal I_j|\times k}$ have orthonormal
    columns, and the coupling
    matrix $B_{ij} \in \mathbb{C}^{k \times k}$ has smaller dimensions
    than $A_{ij}$.

    \item \textbf{Other levels:}
    For each level $\ell = L-1, \dots, 1$ and each pair of distinct
    nodes $i \ne j$ with children $(\alpha, \beta)$ and
    $(\alpha', \beta')$, define the aggregate block
    \begin{equation} \label{eq:aij_other}
    A_{ij} \equiv \begin{bmatrix}
        B_{\alpha\alpha'} & B_{\alpha\beta'} \\
        B_{\beta\alpha'} & B_{\beta\beta'}
    \end{bmatrix}.
    \end{equation}
    This matrix admits a rank-$k$ factorization
    \begin{equation} \label{eq:uv_other}
	    A_{ij} = \tilde{U}_i B_{ij} \left( \tilde{V}_{j} \right)^*
    \end{equation}
    with transfer matrices $\tilde{U}_i,\tilde{V}_j\in\mathbb{C}^{2k\times k}$
    and coupling matrix $B_{ij} \in \mathbb{C}^{k \times k}$.
    \end{enumerate}
\end{definition}

Thus, $k$ is an upper bound on the HSS ranks. Individual block rows and
columns may have smaller ranks.

\begin{remark} \label{rmk:rank}
At an internal level, $A_{ij}$ in \Cref{eq:aij_other} is a
$2k\times2k$ aggregate of coupling matrices from its children, rather than the
physical block $A(\mathcal I_i,\mathcal I_j)$. Multiplying the transfer
matrices down to the leaf bases gives a rank-$k$ factorization of the
physical block $A(\mathcal I_i,\mathcal I_j)$ in the original coordinates. This is the nested-basis property of the HSS
format~\cite{chandrasekaran2006afast,xia2010fast}.
\end{remark}

\subsection{Randomized Range Finding and Frobenius-Norm Estimation}
\label{ss:rand-range-finder}

Assume that a target rank $k > 0$ is known \emph{a priori}. For a
given matrix $A \in \mathbb{C}^{m \times n}$, the fixed-rank randomized
range finder~\cite[Algorithm~4.1]{halko2011lrobust} constructs a matrix
$Q \in \mathbb{C}^{m \times k}$ with orthonormal columns that
approximately span the range of $A$. It draws a standard Gaussian matrix
$\Omega\in\mathbb C^{n\times k}$, forms $Y=A\Omega$, and computes
$Q=\texttt{orth}(Y)$, an orthonormal basis of the range of $Y$ obtained
from a thin QR factorization. Throughout, $\texttt{qr}$ denotes a full QR
factorization. When $A$ is real, these quantities may all be
real. In practice, one commonly samples $k+p$ columns, where the small
oversampling parameter $p$ improves accuracy and reliability at little
additional cost~\cite{tropp2017practical}.

When the rank is unknown, an adaptive construction also needs an
inexpensive stopping test. Let $Q$ be the current basis and
$E\equiv(I-QQ^*)A$ the residual operator. For a block of $b$ Gaussian
probes, draw a standard Gaussian matrix $Z\in\mathbb C^{n\times b}$ and
set $\Omega=Z/\sqrt b$. The standard Gaussian trace identity
\cite{gratton2018improved} gives
\begin{equation}\label{eq:f-norm-estimate}
  \mathbb E\!\left[\|X\Omega\|_{\mathrm F}^2\right]
  =\|X\|_{\mathrm F}^2,
  \qquad X\in\mathbb C^{m\times n},
\end{equation}
so $\|E\Omega\|_{\mathrm F}^2$ is an unbiased estimator of
$\|E\|_{\mathrm F}^2$. This is stochastic trace estimation applied to
$X^*X$. Unbiasedness uses only that $\Omega$ is isotropic,
$\mathbb E[\Omega\Omega^*]=I$, and independent of
$X$~\cite{martinsson2020randomized}; Gaussianity is needed for the
concentration bounds below. Hutchinson's original estimator instead uses Rademacher
probes~\cite{hutchinson1990astochastic}.
Concentration bounds for the Gaussian estimator~\cite{gratton2018improved}
control the deviation of a realization from the true norm. This estimator
underlies adaptive, adjoint-free range
finders~\cite[Algorithm~13]{martinsson2020randomized} and our earlier block
algorithm~\cite{smith2026adaptive}. The adaptive HSS construction of
\Cref{s:adaptive-rank} applies this estimator locally, at every node and
every level, to decide when enough basis vectors have been accumulated. Each
one-sided basis construction is adjoint-free: it applies the operator
but never its adjoint. The full factorization applies this construction separately to $A$
and $A^\top$ in real arithmetic, or to $A$ and $A^*$ in complex arithmetic.
\section{The Compression and Factorization Framework} \label{s:fixed-rank}

This section develops the construction in full: block nullification to
isolate the off-diagonal samples, compression of those samples by unitary
changes of basis, recovery and elimination of the redundant diagonal blocks,
and recursion on the resulting Schur complement. The last two steps turn a
compression into a factorization. We present the construction under the assumption that the HSS rank is known, a convention shared with~\cite{levitt2024linear,yesypenko2026randomized}, which keeps the dimensions uniform and the derivation readable. This assumption is made for presentation only,
not as a proposed mode of use. Real applications rarely supply the ranks,
and for the two-dimensional problems of \Cref{s:result}, the ranks grow with
$N$, so no single $k$ would serve. \Cref{s:adaptive-rank} removes the
assumption, and everything in this section carries over unchanged.

For clarity, the analysis in this section is stated for a real rank-$k$ HSS
matrix $A \in \mathbb{R}^{N \times N}$ accessible only through the maps $x \mapsto Ax$ and $x \mapsto A^\top x$;
the symmetric cases that need only one of the two maps are described below.
To keep the dimensions uniform, we also assume that every off-diagonal block
row and block column has algebraic rank exactly $k$. If a local rank is
smaller, the construction simply retains unused basis directions. Finally,
we assume that every redundant diagonal block $X_i^{rr}$, defined
in \Cref{ss:x_rr_rs_sr}, that the construction selects is nonsingular. This is the usual local-pivot condition.
If it fails, local pivoting or a rank-revealing permutation is required
before elimination.
We present
an algorithm that constructs a factorization of $A$ in a compressed format,
which can be
used, e.g., to solve a linear system $Ax = f$ in $\mathcal{O}(N)$ complexity.

As the first step, we construct an $L$-level HSS tree $\mathcal{T}$
according to \crefdef{hss-tree-definition}. For ease of presentation, we choose
$L$ so that every leaf node has size $N/2^L = 2k$. This convention is not
part of the HSS definition. The method extends to leaves of different sizes
with only notational changes.
Then,
we draw two independent standard real Gaussian test matrices
$\Omega,\, \Phi \in \mathbb{R}^{N \times 3k}$ and form
\begin{equation} \label{e:y=ao}
Y = A\Omega
\qquad \text{and} \qquad
Z = A^\top\Phi.
\end{equation}
The number of columns, $3k$, exceeds the leaf node size $2k$. This excess
is the key to block nullification, the first step of the construction. The
remainder of this section shows how the transformations in
\Cref{fig:actions-on-A} are carried out using only the four matrices $Y$,
$\Omega$, $Z$, and $\Phi$. No further matvecs with $A$ or $A^\top$ are
needed. We compress the off-diagonal blocks
(\Cref{ss:form-u}), decouple the skeleton and redundant indices
within each node (\Cref{ss:x_rr_rs_sr}), and recurse on the
resulting Schur complement up the tree (\Cref{ss:p}). Finally,
\Cref{ss:linear-solve-matvec} describes the factorized solve and the
complexity of the complete algorithm.

When $A$ is real symmetric, the second sketch is redundant. Choosing
$\Phi=\Omega$ gives $Z=A^\top\Omega=Y$, so no separate matvecs with $A^\top$ are needed, and every step in this section may take $V_i\equiv U_i$.
This halves the matvec cost (\Cref{th:complexity-fixed}). For a general complex matrix,
the corresponding construction uses $Z=A^*\Phi$ and two independent
standard circular complex Gaussian test matrices $\Omega$ and $\Phi$. If the matrix is complex
symmetric, $A^\top=A$, access to $A^*$ is again unnecessary: choose
$\Phi=\overline{\Omega}$, so that
\[
  Z=A^*\overline{\Omega}=\overline{A\Omega}=\overline{Y},
  \qquad V_i\equiv\overline{U_i}.
\]
The transformed diagonal block is then
$X_i=U_i^*A_{ii}\overline{U_i}=P_i^\top A_{ii}P_i$ with
$P_i=\overline{U_i}$, and is again complex symmetric. The
Lippmann--Schwinger operator of \Cref{app:lippmann} is complex symmetric. We state
the algorithms for real matrices with both sketches and note these savings
where they apply.

Circular complex sampling is essential to the unitary-invariance argument
used in \creflem{lemma:tilde} and throughout the paper. Real Gaussian
sampling is not unitarily invariant over $\mathbb C$.

\begin{remark}
Matrices that arise from, e.g., the discretization of integral equations or the Gaussian
elimination of discretized elliptic operators are typically not exact HSS matrices, but they can be approximated effectively by HSS matrices~\cite{martinsson2005fast, xia2010superfast}. The algorithms of this paper are designed to be applied to such matrices and are robust on them, as \Cref{s:result} shows.
\end{remark}

\subsection{Compress Off-Diagonal Blocks}\label{ss:form-u}

The first step constructs the sparse matrix $\tilde{A}=U^\top AV$, where
$U$ and $V$ are block diagonal with orthogonal diagonal blocks; see the
second panel of \Cref{fig:actions-on-A}. The low-rank off-diagonal
structure of $A$ becomes sparsity in these bases. We first describe this
sparsity and then explain how to compute $U$ and $V$.

\subsubsection{Sparsity of $\tilde{A}$ via Orthogonal Transformation}

Consider the block partitioning of $A$ corresponding to the leaf level of
the HSS tree $\mathcal{T}$; see the running example in
\Cref{f:hss-structure}. By the rank assumption of this section, the
off-diagonal block row $A_{ii^c}\in\mathbb{R}^{2k\times(N-2k)}$ has rank
$k$ for every leaf
node $i$. Suppose we have an orthogonal matrix $U_i\in\mathbb{R}^{2k\times2k}$
whose first $k$ columns span the column space of $A_{ii^c}$. Analogously, suppose we have an orthogonal matrix
$V_i\in\mathbb{R}^{2k\times2k}$ whose first $k$ columns span the row
space of $A_{i^c i}$. Define
\[
U \equiv \diag(U_i: \text{$i$ is a leaf node})
\quad\text{and}\quad
V \equiv \diag(V_i: \text{$i$ is a leaf node}).
\]
Throughout the paper, $\diag(M_i:i\in\mathcal S)$ denotes the block-diagonal
concatenation of the matrices indexed by $\mathcal S$.
In the new bases, \Cref{e:y=ao} becomes
\begin{equation}\label{e:compression}
\underbrace{U^\top Y}_{\displaystyle \tilde{Y}}
= \underbrace{U^\top A V}_{\displaystyle \tilde{A}}
\underbrace{V^\top \Omega}_{\displaystyle \tilde{\Omega}},
\qquad
\text{and}
\qquad
\underbrace{V^\top Z}_{\displaystyle \tilde{Z}}
= \underbrace{V^\top A^\top U}_{\displaystyle \tilde{A}^\top}
\underbrace{U^\top \Phi}_{\displaystyle \tilde{\Phi}},
\end{equation}
where, in the new bases, every off-diagonal block $A_{ij}$ between two distinct leaf nodes $i \not= j$ becomes
\begin{equation} \label{eq:sparse}
U_i^\top A_{ij} V_j
=
\begin{array}{c@{\;}c}
\left[\begin{array}{cc} B_{ij} & 0 \\[4pt] 0 & 0 \end{array}\right]
&
\begin{array}{@{}l@{}} \scriptstyle k \\[10pt] \scriptstyle k \end{array}
\\[4pt]
\begin{array}{@{}cc@{}} \scriptstyle k\hspace{9pt} & \scriptstyle \hspace{9pt}k \end{array}
&
\end{array}
\qquad (i \not= j),
\end{equation}
where we call the first $k$ row and column indices the \emph{skeleton} ($s$) indices and the
remaining $k$ the \emph{redundant} ($r$) indices. Here $B_{ij}$ is the
$k \times k$ skeleton--skeleton block.
Therefore, the matrix $\tilde{A}$ is much sparser than the original $A$,
as the second panel of \Cref{fig:actions-on-A} illustrates.

The transformation in \Cref{eq:sparse} is conceptual because the blocks
of $A$ are not available explicitly. Instead, we compute the four sample
matrices $\tilde{Y}$, $\tilde{\Omega}$, $\tilde{Z}$, and $\tilde{\Phi}$.
In particular, we partition the four tall-and-skinny matrices $Y$, $\Omega$, $Z$, and $\Phi$ (and their tilde counterparts) into block rows according
to the leaf level of the HSS tree $\mathcal{T}$ and define
\[
Y_i \equiv Y(\mathcal{I}_i,:), \quad
\Omega_i \equiv \Omega(\mathcal{I}_i,:), \quad
Z_i \equiv Z(\mathcal{I}_i,:), \quad
\Phi_i \equiv \Phi(\mathcal{I}_i,:),
\]
for every leaf node $i$. The first equality in
\Cref{f:diag-off_diag-A} illustrates this partition. We then compute the
following $2k\times3k$ blocks:
\begin{equation} \label{eq:yz_tilde}
\boxed{\,
\tilde{Y}_i = U_i^\top Y_i, \quad
\tilde{\Omega}_i = V_i^\top \Omega_i, \quad
\tilde{Z}_i = V_i^\top Z_i, \quad
\tilde{\Phi}_i = U_i^\top \Phi_i.
\,}
\end{equation}
The transformed test matrices remain Gaussian at the leaf level, as
\creflem{lemma:tilde} makes precise.
\begin{lemma} \label{lemma:tilde}
The matrices $\tilde{\Omega}$ and $\tilde{\Phi}$ are standard real
Gaussian matrices. For the complex extension, the same conclusion holds
with standard circular complex Gaussian matrices.
\end{lemma}

\begin{proof}
The basis $V$ is constructed from $(Z,\Phi)$ and is therefore independent
of $\Omega$. Similarly, $U$ is independent of $\Phi$. Conditional on either
orthogonal basis, orthogonal invariance of the standard real Gaussian
distribution implies that $V^\top\Omega$ and $U^\top\Phi$ are again
standard real Gaussian matrices. For complex matrices, replace transposes
by conjugate transposes, orthogonal matrices by unitary matrices, and the
test matrices by independent standard circular complex Gaussian matrices.
Unitary invariance then gives the corresponding result.
\qed
\end{proof}

\paragraph{Running example (continued).}
For the running example (\Cref{f:hss-structure}), the compressed system
$\tilde{Y} = \tilde{A}\tilde{\Omega}$ from \Cref{e:compression} takes the form
\Cref{e:uty-example},
where the diagonal block of $\tilde{A}$ at node $i$ is defined as
$X_i \equiv U_i^\top A_{ii} V_i$. Superscripts select the corresponding
sub-blocks. For example, $X_i^{rs}$ is the redundant-row, skeleton-column
part of $X_i$, and $\tilde{Y}_i$ and $\tilde{\Omega}_i$ split into $s$ and
$r$ rows in the same way.
The transposed system
$\tilde{Z}=\tilde{A}^\top\tilde{\Phi}$ has the same sparsity structure, with
$\tilde{Z}$ and $\tilde{\Phi}$ partitioned analogously.
\begin{equation} \label{e:uty-example}
\left[\begin{array}{c}
\tilde{Y}_4^{s} \\
\tilde{Y}_4^{r} \\
\hline
\tilde{Y}_5^{s} \\
\tilde{Y}_5^{r} \\
\hline
\tilde{Y}_6^{s} \\
\tilde{Y}_6^{r} \\
\hline
\tilde{Y}_7^{s} \\
\tilde{Y}_7^{r} \\
\end{array}\right]
=
\left[
\begin{array}{cc|cc|cc|cc}
X_4^{ss} & X_4^{sr} & B_{45} & 0 & B_{46} & 0 & B_{47} & 0 \\
X_4^{rs} & X_4^{rr} & 0 & 0 & 0 & 0 & 0 & 0 \\
\hline
B_{54} & 0 & X_5^{ss} & X_5^{sr} & B_{56} & 0 & B_{57} & 0 \\
0 & 0 & X_5^{rs} & X_5^{rr} & 0 & 0 & 0 & 0 \\
\hline
B_{64} & 0 & B_{65} & 0 & X_6^{ss} & X_6^{sr} & B_{67} & 0 \\
0 & 0 & 0 & 0 & X_6^{rs} & X_6^{rr} & 0 & 0 \\
\hline
B_{74} & 0 & B_{75} & 0 & B_{76} & 0 & X_7^{ss} & X_7^{sr} \\
0 & 0 & 0 & 0 & 0 & 0 & X_7^{rs} & X_7^{rr}
\end{array}
\right]
\left[\begin{array}{c}
\tilde{\Omega}_4^{s} \\
\tilde{\Omega}_4^{r} \\
\hline
\tilde{\Omega}_5^{s} \\
\tilde{\Omega}_5^{r} \\
\hline
\tilde{\Omega}_6^{s} \\
\tilde{\Omega}_6^{r} \\
\hline
\tilde{\Omega}_7^{s} \\
\tilde{\Omega}_7^{r} \\
\end{array}
\right]
.
\end{equation}
The explicit zeros show the sparsity created by the basis change. The
matrix $\tilde{A}$ is displayed only for illustration and is not formed by
the algorithm.

\subsubsection{Computation of $U_i$ and $V_i$}

Block nullification~\cite{levitt2024linear} supplies the off-diagonal
samples from which we compute $U_i$ and $V_i$ for the known rank $k$;
\Cref{s:adaptive-rank} determines $k$ adaptively instead. We now show how to
compute $U_i$
from $\Omega_i$ and $Y_i$. The construction of $V_i$ from $\Phi_i$ and $Z_i$
is analogous.

\vspace{1em}
\noindent\textbf{Step 1: construct nullspace basis $Q_i$.}
Consider the leaf-level block partitioning of $A$ and the decomposition
$A=\diag(A)+\offdiag(A)$, where $\diag(A)$ collects the diagonal blocks
$A_{ii}\equiv A(\mathcal I_i,\mathcal I_i)$ and
$\offdiag(A)\equiv A-\diag(A)$. The block row of
$Y=A\Omega=\diag(A)\Omega+\offdiag(A)\Omega$ restricted to leaf node $i$
then reads

    \begin{equation} \label{e:leaf-decomp}
    \begin{tikzpicture}[baseline=(Yi.base)]
      \node[anchor=base] (Yi)   at (0.00,0) {$Y_i$};
      \node[anchor=base] (eq)   at (0.80,0) {$=$};
      \node[anchor=base] (Aii)  at (1.50,0) {$A_{ii}$};
      \node[anchor=base] (mul1) at (2.20,0) {$\times$};
      \node[anchor=base] (Omi)  at (2.90,0) {$\Omega_i$};
      \node[anchor=base] (plus) at (3.70,0) {$+$};
      \node[anchor=base] (Aic)  at (4.70,0) {$A_{ii^c}$};
      \node[anchor=base] (mul2) at (5.90,0) {$\times$};
      \node[anchor=base] (Omic) at (6.60,0) {$\Omega_{i^c}.$};

      \draw[thick] (-0.30,-0.35) rectangle (0.30,-0.80);   %
      \draw[thick] ( 1.28,-0.35) rectangle (1.72,-0.80);   %
      \draw[thick] ( 2.60,-0.35) rectangle (3.20,-0.80);   %
      \draw[thick] ( 4.10,-0.35) rectangle (5.30,-0.80);   %
      \draw[thick] ( 6.30,-0.35) rectangle (6.90,-1.55);   %

      \node[font=\scriptsize, anchor=north] at (0.00,-0.86) {$2k \times 3k$};
      \node[font=\scriptsize, anchor=north] at (1.50,-0.86) {$2k \times 2k$};
      \node[font=\scriptsize, anchor=north] at (2.90,-0.86) {$2k \times 3k$};
      \node[font=\scriptsize, anchor=north] at (4.70,-0.86) {$2k \times (N-2k)$};
      \node[font=\scriptsize, anchor=west]  at (7.05,-0.95) {$(N-2k) \times 3k$};
    \end{tikzpicture}
    \end{equation}

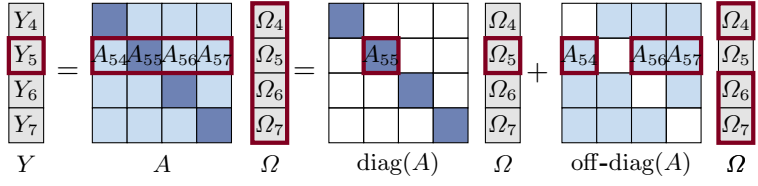
\begin{figure}
\centering
\begin{tikzpicture}[
    scale=0.46,
    cellline/.style={draw=black, thin},
    hl/.style={draw=orangecolor, ultra thick},
    oplabel/.style={font=\fontsize{11}{12}\selectfont},
    blabel/.style={font=\fontsize{9}{10}\selectfont\bfseries},
    slabel/.style={font=\fontsize{9}{10}\selectfont},
    annot/.style={font=\fontsize{9}{10}\selectfont, color=annotcolor}
]
    \definecolor{diagcolor}{RGB}{110, 130, 180}
    \definecolor{offdiagcolor}{RGB}{200, 220, 240}
    \definecolor{orangecolor}{RGB}{128, 0, 32} %
    \definecolor{datacolor}{RGB}{228, 228, 228}
    \definecolor{annotcolor}{RGB}{60, 120, 220}

    \begin{scope}[shift={(0,0)}]
        \foreach \r in {0,1,2,3} {\filldraw[cellline, fill=datacolor] (0,3-\r) rectangle (1,4-\r);}
        \draw[hl] (0,2) rectangle (1,3);
        \node[slabel] at (0.5,3.5) {$Y_4$};
        \node[slabel] at (0.5,2.5) {$Y_5$};
        \node[slabel] at (0.5,1.5) {$Y_6$};
        \node[slabel] at (0.5,0.5) {$Y_7$};
        \node[blabel] at (0.5,-0.6) {$Y$};
    \end{scope}
    \node[oplabel] at (1.7,2) {$=$};
    \begin{scope}[shift={(2.4,0)}]
        \foreach \c in {0,1,2,3} {\foreach \r in {0,1,2,3} {
            \ifnum\c=\r \filldraw[cellline, fill=diagcolor] (\c,3-\r) rectangle (\c+1,4-\r);
            \else \filldraw[cellline, fill=offdiagcolor] (\c,3-\r) rectangle (\c+1,4-\r);\fi}}
        \draw[hl] (0,2) rectangle (4,3);
        \node[slabel] at (0.5,2.5) {$A_{54}$};
        \node[slabel] at (1.5,2.5) {$A_{55}$};
        \node[slabel] at (2.5,2.5) {$A_{56}$};
        \node[slabel] at (3.5,2.5) {$A_{57}$};
        \node[blabel] at (2,-0.6) {$A$};
    \end{scope}
    \begin{scope}[shift={(7.0,0)}]
        \foreach \r in {0,1,2,3} {\filldraw[cellline, fill=datacolor] (0,3-\r) rectangle (1,4-\r);}
        \node[slabel] at (0.5,3.5) {$\Omega_4$};
        \node[slabel] at (0.5,2.5) {$\Omega_5$};
        \node[slabel] at (0.5,1.5) {$\Omega_6$};
        \node[slabel] at (0.5,0.5) {$\Omega_7$};
        \draw[hl] (0,0) rectangle (1,4);
        \node[blabel] at (0.5,-0.6) {$\Omega$};
    \end{scope}

    \node[oplabel] at (8.5,2) {$=$};
    \begin{scope}[shift={(9.2,0)}]
        \foreach \c in {0,1,2,3} {\foreach \r in {0,1,2,3} {
            \ifnum\c=\r \filldraw[cellline, fill=diagcolor] (\c,3-\r) rectangle (\c+1,4-\r);
            \else \draw[cellline] (\c,3-\r) rectangle (\c+1,4-\r);\fi}}
        \draw[hl] (1,2) rectangle (2,3);
        \node[slabel] at (1.5,2.5) {$A_{55}$};
        \node[blabel] at (2,-0.6) {$\diag(A)$};
    \end{scope}
    \begin{scope}[shift={(13.7,0)}]
        \foreach \r in {0,1,2,3} {\filldraw[cellline, fill=datacolor] (0,3-\r) rectangle (1,4-\r);}
        \node[slabel] at (0.5,3.5) {$\Omega_4$};
        \node[slabel] at (0.5,2.5) {$\Omega_5$};
        \node[slabel] at (0.5,1.5) {$\Omega_6$};
        \node[slabel] at (0.5,0.5) {$\Omega_7$};
        \node[blabel] at (0.5,-0.6) {$\Omega$};
        \draw[hl] (0,2) rectangle (1,3);
    \end{scope}
    \node[oplabel] at (15.2,2) {$+$};
    \begin{scope}[shift={(15.9,0)}]
        \foreach \c in {0,1,2,3} {\foreach \r in {0,1,2,3} {
            \ifnum\c=\r \draw[cellline] (\c,3-\r) rectangle (\c+1,4-\r);
            \else \filldraw[cellline, fill=offdiagcolor] (\c,3-\r) rectangle (\c+1,4-\r);\fi}}
        \draw[hl] (0,2) rectangle (1,3);
        \draw[hl] (2,2) rectangle (4,3);
        \node[slabel] at (0.5,2.5) {$A_{54}$};
        \node[slabel] at (2.5,2.5) {$A_{56}$};
        \node[slabel] at (3.5,2.5) {$A_{57}$};
        \node[blabel] at (2,-0.6) {$\offdiag(A)$};
    \end{scope}
    \begin{scope}[shift={(20.4,0)}]
    	\foreach \r in {0,1,2,3} {\filldraw[cellline, fill=datacolor] (0,3-\r) rectangle (1,4-\r);}
        \node[slabel] at (0.5,3.5) {$\Omega_4$};
        \node[slabel] at (0.5,2.5) {$\Omega_5$};
        \node[slabel] at (0.5,1.5) {$\Omega_6$};
        \node[slabel] at (0.5,0.5) {$\Omega_7$};
        \node[blabel] at (0.5,-0.6) {$\Omega$};
        \draw[hl] (0,3) rectangle (1,4);
        \draw[hl] (0,0) rectangle (1,2);
        \node[blabel] at (0.5,-0.6) {$\Omega$};
    \end{scope}
\end{tikzpicture}
\caption{
Decomposition of $Y = A\Omega =
\diag(A)\,\Omega \, + \, \offdiag(A)\,\Omega$ for the running example in \Cref{f:hss-structure}. 
For a leaf node, e.g., node 5, we have the decomposition
 $Y_5 = A_{55}\,\Omega_5 + A_{55^c}\,\Omega_{5^c}$.
}
\label{f:diag-off_diag-A}
\end{figure}

\Cref{f:diag-off_diag-A} shows this decomposition for the running
example. In \Cref{e:leaf-decomp},
$\Omega_i \in \mathbb{R}^{2k \times 3k}$ has full row rank $2k$ with
probability one because its entries are independent Gaussian variables
with a continuous distribution. Consequently,
$\dim\ker(\Omega_i)=k$ almost surely, and there is a matrix
$Q_i \in \mathbb{R}^{3k \times k}$ whose columns form an orthonormal
basis for this nullspace. We write
\[
	Q_i = \texttt{null}(\Omega_i, k),
\]
which can be computed by taking the last $k$ columns of the $Q$ factor from a full QR decomposition of the matrix $\Omega_i^\top$.

\vspace{1em}
\noindent\textbf{Step 2: construct low-rank basis.}
From the previous step, we have $\Omega_iQ_i = 0 \in \mathbb{R}^{2k \times k}$. As a result, right-multiplying \Cref{e:leaf-decomp} by $Q_i$ removes the diagonal contribution:
   \[
    Y_i Q_i =
    A_{ii}\,(\Omega_i Q_i)
    + A_{ii^c}\,(\Omega_{i^c} Q_i) = A_{ii^c}\,(\Omega_{i^c} Q_i).
    \]
    Because $Q_i$ depends only on $\Omega_i$, it is independent of
    $\Omega_{i^c}$. Conditional on $Q_i$, the product
    $\Omega_{i^c}Q_i$ is therefore standard Gaussian, because right
    multiplication by a matrix with orthonormal columns preserves the
    distribution. Thus, $Y_i Q_i \in \mathbb{R}^{2k \times k}$ is a
    randomized sketch of $A_{ii^c}$. Since $A_{ii^c}$ has rank $k$ by
    the rank assumption of this section, this sketch captures its column
    space with probability one. See the fixed-rank range-finder analysis
    in~\cite[Section~10]{halko2011lrobust}. Following the construction in
    \Cref{ss:rand-range-finder}, we compute a full QR factorization
    \begin{equation} \label{eq:basis}
      [U_i, \sim] = \texttt{qr}(Y_i Q_i)
    \end{equation}
    and obtain an orthogonal matrix $U_i \in \mathbb{R}^{2k \times 2k}$ whose
    first $k$ columns span the column space of $A_{ii^c}$ with probability one.

We summarize this two-step construction in \Cref{alg:u-function}.

\begin{algorithm}
\caption{\textsc{ComputeBasis}}
\label{alg:u-function}
\begin{algorithmic}[1]
\Require Matrices $\Omega_i,Y_i \in \mathbb{R}^{m \times p}$ and integer (rank) $0 < k \le p - m$
\Statex    \hfill
    \Comment{\gray{$\triangleright$~$\Omega_i$ has at least $k$ more columns than rows.}}
\Ensure Orthogonal matrix $U_i \in \mathbb{R}^{m \times m}$
\medskip
\Function{ComputeBasis}{$\Omega_i,Y_i,k$}
    \State ${Q}_i \gets \texttt{null}(\Omega_i, k)$ \label{line:compute-nullspace}
    \hfill
    \Comment{\gray{$\triangleright$~columns of $Q_i \in \mathbb{R}^{p \times k}$ belong to the nullspace of $\Omega_i$.}}

    \State $[U_i, \sim] \gets \texttt{qr}(Y_i \, {Q}_i)$ \label{line:compute-basis}
    \hfill
    \Comment{\gray{$\triangleright$~(fixed-rank) randomized range finder in \Cref{ss:rand-range-finder}.}}
    \State \Return $U_i$
\EndFunction
\end{algorithmic}
\end{algorithm}

\begin{lemma}\label{lemma:compute-basis-cost}
With the full-QR implementation described in this subsection,
\textnormal{\textsc{ComputeBasis}} applied to
$\Omega_i,Y_i\in\mathbb R^{m\times p}$ and a requested nullity $k$
requires
\[
  \mathcal O(p^2m+mpk+m^2k)\subseteq\mathcal O(p^3)
\]
flops. In particular, if $p=m+k$ and $m=\mathcal O(k)$, its cost is
$\mathcal O(k^3)$.
\end{lemma}

\begin{proof}
Forming the requested columns of the full $Q$ factor of
$\Omega_i^\top\in\mathbb R^{p\times m}$ costs
$\mathcal O(p^2m)$ flops~\cite[Section~5.2]{golub2013matrix}. Forming $Y_iQ_i$ costs
$\mathcal O(mpk)$, and constructing the full orthogonal factor $U_i$ from
this $m\times k$ matrix costs $\mathcal O(m^2k)$. Since $m,k\leq p$, the
sum is $\mathcal O(p^3)$.
\qed
\end{proof}

\begin{remark}
If $A$ is only approximately rank-$k$ HSS and each off-diagonal block
row $A_{ii^c}$ has rapidly decaying singular values, the first $k$
columns of $U_i$ in \Cref{eq:basis} provide a randomized approximation
to its dominant left singular subspace. Standard probabilistic error
bounds apply, and oversampling can improve the
approximation~\cite{halko2011lrobust}.
\end{remark}
\subsection{Elimination via Redundant Diagonal Blocks}\label{ss:x_rr_rs_sr}

We construct $\hat{A} = R \tilde{A} C$ with two block
diagonal matrices $R$ and $C$ that eliminate all skeleton--redundant couplings (see the third panel in \Cref{fig:actions-on-A}). As before, we first explain the structure of $\hat{A}$ and then discuss how to compute the matrices $R$ and $C$.

\subsubsection{Sparsity of $\hat{A}$ via Gaussian Elimination} \label{ss:ge-rl}
Consider again the block partitioning of $\tilde{A}$ at the leaf level of the HSS tree $\mathcal{T}$ (see the running example in \Cref{e:uty-example}). After the unitary transformation in \Cref{ss:form-u}, the redundant indices of every leaf node are coupled only to indices of the same node and are decoupled from all other leaf nodes. This partial decoupling motivates us to further decouple the redundant and skeleton indices within the same leaf node. Recall the definition $
X_i
=
\begin{bmatrix}
X_{i}^{ss} & X_{i}^{sr} \\
X_{i}^{rs} & X_{i}^{rr}
\end{bmatrix}
$
and that we have $k$ redundant and $k$ skeleton indices for every leaf node. By the standing local-pivot assumption, $X_{i}^{rr} \in \mathbb{R}^{k \times k}$ is nonsingular, so we define the following $2k \times 2k$ matrices
\begin{equation} \label{eq:rc}
R_{i} =
\left[\begin{array}{cc}
I_k & -X_i^{sr} \left( X_{i}^{rr} \right)^{-1}  \\
0 & I_k
\end{array}\right]
\quad
\text{and}
\quad
C_{i} =
\left[\begin{array}{cc}
I_k & 0  \\
-(X_{i}^{rr})^{-1}X_i^{rs} & I_k
\end{array}\right]
,
\end{equation}
for every leaf node $i$, where $I_k$ denotes the $k \times k$ identity matrix.
Concatenate $R_i$ and $C_i$ along the diagonal to form
\[
R \equiv \diag(R_i: \text{$i$ is a leaf node})
\quad\text{and}\quad
C \equiv \diag(C_i: \text{$i$ is a leaf node}).
\]
With these two matrices, \Cref{e:compression} is equivalent to
\begin{equation} \label{e:elimination}
\underbrace{R \tilde{Y}}_{\displaystyle \hat{Y}}
= \underbrace{R \tilde{A} C}_{\displaystyle \hat{A}}
\underbrace{C^{-1} \tilde{\Omega}}_{\displaystyle \hat{\Omega}},
\qquad
\text{and}
\qquad
\underbrace{C^\top \tilde{Z}}_{\displaystyle \hat{Z}}
= \underbrace{C^\top \tilde{A}^\top R^\top}_{\displaystyle \hat{A}^\top}
\underbrace{R^{-\top} \tilde{\Phi}}_{\displaystyle \hat{\Phi}},
\end{equation}
where every diagonal block $X_i$, after elimination, becomes
\begin{equation} \label{eq:diag}
R_i
\begin{bmatrix}
X_{i}^{ss} & X_{i}^{sr} \\
X_{i}^{rs} & X_{i}^{rr}
\end{bmatrix}
C_i
=
\begin{bmatrix}
S_{i} & 0 \\
0 & X_{i}^{rr}
\end{bmatrix}
,
\end{equation}
Here $S_i \equiv X_i^{ss} - X_i^{sr}(X_i^{rr})^{-1}X_i^{rs}$ is the Schur complement.
The off-diagonal blocks of $\tilde{A}$ are not modified. Thus, $\hat{A}$ has
no coupling between skeleton and redundant indices, as shown in the
third panel of \Cref{fig:actions-on-A}.

The transformation in \Cref{eq:diag} is likewise conceptual: neither
$A$ nor $\tilde A$ is formed explicitly. Instead, we compute
$\hat{Y}$, $\hat{\Omega}$, $\hat{Z}$, and $\hat{\Phi}$. We partition these
tall-and-skinny matrices into block rows, split each block into $s$ and
$r$ indices according to the leaf level of $\mathcal T$, and compute the
following $2k\times3k$ blocks:
\begin{equation} \label{eq:yz_hat}
\fbox{$\displaystyle
\begin{aligned}
& \hat{Y}_i = R_i \tilde{Y}_i
=
\begin{bmatrix}
\tilde{Y}_i^s - X_i^{sr} \left( X_{i}^{rr} \right)^{-1} \tilde{Y}_i^r \\
\tilde{Y}_i^r
\end{bmatrix}
,
\quad
\hat{\Omega}_i = C_i^{-1} \tilde{\Omega}_i
=
\begin{bmatrix}
\tilde{\Omega}_i^s \\
\star
\end{bmatrix}
, \\
& \hat{Z}_i = C_i^\top \tilde{Z}_i
=
\begin{bmatrix}
\tilde{Z}_i^s - \left( X_i^{rs} \right)^\top
\left( X_{i}^{rr} \right)^{-\top} \tilde{Z}_i^r \\
\tilde{Z}_i^r
\end{bmatrix}
,
\quad
\hat{\Phi}_i = R_i^{-\top} \tilde{\Phi}_i
=
\begin{bmatrix}
\tilde{\Phi}_i^s \\
\star
\end{bmatrix}
,
\end{aligned}$}
\end{equation}
where the superscript $s$ selects the first $k$ (skeleton) rows and
$r$ selects the remaining $k$ (redundant) rows of each block. The symbol $\star$
denotes quantities that are not needed and therefore not computed.
The elimination leaves the skeleton rows of the transformed test matrices
unchanged.
\begin{lemma} \label{lemma:hat}
The matrices $\hat{\Omega}_i^s = \tilde{\Omega}_i^s$ and
$\hat{\Phi}_i^s = \tilde{\Phi}_i^s$ are Gaussian matrices.
\end{lemma}

\begin{proof}
The equalities follow directly from the triangular forms of $C_i^{-1}$
and $R_i^{-\top}$ in \Cref{eq:rc}. That they are Gaussian then follows from
\creflem{lemma:tilde}.
\qed
\end{proof}

\paragraph{Running example (continued).}

For the running example (\Cref{f:hss-structure}), the system
$\hat{Y} = \hat{A}\hat{\Omega}$ from \Cref{e:elimination} takes the form
in \Cref{e:elimination-example}.
The transposed system
$\hat{Z}=\hat{A}^\top\hat{\Phi}$ has the same structure, with
$\hat{Z}$ and $\hat{\Phi}$ partitioned in the same way.
\begin{equation} \label{e:elimination-example}
\left[\begin{array}{c}
\hat{Y}_4^{s} \\
\hat{Y}_4^{r} \\
\hline
\hat{Y}_5^{s} \\
\hat{Y}_5^{r} \\
\hline
\hat{Y}_6^{s} \\
\hat{Y}_6^{r} \\
\hline
\hat{Y}_7^{s} \\
\hat{Y}_7^{r} \\
\end{array}\right]
=
\left[
\begin{array}{cc|cc|cc|cc}
S_4 & 0 & B_{45} & 0 & B_{46} & 0 & B_{47} & 0 \\
0 & X_4^{rr} & 0 & 0 & 0 & 0 & 0 & 0 \\
\hline
B_{54} & 0 & S_5 & 0 & B_{56} & 0 & B_{57} & 0 \\
0 & 0 & 0 & X_5^{rr} & 0 & 0 & 0 & 0 \\
\hline
B_{64} & 0 & B_{65} & 0 & S_6 & 0 & B_{67} & 0 \\
0 & 0 & 0 & 0 & 0 & X_6^{rr} & 0 & 0 \\
\hline
B_{74} & 0 & B_{75} & 0 & B_{76} & 0 & S_7 & 0 \\
0 & 0 & 0 & 0 & 0 & 0 & 0 & X_7^{rr}
\end{array}
\right]
\left[\begin{array}{c}
\hat{\Omega}_4^{s} \\
\hat{\Omega}_4^{r} \\
\hline
\hat{\Omega}_5^{s} \\
\hat{\Omega}_5^{r} \\
\hline
\hat{\Omega}_6^{s} \\
\hat{\Omega}_6^{r} \\
\hline
\hat{\Omega}_7^{s} \\
\hat{\Omega}_7^{r} \\
\end{array}
\right]
.
\end{equation}
The explicit zeros show that the elimination decouples every redundant block.
We emphasize that $\hat{A}$ is shown only for illustration and is not
computed by the algorithm.

\subsubsection{Computation of $R_i$ and $C_i$}

The definitions of $R_i$ and $C_i$ in \Cref{eq:rc} require
the blocks $X_i^{rr}$, $X_i^{rs}$, and
$X_i^{sr}$, which are not available explicitly. We now explain
how to estimate the two $k \times k$ blocks $X_i^{rr}$ and $X_i^{rs}$
from the two $2k \times 3k$ blocks $\tilde{Y}_i$ and $\tilde{\Omega}_i$. The block $X_i^{sr}$
is obtained in the same way from $\tilde{Z}_i$ and $\tilde{\Phi}_i$.

Restricting \Cref{e:compression} to the redundant indices of leaf node $i$ gives
\begin{equation} \label{eq:ls}
\tilde{Y}_i^{r}
=
X_i^r
\tilde{\Omega}_i
,
\end{equation}
where
$X_i^r = \left[
\begin{array}{cc}
X_i^{rs} & X_i^{rr}
\end{array}
\right]
\in \mathbb{R}^{k \times 2k}
$
is the horizontal concatenation of the two blocks. For a matrix that is
only approximately HSS, \Cref{eq:ls} need not hold exactly. We therefore
recover $X_i^r$ in the least-squares sense,
\begin{equation} \label{eq:ls2}
\min_{X_i^{r}}
\left\|
\tilde{Y}_i^{r} -
X_i^r \tilde{\Omega}_i
\right\|_{\mathrm{F}}^2
,
\end{equation}
which we solve by QR factorization~\cite[Section~5.3]{golub2013matrix}.
No regularization is applied. The system is overdetermined, and its
conditioning in the tested problems tracks the ratio of equations to
unknowns. \Cref{ss:ls-conditioning} summarizes that ratio and the
level-median $\kappa(\tilde\Omega_i)$ for four representative problems.
The reported level medians are below $21$ in 2D, and the largest observed
value over nodes is $39$, at a nearly square 1D leaf. These measurements
support the use of unregularized recovery in these experiments.

When $A$ is real symmetric, we may take $V_i\equiv U_i$, so that
$X_i=U_i^\top A_{ii}U_i$ is symmetric. If $A$ is also positive definite,
then $X_i$ is SPD and the diagonal block $X_i^{rr}$, a principal submatrix,
is SPD as well. For the complex-symmetric specialization described above,
$V_i\equiv\overline{U_i}$ and
$X_i=U_i^*A_{ii}\overline{U_i}$ is complex symmetric.
However, the solution of the least-squares problem \Cref{eq:ls2}
is not necessarily symmetric or positive definite.
We enforce symmetry by the replacement
$X_i^{rr}\gets\left(X_i^{rr}+(X_i^{rr})^\top\right)/2$.
In our numerical experiments, positive definiteness is usually preserved.
Otherwise, we add a shift $\sigma I$, with $\sigma$ initialized to
$10^{-12}\|X_i^{rr}\|_{\mathrm F}$ and doubled until the Cholesky
factorization succeeds.
Given the updated $X_i^{rr}$, we can also improve the accuracy of the block $X_i^{rs}$ by solving another least-squares problem:
\begin{equation}
\min_{X_i^{rs}} \left\|
\left(
\tilde{Y}_i^{r} - X_i^{rr}
\tilde{\Omega}_i^{r}
\right)
- X_i^{rs} \tilde{\Omega}_i^{s} \right\|_{\mathrm{F}}^2
.
\end{equation}

\subsection{Schur Complement Structure}\label{ss:p}

We now permute $\hat A$ as in the last panel of
\Cref{fig:actions-on-A}. After the elimination in \Cref{ss:x_rr_rs_sr}, the redundant
indices of every node are decoupled from the rest of the matrix. The
skeleton--skeleton couplings form a smaller Schur complement problem
with the same HSS structure. We can therefore apply
\Cref{ss:form-u,ss:x_rr_rs_sr} recursively.

For every leaf node $i$, let $P_i^s$ and $P_i^r$ select the first $k$
skeleton coordinates and the remaining $k$ redundant coordinates,
respectively. Thus,
$
P^s_i
=
\begin{bmatrix}
I_k & 0_k
\end{bmatrix}
$
and
$
P^r_i
=
\begin{bmatrix}
0_k & I_k
\end{bmatrix}
$
where $I_k$ and $0_k$ are the $k\times k$ identity and zero matrices.
Define the permutation matrix
$P
\equiv
\begin{bmatrix}
P_r \\
P_s
\end{bmatrix} $
with
\[
P_r
\equiv
\diag(P^r_i: \text{$i$ is a leaf node})
\quad
\text{and}
\quad
P_s
\equiv
\diag(P^s_i: \text{$i$ is a leaf node})
,
\]

Permuting rows and columns accordingly, \Cref{e:elimination} becomes
\begin{equation} \label{e:pruty}
P \hat{Y} = (P\hat{A}P^\top)(P\hat{\Omega}),
\qquad
\text{and}
\qquad
P\hat{Z} = (P\hat{A}^\top P^\top)(P\hat{\Phi}),
\end{equation}
where the matrix $P \hat{A} P^\top$ places all
redundant indices before the skeleton indices.
Since $\hat{A}$ has no skeleton--redundant coupling, the two off-diagonal blocks of $P \hat{A} P^\top$ are zero:
\begin{equation} \label{e:pap}
P \hat{A} P^\top
=
\begin{bmatrix}
P_r \hat{A} P_r^\top &  \\
 & P_s \hat{A} P_s^\top
\end{bmatrix}
,
\end{equation}
where the first diagonal block $P_r \hat{A} P_r^\top = \diag(X_i^{rr}: \text{$i$ is a leaf node})$, and the second diagonal block is the Schur complement,
which satisfies the following equations:
\begin{equation} \label{eq:coarse}
\underbrace{P_s \hat{Y}}_{\displaystyle {Y}^{(L-1)}}
=
\underbrace{P_s \hat{A} P_s^\top}_{\displaystyle A^{(L-1)}}
\underbrace{P_s \hat{\Omega}}_{\displaystyle {\Omega}^{(L-1)}}
,
\qquad
\text{and}
\qquad
\underbrace{P_s \hat{Z}}_{\displaystyle {Z}^{(L-1)}}
= \underbrace{P_s \hat{A}^\top P_s^\top}_{\displaystyle \left( A^{(L-1)} \right)^\top}
\underbrace{P_s \hat{\Phi}}_{\displaystyle {\Phi}^{(L-1)}}
.
\end{equation}

The superscript $L-1$ records that $A^{(L-1)}$ is the rank-$k$ HSS
matrix associated with the top $L-1$ levels of $\mathcal T$, that is, levels
$0$ through $L-1$ once the leaf level has been eliminated. Its
off-diagonal blocks are those in \Cref{eq:uv_other}.

As before, we never form the matrix $A^{(L-1)}$ explicitly. Instead, we
derive the inputs and outputs of the matvecs with $A^{(L-1)}$ and its transpose.
Specifically, we derive ${Y}^{(L-1)}$, ${\Omega}^{(L-1)}$,
${Z}^{(L-1)}$, and ${\Phi}^{(L-1)}$ as follows: %
Let the two children of a node $i$ at level $L-1$ be $\alpha$ and $\beta$. We compute the $2k \times 3k$ blocks
\begin{equation}
\fbox{
$\displaystyle
Y_i
=
 \begin{bmatrix} \hat{Y}_{\alpha}^s \\ \hat{Y}_{\beta}^s \end{bmatrix}
 ,
 \quad
\Omega_i
=
 \begin{bmatrix} \hat{\Omega}_{\alpha}^s \\ \hat{\Omega}_{\beta}^s \end{bmatrix}
 ,
 \quad
Z_i
=
 \begin{bmatrix} \hat{Z}_{\alpha}^s \\ \hat{Z}_{\beta}^s \end{bmatrix}
 ,
 \quad
 \Phi_i
 =
  \begin{bmatrix} \hat{\Phi}_{\alpha}^s \\ \hat{\Phi}_{\beta}^s \end{bmatrix}
  .
$}
\end{equation}
We drop the superscript here to simplify the notation, because the tree node $i$ is at level $L-1$ and the level is therefore implied.

At the first upward step, the reduced test matrices remain Gaussian.
\begin{lemma} \label{t:omega-iid}
The matrices $\Omega^{(L-1)}$ and $\Phi^{(L-1)}$ are random Gaussian matrices.
\end{lemma}

\begin{proof}
By \creflem{lemma:hat}, the skeleton blocks
$\hat\Omega_i^s=\tilde\Omega_i^s$ and
$\hat\Phi_i^s=\tilde\Phi_i^s$ are Gaussian. Forming $\Omega^{(L-1)}$ and
$\Phi^{(L-1)}$ only stacks these blocks, so both matrices are Gaussian.
\qed
\end{proof}

\paragraph{Running example (continued).}

For the running example in \Cref{f:hss-structure}, the first equation in
\Cref{e:pruty} is \Cref{e:elimination-example} with the redundant
coordinates placed first:
\begin{equation} \label{e:pruty-example}
\left[\begin{array}{c}
\hat{Y}_4^{r} \\
\hat{Y}_5^{r} \\
\hat{Y}_6^{r} \\
\hat{Y}_7^{r} \\
\hline
\hat{Y}_4^{s} \\
\hat{Y}_5^{s} \\
\hline
\hat{Y}_6^{s} \\
\hat{Y}_7^{s} \\
\end{array}
\right]
=
\left[
\begin{array}{cccc|cc|cc}
X_4^{rr} & 0 & 0 & 0 & 0 & 0 & 0 & 0 \\
0 & X_5^{rr} & 0 & 0 & 0 & 0 & 0 & 0 \\
0 & 0 & X_6^{rr} & 0 & 0 & 0 & 0 & 0 \\
0 & 0 & 0 & X_7^{rr} & 0 & 0 & 0 & 0 \\
\hline
0 & 0 & 0 & 0 & S_4 & B_{45} & B_{46} & B_{47} \\
0 & 0 & 0 & 0 & B_{54} & S_5 & B_{56} & B_{57} \\
\hline
0 & 0 & 0 & 0 & B_{64} & B_{65} & S_6 & B_{67} \\
0 & 0 & 0 & 0 & B_{74} & B_{75} & B_{76} & S_7
\end{array}
\right]
\left[\begin{array}{c}
\hat{\Omega}_4^{r} \\
\hat{\Omega}_5^{r} \\
\hat{\Omega}_6^{r} \\
\hat{\Omega}_7^{r} \\
\hline
\hat{\Omega}_4^{s} \\
\hat{\Omega}_5^{s} \\
\hline
\hat{\Omega}_6^{s} \\
\hat{\Omega}_7^{s} \\
\end{array}
\right]
,
\end{equation}
where we have
\begin{equation} \label{e:a-level-up}
A^{(1)}
=
\left[
\begin{array}{c|c}
A^{(1)}_{22} & A^{(1)}_{23}  \\
\hline
A^{(1)}_{32} & A^{(1)}_{33}
\end{array}
\right]
=
\left[
\begin{array}{cc|cc}
S_4 & B_{45} & B_{46} & B_{47} \\
B_{54} & S_5 & B_{56} & B_{57} \\
\hline
B_{64} & B_{65} & S_6 & B_{67} \\
B_{74} & B_{75} & B_{76} & S_7
\end{array}
\right],
\end{equation}
which retains the HSS structure. The $2 \times 2$ blocking
corresponds to the two parent nodes 2 and 3 at level 1 of the HSS tree in \Cref{fig:hss-tree-diagram}.

Correspondingly, the inputs and outputs of our matvecs are
\begin{equation} \label{e:y_hat2}
\renewcommand{\arraystretch}{1.3}
Y^{(1)}
=
\left[\begin{array}{c}
Y_2  \\
\hline
Y_3  \\
\end{array}\right]
=
\left[\begin{array}{c}
\hat{Y}_4^{s} \\
\hat{Y}_5^{s} \\
\hline
\hat{Y}_6^{s} \\
\hat{Y}_7^{s}
\end{array}\right]
,
\,
\,
\Omega^{(1)}
=
\left[\begin{array}{c}
\Omega_2  \\
\hline
\Omega_3  \\
\end{array}\right]
=
\left[\begin{array}{c}
\hat{\Omega}_4^{s} \\
\hat{\Omega}_5^{s} \\
\hline
\hat{\Omega}_6^{s} \\
\hat{\Omega}_7^{s}
\end{array}\right]
=
\left[\begin{array}{c}
\tilde{\Omega}_4^{s} \\
\tilde{\Omega}_5^{s} \\
\hline
\tilde{\Omega}_6^{s} \\
\tilde{\Omega}_7^{s}
\end{array}\right]
\end{equation}
where the last equality follows from \Cref{eq:yz_hat}, and
 $\tilde{\Omega}_i^{s}$ ($i=4,5,6,7$) is a random Gaussian matrix by \Cref{eq:yz_tilde} and \creflem{lemma:tilde}.

\subsection{Entire Algorithm and Complexity}\label{ss:linear-solve-matvec}

Set $A^{(L)}\equiv A$, $Y^{(L)}\equiv Y$,
$\Omega^{(L)}\equiv\Omega$, $Z^{(L)}\equiv Z$, and
$\Phi^{(L)}\equiv\Phi$ in \Cref{e:y=ao}. The constructions in
\Cref{ss:form-u,ss:x_rr_rs_sr,ss:p} recursively produce
\begin{equation}
Y^{(\ell)} = A^{(\ell)}\Omega^{(\ell)},
\qquad \text{and} \qquad
Z^{(\ell)} = (A^{(\ell)})^\top\Phi^{(\ell)},
\end{equation}
corresponding to the HSS tree level $\ell=L,L-1,\ldots,1,0$.
All four matrices in this pair of identities have $3k$ columns.
In particular, at the root level, we have
\begin{equation} \label{e:y_hat2=a_hat2*o_hat2}
Y^{(0)} = A^{(0)}\Omega^{(0)}, \qquad \text{and} \qquad
Z^{(0)} = (A^{(0)})^\top\Phi^{(0)},
\end{equation}
which involves
the small $2k \times 2k$ matrix $A^{(0)}$.
Since the root level has no off-diagonal blocks,
$X_1^{rr}\equiv A^{(0)}$, and we recover it from
$Y^{(0)}=X_1^{rr}\Omega^{(0)}$ by the least-squares procedure
of \Cref{ss:x_rr_rs_sr}. \Cref{alg:fixed-rank} summarizes the complete
construction.

\begin{algorithm}
\caption{Fixed-Rank HSS Compression and Factorization}
\label{alg:fixed-rank}
\begin{algorithmic}[1]
\Require HSS tree $\mathcal T$ of depth $L$ with leaf size $2k$, the
convention adopted in \Cref{s:fixed-rank};
sample matrices $Y,Z\in\mathbb R^{N\times3k}$ and test matrices
$\Omega,\Phi\in\mathbb R^{N\times3k}$ satisfying
$Y=A\Omega$ and $Z=A^\top\Phi$
\Ensure Factors $\{U_i,V_i,X_i^{rr},X_i^{rs},X_i^{sr}\}_{i\ne1}$
and root block $X_1^{rr}$
\Function{FixedRankHSS}{$\mathcal T,Y,Z,\Omega,\Phi,k$}
\For{level $\ell=L,L-1,\dots,0$}
    \If{$\ell=0$}
        \State Let $\alpha,\beta$ be the children of the root
        \State $\Omega_1 \gets \begin{bmatrix}\tilde\Omega_\alpha^s\\ \tilde\Omega_\beta^s\end{bmatrix}$
        \State $Y_1 \gets
        \begin{bmatrix}\tilde Y_\alpha^s-X_\alpha^{sr}(X_\alpha^{rr})^{-1}\tilde Y_\alpha^r\\ \tilde Y_\beta^s-X_\beta^{sr}(X_\beta^{rr})^{-1}\tilde Y_\beta^r\end{bmatrix}$
        \State Recover $X_1^{rr}\equiv A^{(0)}$ from
        $Y_1=X_1^{rr}\Omega_1$ by least squares
    \Else
        \For{node $i$ at level $\ell$}
            \If{$i$ is a leaf}
                \State $(Y_i,\Omega_i,Z_i,\Phi_i) \gets
                (Y,\Omega,Z,\Phi)(\mathcal I_i,:)$
            \Else
                \State Let $\alpha,\beta$ be the children of $i$
                \State $\Omega_i \gets \begin{bmatrix}\tilde\Omega_\alpha^s\\ \tilde\Omega_\beta^s\end{bmatrix}$,
                $\Phi_i \gets \begin{bmatrix}\tilde\Phi_\alpha^s\\ \tilde\Phi_\beta^s\end{bmatrix}$
                \State $Y_i \gets
                \begin{bmatrix}\tilde Y_\alpha^s-X_\alpha^{sr}(X_\alpha^{rr})^{-1}
                \tilde Y_\alpha^r\\ \tilde Y_\beta^s-X_\beta^{sr}(X_\beta^{rr})^{-1}
                \tilde Y_\beta^r\end{bmatrix}$
                \State $Z_i \gets
                \begin{bmatrix}\tilde Z_\alpha^s-(X_\alpha^{rs})^\top
                (X_\alpha^{rr})^{-\top}\tilde Z_\alpha^r\\ \tilde Z_\beta^s-(X_\beta^{rs})^\top
                (X_\beta^{rr})^{-\top}\tilde Z_\beta^r\end{bmatrix}$
            \EndIf
            \State $U_i \gets \textsc{ComputeBasis}(\Omega_i,Y_i,k)$,
            $V_i \gets \textsc{ComputeBasis}(\Phi_i,Z_i,k)$
            \State $\tilde Y_i \gets U_i^\top Y_i$,
            $\tilde\Omega_i \gets V_i^\top\Omega_i$,
            $\tilde Z_i \gets V_i^\top Z_i$,
            $\tilde\Phi_i \gets U_i^\top\Phi_i$
            \State Recover $X_i^{rr},X_i^{rs},X_i^{sr}$
            \hfill \Comment{\gray{$\triangleright$~\Cref{ss:x_rr_rs_sr}}}
        \EndFor
    \EndIf
\EndFor
\State \Return $\{U_i,V_i,X_i^{rr},X_i^{rs},X_i^{sr}\}_{i\ne1}$,
$X_1^{rr}$
\EndFunction
\end{algorithmic}
\end{algorithm}

\begin{remark}[Cross-dependence at coarser levels]\label{rem:cross-dependence}
The Gaussian statement in \creflem{t:omega-iid} applies only to the first
upward step. At coarser levels, the recursively constructed $U$ and $V$
bases introduce cross-dependence, so the transformed test matrices need
not have independent Gaussian entries. We nevertheless continue using the
transformed test matrices at the coarser levels. The same observation, and
the same practical resolution, appear
in~\cite[Remark~4.3]{levitt2024linear}.
Gaussian entries are not required for correctness. \Cref{th:correctness-fixed}
shows that the factorization is exact whenever the sketches reveal the
correct subspaces, whatever the distribution of the test matrices. The reused
test matrices also remain generic in a measurable sense. \Cref{ss:ls-conditioning}
reports that the level-median condition numbers in four representative
problems track the Marchenko--Pastur benchmark for Gaussian matrices of
the same aspect ratio. This is conditioning evidence, not a distributional
or independence guarantee.
\end{remark}

\begin{remark}[Additional test columns]\label{rem:additional-columns}
The identities behind \Cref{eq:coarse} hold column by column, so the test
columns need not all be drawn at the outset. A batch of $b$ columns drawn
after the factors of the levels $L,\dots,\ell+1$ have been stored can
still be carried to level $\ell$. It cannot be drawn for $A^{(\ell)}$
directly, because $A^{(\ell)}$ is available only through those factors.
We draw it at the leaf level instead:
$\Omega_{\mathrm{add}}\in\mathbb R^{N\times b}$ is standard Gaussian, and
$Y_{\mathrm{add}}=A\Omega_{\mathrm{add}}$ costs $b$ matvecs. The pair is then
carried up the tree through the stored transformations, one level at a
time, exactly as the original columns of $\Omega$ and $Y$. At a finer level
$\ell'=L,\ldots,\ell+1$, the columns are partitioned into block rows by
the nodes $j$ of that level. Each node applies the basis change of
\Cref{eq:yz_tilde} with its stored $U_j$ and $V_j$, and then the
elimination of \Cref{eq:yz_hat} with its stored $X_j^{sr}(X_j^{rr})^{-1}$:
\[
\begin{gathered}
  \tilde Y_j = U_j^\top Y_{j,\mathrm{add}},\qquad
  \tilde\Omega_j = V_j^\top\Omega_{j,\mathrm{add}},\\
  \hat Y_j^{s} = \tilde Y_j^{s}-X_j^{sr}(X_j^{rr})^{-1}\tilde Y_j^{r},\qquad
  \hat\Omega_j^{s} = \tilde\Omega_j^{s}.
\end{gathered}
\]
Only the skeleton rows are passed up. The parent of $j$ at level $\ell'-1$
receives the stacked skeleton rows of its children, as in \Cref{eq:coarse},
and the elimination leaves the skeleton rows of the test matrix unchanged
(\creflem{lemma:hat}). Every step uses stored factors only, so the $b$
matvecs that produced $Y_{\mathrm{add}}$ are the only matvecs involved.
Because the identities behind \Cref{eq:coarse} hold column by column, the
columns that arrive at level $\ell$ satisfy
$Y^{(\ell)}_{\mathrm{add}}=A^{(\ell)}\Omega^{(\ell)}_{\mathrm{add}}$
exactly. A batch $\Phi_{\mathrm{add}}$ with
$Z_{\mathrm{add}}=A^\top\Phi_{\mathrm{add}}$ is carried in the same way,
with the roles of $U_j$ and $V_j$ exchanged. At a node of size $m$ and
rank $k$, carrying one column through one level costs $\mathcal O(mk)$
for the reflectors and $\mathcal O((m-k)^2)$ for the solve with
$X_j^{rr}$. This is the cost of one level of the factorized solve below.

Such a batch is also free of the cross-dependence of
\crefrem{rem:cross-dependence}. The bases $U_j$ and $V_j$ were built
before it was drawn, so they act on it as fixed orthogonal matrices, and
orthogonal invariance applies as in \creflem{lemma:tilde}. The batch
therefore arrives at level $\ell$ still Gaussian and independent of
$A^{(\ell)}$. Columns that were used to build those bases lose this
property. The adaptive construction of \Cref{s:adaptive-rank} draws such
batches when the columns already present at level $\ell$ do not suffice.
\end{remark}

\paragraph{Factorization form.}\label{ss:handle-the-root}
\Cref{alg:fixed-rank} implicitly constructs a factorization of
$A$. At the leaf level, \Cref{e:pruty} gives
\[
P \hat{A} P^\top
=
P R \tilde{A} C P^\top
=
\underbrace{P R U^\top }_{\displaystyle W^{(L)}}
{A}
\underbrace{V C P^\top }_{\displaystyle S^{(L)}}
.
\]
The analogous transformations $W^{(\ell)}$ and $S^{(\ell)}$ act on
$A^{(\ell)}$ for $\ell=L-1,\ldots,1$. Let
$\bar W^{(\ell)}$ and $\bar S^{(\ell)}$ denote these transformations
embedded in the original coordinates, with identity blocks on the variables
eliminated at finer levels. Then
\begin{equation}\label{eq:recursive-a-root}
D \equiv
\diag\bigl(X_i^{rr}:i\in\mathcal T\bigr)
=
\bar W^{(1)}\cdots\bar W^{(L-1)}W^{(L)}
A
S^{(L)}\bar S^{(L-1)}\cdots\bar S^{(1)},
\end{equation}
where $X_1^{rr}\equiv A^{(0)}$ is the root block. Consequently,
\[
A^{-1}
=S^{(L)}\bar S^{(L-1)}\cdots\bar S^{(1)}
D^{-1}
\bar W^{(1)}\cdots\bar W^{(L-1)}W^{(L)}.
\]
Every factor is a composition of block-diagonal orthogonal and triangular
transformations and a permutation, so it can be applied efficiently.
For the running example,
$D=\diag(X_4^{rr},X_5^{rr},X_6^{rr},X_7^{rr},
X_2^{rr},X_3^{rr},X_1^{rr})$, ordered from the leaves to the root.

\paragraph{Factorized solve.}
Applying \Cref{eq:recursive-a-root} consists of an upward traversal, a
dense solve with the root block, and a downward traversal, as in a standard
HSS solve~\cite{chandrasekaran2006afast,xia2010fast}. Starting with
$y_i=f(\mathcal I_i)$ at the leaves, the upward traversal applies $U_i^\top$
and eliminates the redundant component via
\[
  y_i^s\gets y_i^s-X_i^{sr}(X_i^{rr})^{-1}y_i^r;
\]
an internal node receives the stacked skeleton components of its children.
After solving with the root block $X_1^{rr}$, the downward traversal splits
the parent solution among its children, recovers
$x_i^r=(X_i^{rr})^{-1}(y_i^r-X_i^{rs}x_i^s)$, and applies $V_i$. The leaf
vectors are then scattered to the original coordinates. This solve and the analogous matvec with the HSS factorization both cost $\mathcal O(Nk)$. We omit the
standard details.

The derivation of the factorization form also gives
\Cref{th:correctness-fixed}, a deterministic correctness statement. The
theorem isolates the generic-sketch condition needed at coarser levels, where
the recursively transformed test matrices are no longer independent Gaussian
matrices.
\begin{theorem}[Exactness under successful sketches]\label{th:correctness-fixed}
Let $A\in\mathbb R^{N\times N}$ be an exact rank-$k$ HSS matrix satisfying
the standing local-pivot assumption. Assume that every coefficient matrix
in the local and root least-squares recoveries has full row rank and that
these systems are solved exactly. If, at every nonroot node,
\textnormal{\textsc{ComputeBasis}} returns the exact off-diagonal column and row spaces,
then \Cref{alg:fixed-rank} constructs the exact factorization
\Cref{eq:recursive-a-root}. Consequently, the factorized solve returns the
exact solution in exact arithmetic whenever $A$ is nonsingular.
\end{theorem}

\begin{proof}
At a leaf, the basis changes in \Cref{eq:sparse} annihilate every redundant
off-diagonal row and column. The transformations in \Cref{eq:rc} then
eliminate the remaining skeleton--redundant couplings exactly, leaving the
Schur complement on the skeleton indices. The full-row-rank assumptions make
the local and root coefficient blocks uniquely recoverable by the stated
least-squares solves. The reduced matrix has the same HSS structure on the
parent level, so induction over the tree proves
\Cref{eq:recursive-a-root}. Inverting the transformations in reverse order
gives the factorized solve. At the leaf level, the required range conditions
hold with probability one for the independent Gaussian test matrices, by the
argument in \Cref{ss:form-u}. At coarser levels, they are the explicit
hypothesis of the theorem, consistent with \crefrem{rem:cross-dependence}.
\qed
\end{proof}

Finally, we summarize the operation counts.
\begin{theorem}\label{th:complexity-fixed}
Given a rank-$k$ HSS matrix $A \in \mathbb{R}^{N \times N}$, the construction of \Cref{s:fixed-rank} uses
$3k$ matvecs with each of $A$ and $A^\top$. Excluding these
matvecs, it requires $\mathcal{O}(Nk^2)$ flops for factorization,
$\mathcal{O}(Nk)$ flops for the linear solve, and $\mathcal{O}(Nk)$
storage. If one matvec with $A$ and one with $A^\top$ cost $T_A$ and $T_{A^\top}$,
respectively, the total construction cost is
$3k(T_A+T_{A^\top})+\mathcal O(Nk^2)$. If $A^\top=A$, the choice
$\Phi=\Omega$ makes the second sketch redundant, and the count falls to
$3k$ matvecs with $A$ alone, at a total cost of
$3kT_A+\mathcal O(Nk^2)$.
\end{theorem}

\begin{proof}
The number of matvecs equals the number of columns of
the test matrices $\Omega$ and $\Phi$ in \Cref{e:y=ao}. The flop and
storage analysis is standard, and we omit it. The bounds are
consistent with the complexity analyses
in~\cite[Section~4.4]{levitt2024linear}
and~\cite[Section~4.5.1]{yesypenko2026randomized}.
\qed
\end{proof}
\section{Adaptive-Rank Matrix-Free HSS Factorization} \label{s:adaptive-rank}

\Cref{s:fixed-rank} presents the framework with the HSS rank $k$ assumed
known. A known rank sets the number $p=m+k$ of Gaussian test vectors, the columns of the test matrices $\Omega$ and $\Phi$, in \textsc{ComputeBasis}
(\Cref{alg:u-function}), where $m$ is the leaf size (assumed to be $2k$ in
\Cref{s:fixed-rank}). In applications, however, the off-diagonal ranks are
rarely known \emph{a priori} and can vary across nodes and levels. The central
task of this section is therefore to determine the required value of $p$, and
hence the local HSS rank, from a prescribed tolerance $\epsilon$.

Apart from this rank selection, we retain the framework of
\Cref{s:fixed-rank}: nullify the diagonal block, compress the resulting
off-diagonal sample, recover the compressed diagonal blocks, eliminate the redundant variables,
and recurse up the tree.  Thus, the adaptive method changes how the local
bases $U_i$ and $V_i$ are constructed for every node at every level, but not the subsequent
compression and factorization.
As in \Cref{s:fixed-rank}, the analysis is presented for real matrices,
and $^\top$ denotes transpose. The complex extension described in
\creflem{lemma:tilde} replaces real Gaussian tests by circular complex
Gaussian tests and transposes by conjugate transposes.

For clarity, \Cref{ss:naive-adaptive,s:extend-nullspace-basis,%
ss:build-u-v-from-nullspace}
develop the adaptive construction at the leaf level, where the test
matrices are Gaussian and the error estimates come with a probabilistic guarantee. Each
leaf is tested independently and first obtains a provisional rank, which
may differ among leaves. \Cref{s:rank-refinement} adds a pruning step and
the level-wide rules that convert the provisional ranks into a common rank
per level. \Cref{ss:the-adaptive-rank-algorithm} extends the construction
to coarser levels and gives the complete algorithm and its complexity.

\subsection{Restarted Recomputation Baseline}\label{ss:naive-adaptive}

At a leaf node with $m$ indices, a natural adaptive strategy is to try
ranks $b,2b,\ldots$ by rerunning \textsc{ComputeBasis} with
$p=m+b,m+2b,\ldots$ until the tolerance is met. The test and sample
matrices can be built incrementally, so $J$ trials require only $m+Jb$ matvecs with each of $A$ and $A^\top$. Nevertheless, the nullspace and range bases are
recomputed at every trial. By \creflem{lemma:compute-basis-cost}, the total
cost over $J$ trials is
\begin{align*}
  \sum_{j=1}^{J}\mathcal O((m+jb)^3)
  &=\mathcal O\!\left(Jm^3+J^2m^2b+J^3mb^2+J^4b^3\right)\\
  &=\mathcal O(Jk^3)=\mathcal O(k^4/b),
  \qquad k=Jb,\quad m=\mathcal O(k),
\end{align*}
Our alternative, developed in \Cref{s:extend-nullspace-basis}, instead
retains both bases and extends the nullspace basis by one block per trial
with a Householder row update; we refer to this approach simply as the
\emph{row update}. Its total
cost is $\mathcal O(k^3)$, the same order as one fixed-rank construction
(\creflem{lemma:compute-basis-cost}); see \creflem{lemma:null-space-costs}.
Thus, the row update saves a factor of $J$ in local work while using the
same samples. Restarted recomputation
remains a useful stability reference because every trial uses a fresh QR
factorization.

The remainder of this section describes the nullspace update, the
corresponding extension of the basis computed on line~\ref{line:compute-basis} of
\Cref{alg:u-function}, and the rank-pruning step that follows it. For brevity, we describe the
construction of $U_i$ from $Y$ and $\Omega$. The construction of $V_i$
from $Z$ and $\Phi$ is analogous.

\subsection{Extending the Nullspace Basis}\label{s:extend-nullspace-basis}

This subsection extends line~\ref{line:compute-nullspace} of
\Cref{alg:u-function} to the adaptive setting. When the rank is known, that
line computes $Q_i$ from a test matrix with a prescribed number of columns.
Here, we extend both the test matrix and its nullspace basis one batch at a
time until the rank is determined.

The update itself is classical. It appends rows to a QR
factorization~\cite{daniel1976reorthogonalization,golub2013matrix,hammarling2008updating},
and active-set methods for constrained optimization maintain null-space
bases through such updates~\cite{gill1974methods}. What it provides here is
specific to block nullification. Appending a batch of test columns keeps
the current nullspace basis exactly and adds $b$ new directions
(\creflem{lemma:retained-basis-preserved}). The samples already formed
from the old basis therefore stay valid, and the new directions give a
fresh Gaussian sample of the off-diagonal block
(\Cref{t:psi-gaussian-iid}). Each rank trial costs only the new batch, and
the local work over all trials matches one fixed-rank construction. Earlier
block-nullification
methods~\cite{levitt2024linear,yesypenko2026randomized,amsel2026quasioptimal}
fix the number of test columns in advance. To our knowledge, no earlier
method makes block nullification adaptive.

We describe the construction for a representative leaf node $i$, whose
local test matrix $\Omega_i$ has $m$ rows.  Because every quantity in this subsection
belongs to this node, we omit the subscript $i$. We write ``old,'' ``add,''
and ``new'' for the retained, appended, and updated quantities; for example,
$Y_{\mathrm{new}}=A\Omega_{\mathrm{new}}=\begin{bmatrix}Y_{\mathrm{old}} & Y_{\mathrm{add}}\end{bmatrix}$.  We begin with a square
test block $\Omega_0\in\mathbb{R}^{m\times m}$.  Before a subsequent
one-batch update, let
\[
  \Omega_{\mathrm{old}} \in\mathbb{R}^{m\times p}, \qquad Q_{\mathrm{old}} \in\mathbb{R}^{p\times(p-m)}
\]
denote the retained test matrix and its nullspace basis, so that $p-m$ is the current trial rank. We draw a new block
$\Omega_{\mathrm{add}}\in\mathbb{R}^{m\times b}$ and set
\[
  \Omega_{\mathrm{new}}
  \equiv\begin{bmatrix}\Omega_{\mathrm{old}} & \Omega_{\mathrm{add}}\end{bmatrix},
  \qquad
  \Omega_{\mathrm{new}}\in\mathbb R^{m\times(p+b)}.
\]
The new trial rank is therefore $p-m+b$. We retain the old basis, padded
with zero rows, and append the new directions,
\begin{equation}\label{eq:nullspace-update-structure}
  Q_{\mathrm{old}}^{\uparrow}
  \equiv
  \begin{bmatrix}Q_{\mathrm{old}}\\ 0_{b\times(p-m)}\end{bmatrix},
  \qquad
  Q_{\mathrm{new}}
  \equiv
  \begin{bmatrix} Q_{\mathrm{old}}^{\uparrow} & Q_{\mathrm{add}}\end{bmatrix}.
\end{equation}
Here, $Q_{\mathrm{add}}$ has $b$ orthonormal columns in
$\ker(\Omega_{\mathrm{new}})$ and is orthogonal to the retained basis.
On the first update,
$\Omega_{\mathrm{old}}=\Omega_0$ and $p=m$, so $Q_{\mathrm{old}}$ has no columns.  We construct $Q_{\mathrm{add}}$ by an orthogonal update, which keeps every
transformation of the compression stage orthogonal.
The nullspace we need is available inside a QR factorization.
Write the full QR factorization of the transposed test matrix as
\begin{equation}\label{eq:omega-transpose-qr}
  \Omega_{\mathrm{old}}^\top
  =\mathcal Q\begin{bmatrix}R\\0\end{bmatrix},
  \qquad
  \mathcal Q=\begin{bmatrix}\mathcal Q_1 & \mathcal Q_2\end{bmatrix}
  \in\mathbb R^{p\times p},
  \qquad
  R\in\mathbb R^{m\times m},
\end{equation}
with $\mathcal Q_1\in\mathbb R^{p\times m}$. Then $\Omega_{\mathrm{old}}\mathcal Q_2=0$, so the trailing $p-m$ columns of
$\mathcal Q$ are exactly the nullspace basis $Q_{\mathrm{old}}$ that we need.

The $b$ new columns of $\Omega_{\mathrm{new}}$ become $b$ new rows of its
transpose $\Omega_{\mathrm{new}}^\top$. With $B\equiv\Omega_{\mathrm{add}}^\top\in
\mathbb R^{b\times m}$, \Cref{eq:omega-transpose-qr} gives
\begin{equation}\label{eq:row-update-stack}
  \Omega_{\mathrm{new}}^\top
  =\begin{bmatrix}\Omega_{\mathrm{old}}^\top\\ B\end{bmatrix}
  =\begin{bmatrix}\mathcal Q&0\\0&I_b\end{bmatrix}
   \begin{bmatrix}R\\ 0_{(p-m)\times m}\\ B\end{bmatrix}.
\end{equation}
Updating a QR factorization after rows are appended is
classical~\cite{daniel1976reorthogonalization,golub2013matrix}: $m$
Householder reflectors $G_1,\dots,G_m$ reduce the stacked matrix in
\Cref{eq:row-update-stack} to upper triangular form. Write $R^{(j)}$ and
$B^{(j)}$ for the values of $R$ and $B$ after $G_1,\dots,G_{j-1}$ have been
applied, with $R^{(1)}=R$ and $B^{(1)}=B$. For $j=1,\dots,m$, let
\[
  S_j\equiv\{j\}\cup\{p+1,\dots,p+b\},
  \qquad
  g_j\equiv\begin{bmatrix}R^{(j)}_{jj}\\ B^{(j)}(:,j)\end{bmatrix}
  \in\mathbb R^{b+1},
\]
and let $G_j\in\mathbb R^{(p+b)\times(p+b)}$ act on the coordinates in
$S_j$ as the $(b+1)\times(b+1)$ Householder reflector that maps $g_j$ to a
multiple of $e_1$, and as the identity on all other coordinates. The reflector $G_j$ annihilates column $j$ of
$B^{(j)}$ into the pivot $R^{(j)}_{jj}$, and it modifies only row $j$
of $R^{(j)}$ and the rows of $B^{(j)}$; the zero block has no row in any
$S_j$ and remains unchanged. The reflectors give
\begin{equation}\label{eq:row-update}
\begin{aligned}
  \begin{bmatrix}R\\ 0_{(p-m)\times m}\\ B\end{bmatrix}
  &=G_1\cdots G_m
  \begin{bmatrix}R_{\mathrm{new}}\\ 0_{(p-m+b)\times m}\end{bmatrix},
  \\
  \Omega_{\mathrm{new}}^\top
  &=\underbrace{\begin{bmatrix}\mathcal Q&0\\0&I_b\end{bmatrix}
   G_1\cdots G_m}_{\mathcal Q_{\mathrm{new}}}
  \begin{bmatrix}R_{\mathrm{new}}\\0\end{bmatrix},
\end{aligned}
\end{equation}
a full QR factorization of $\Omega_{\mathrm{new}}^\top$ with
$R_{\mathrm{new}}\in\mathbb R^{m\times m}$. The trailing $p-m+b$ columns
of $\mathcal Q_{\mathrm{new}}$ span $\ker(\Omega_{\mathrm{new}})$ and
form $Q_{\mathrm{new}}$.

The next lemma shows that the structure of these reflectors preserves the
retained basis.

\begin{lemma}\label{lemma:retained-basis-preserved}
For $m<j\le p$, the $j$th column of
$\mathcal Q_{\mathrm{new}}$ equals the $j$th column of $\mathcal Q$ padded
with $b$ zeros, so the update reproduces
\[
  Q_{\mathrm{new}}
  =\begin{bmatrix} Q_{\mathrm{old}}^{\uparrow} & Q_{\mathrm{add}}\end{bmatrix}
\]
with $Q_{\mathrm{old}}^{\uparrow}$ exactly as in
\Cref{eq:nullspace-update-structure}.
\end{lemma}

\begin{proof}
Fix $m<j\le p$ and let $e_j\in\mathbb R^{p+b}$ be the $j$th unit vector.
For every $i\le m$, the index $j$ lies outside $S_i$, because
$i\le m<j\le p<p+1$. Hence $G_ie_j=e_j$ for every $i$, and
$G_1\cdots G_me_j=e_j$. By \Cref{eq:row-update},
$\mathcal Q_{\mathrm{new}}e_j
=\begin{bmatrix}\mathcal Q&0\\0&I_b\end{bmatrix}e_j$, which is the $j$th
column of $\mathcal Q$ padded with $b$ zeros. For $m<j\le p$, these columns
of $\mathcal Q$ are the columns of $Q_{\mathrm{old}}$ by
\Cref{eq:omega-transpose-qr}, and padding them gives
$Q_{\mathrm{old}}^{\uparrow}$.
The remaining columns $p+1,\dots,p+b$ of $\mathcal Q_{\mathrm{new}}$ lie in
$\ker(\Omega_{\mathrm{new}})$ and are orthogonal to
$Q_{\mathrm{old}}^{\uparrow}$, because $\mathcal Q_{\mathrm{new}}$ is
orthogonal; they form $Q_{\mathrm{add}}$.
\qed
\end{proof}

The padded structure of \Cref{eq:nullspace-update-structure} thus comes
from the update itself, not from orthogonalizing against the retained
basis. The retained columns are never recomputed, reorthogonalized, or
passed to the update; \Cref{alg:nullspace-row-update} forms only
$Q_{\mathrm{add}}$.

\begin{algorithm}[H]
\caption{One-Batch Nullspace Extension by an Orthogonal Row Update}
\label{alg:nullspace-row-update}
\begin{algorithmic}[1]
\Require Triangular factor $R\in\mathbb R^{m\times m}$ and range factor
$\mathcal Q_1\in\mathbb R^{p\times m}$ of $\Omega_{\mathrm{old}}^\top$,
appended block $\Omega_{\mathrm{add}}\in\mathbb R^{m\times b}$
\Ensure Appended nullspace block $Q_{\mathrm{add}}$, updated $R$ and
$\mathcal Q_1$
\Function{ExtendNullspace}{$R,\mathcal Q_1,\Omega_{\mathrm{add}}$}
\State $B\gets\Omega_{\mathrm{add}}^\top$,\quad
       $C\gets\begin{bmatrix}\mathcal Q_1 & 0_{p\times b}\\ 0_{b\times m} & I_b\end{bmatrix}$
\For{$j=1,\dots,m$}
    \State build the Householder reflector $G_j$ from $\begin{bmatrix} R_{jj}\\ B(:,j)\end{bmatrix}$
    \Comment{\gray{$\triangleright$~acts on $S_j$}}
    \State apply $G_j$ on the left to rows $\{j\}\cup\{m+1,\dots,m+b\}$
           of $\begin{bmatrix} R\\ B\end{bmatrix}$
    \Comment{\gray{$\triangleright$~the rows $S_j$ in \Cref{eq:row-update}}}
    \State apply $G_j$ on the right to columns $\{j\}\cup\{m+1,\dots,m+b\}$
           of $C$
    \Comment{\gray{$\triangleright$~the columns $S_j$ of the left factor in \Cref{eq:row-update-stack}}}
\EndFor
\State \Return $Q_{\mathrm{add}}\gets C(:,m+1:m+b)$,\quad
       $R$,\quad $\mathcal Q_1\gets C(:,1:m)$ 
\EndFunction
\end{algorithmic}
\end{algorithm}

\begin{lemma}\label{lemma:null-space-costs}
Let $m$ be the node size, $b$ the batch size, and $J$ the number of
batches. Over all $J$ batches, the nullspace extension at one node costs
\[
\mathcal{O}(m^3+Jm^2b+J^2mb^2)
=\mathcal{O}\bigl((m+Jb)^3\bigr).
\]
\end{lemma}

\begin{proof}
The full QR factorization of the square initial block $\Omega_0^\top$ costs
$\mathcal{O}(m^3)$. At batch $t$, the test matrix has $p=m+(t-1)b$
columns, and \Cref{alg:nullspace-row-update} applies $m$ reflectors, each
acting on $b+1$ rows or columns. Per batch, the left applications to
$\begin{bmatrix} R\\ B\end{bmatrix}$ cost $\mathcal{O}(m^2b)$, and the right
applications to the $(p+b)\times(m+b)$ matrix $C$ cost
$\mathcal{O}(pmb)=\mathcal{O}(m^2b+tmb^2)$. Summing over the $J$ batches
and adding the initial factorization gives the bound.
\qed
\end{proof}

For $k=Jb$ and $m=\mathcal O(k)$, the bound is $\mathcal O(k^3)$, the same
order as one fixed-rank construction at rank $k$
(\creflem{lemma:compute-basis-cost}). Because every operation in
\Cref{alg:nullspace-row-update} is an orthogonal transformation, the computed
$\mathcal Q_{\mathrm{new}}$ is orthogonal to working precision, and the
nullspace residual satisfies
$\|\Omega_{\mathrm{new}}Q_{\mathrm{new}}\|=\mathcal O(u) \|\Omega_{\mathrm{new}}\|$, where $u$ is the unit roundoff, independently of $\kappa(\Omega_0)$.

In the complex case, $\Omega^\top$ is replaced by $\Omega^*$ throughout,
and the orthogonal transformations become unitary.

\subsection{Building the Local Basis $U_i$}\label{ss:build-u-v-from-nullspace}

The goal is a local basis $U_i$ satisfying
\begin{equation}\label{eq:error-stopping-criterion}
  \|(I-U_iU_i^\top)A_{ii^c}\|_{\mathrm{F}}
  \leq\epsilon\,\|A_{ii^c}\|_{\mathrm{F}}.
\end{equation}
This subsection extends line~\ref{line:compute-basis} of
\Cref{alg:u-function} to the adaptive setting.  When the rank is known, that line constructs the
complete orthogonal matrix $U_i$ from the single sample $Y_iQ_i$. In the
adaptive setting, each nullspace extension produces a new sample block, and we grow
the leading columns of $U_i$ from these blocks until \Cref{eq:error-stopping-criterion} is met. A full QR
factorization then completes these columns to a square orthogonal matrix. As in
\Cref{s:extend-nullspace-basis}, we consider a leaf node $i$.
This subsection describes the basic single-node procedure;
\Cref{s:rank-refinement} adds the pruning step and the level-wide rules
that coordinate the nodes of a level.

The nullspace update produces the new off-diagonal sample without forming
any block of $A$.  Partitioning the columns of $A$ into $\mathcal I_i$ and
$\mathcal I_i^c$, as in \Cref{e:leaf-decomp} and \Cref{f:diag-off_diag-A}, gives
\[
  Y_{i,\mathrm{new}}
  =A_{ii}\Omega_{i,\mathrm{new}}
   +A_{ii^c}\Omega_{i^c,\mathrm{new}},
\]
in which neither $A_{ii}$ nor $A_{ii^c}$ is formed.
Multiplication by $Q_{i,\mathrm{new}}$ annihilates the diagonal term and
gives
\begin{equation}\label{eq:psi-shows-up}
  Y_{i,\mathrm{new}}Q_{i,\mathrm{new}}
  =A_{ii^c}
  \begin{bmatrix} \Omega_{i^c,\mathrm{old}}Q_{i,\mathrm{old}} & \Psi_{i,\mathrm{add}}\end{bmatrix},
  \qquad
  \Psi_{i,\mathrm{add}}
  \equiv\Omega_{i^c,\mathrm{new}}Q_{i,\mathrm{add}}.
\end{equation}
Consequently, the appended nullspace block $Q_{i,\mathrm{add}}$ supplies the fresh sample
\[
  G_{i,\mathrm{add}}
  \equiv A_{ii^c}\Psi_{i,\mathrm{add}}
  =Y_{i,\mathrm{new}}Q_{i,\mathrm{add}}.
\]

Let $U_{i,\mathrm{old}}$ be the basis retained from the preceding
batches and form
\begin{equation}\label{eq:adaptive-residual-block}
  E_{i,\mathrm{add}}
  \equiv(I-U_{i,\mathrm{old}}U_{i,\mathrm{old}}^\top)
  G_{i,\mathrm{add}}.
\end{equation}
\paragraph{Error estimators.}
For a fresh appended block, \Cref{eq:f-norm-estimate} motivates the
sample residual
\begin{equation}
  \rho_i
  \equiv\frac{1}{\sqrt{b}}
  \|E_{i,\mathrm{add}}\|_{\mathrm{F}},
\end{equation}
which estimates $\|(I-U_{i,\mathrm{old}}U_{i,\mathrm{old}}^\top)
A_{ii^c}\|_{\mathrm{F}}$.  Likewise,
\begin{equation} \label{eq:rhs-norm-estimate}
  \|A_{ii^c}\|_{\mathrm{F}}^2
  =\frac{1}{b}\,\mathbb{E}
  \|A_{ii^c}\Psi_{i,\mathrm{add}}\|_{\mathrm{F}}^2
  =\frac{1}{b}\,\mathbb{E}
  \|G_{i,\mathrm{add}}\|_{\mathrm{F}}^2,
\end{equation}
so the threshold
\[
  \eta_i\equiv\frac{\epsilon}{\sqrt b}\,\|G_{i,\mathrm{add}}\|_{\mathrm{F}}
\]
estimates the right-hand side of \Cref{eq:error-stopping-criterion}. The
practical stopping test is
\begin{equation}\label{eq:adaptive-practical-test}
  \rho_i\leq\eta_i.
\end{equation}

If the stopping test in \Cref{eq:adaptive-practical-test} fails, we
orthonormalize the residual block and append the result:
\begin{equation} \label{eq:ui_added}
  U_{i,\mathrm{add}}
  =\texttt{orth}(E_{i,\mathrm{add}}),
  \qquad
  U_{i,\mathrm{new}}
  \equiv\begin{bmatrix} U_{i,\mathrm{old}} & U_{i,\mathrm{add}}\end{bmatrix}.
\end{equation}
The next test then uses a fresh batch. If the test passes, no columns are added and
$U_{i,\mathrm{new}}=U_{i,\mathrm{old}}$.  This convention also covers
the first batch by taking $U_{i,\mathrm{old}}$ to be empty.

The estimates $\rho_i$ and $\eta_i$ use only samples that the basis
construction already produces. Every batch estimates the same norm
$\|A_{ii^c}\|_{\mathrm{F}}$, so our implementation lowers the variance of
$\eta_i$ by averaging $\|G_{i,\mathrm{add}}\|_{\mathrm{F}}^2$ over all
completed batches of the node. For simplicity, we keep the definition of $\eta_i$ above.

\paragraph{Justification of the error estimates.}\label{s:checking-errors}

Since $A_{ii^c}$ is never formed explicitly, both sides of
\Cref{eq:error-stopping-criterion} are estimated
by \Cref{eq:f-norm-estimate}. The
estimator is unbiased for any isotropic test matrix that is independent of
the block it tests, and Gaussian test matrices also give its concentration
bounds. The basis construction already produces such a matrix: by the
following theorem, $\Psi_{i,\mathrm{add}}$ is standard Gaussian and
independent of $U_{i,\mathrm{old}}$.

\begin{theorem}\label{t:psi-gaussian-iid}
Suppose $\Omega_{i^c,\mathrm{new}}$ is a standard real Gaussian matrix
independent of $\Omega_{i,\mathrm{new}}$. Then
$\Omega_{i^c,\mathrm{new}}Q_{i,\mathrm{new}}
=\begin{bmatrix}\Omega_{i^c,\mathrm{old}}Q_{i,\mathrm{old}} &
\Psi_{i,\mathrm{add}}\end{bmatrix}$ in \Cref{eq:psi-shows-up} is standard
real Gaussian. In particular, $\Psi_{i,\mathrm{add}}$ is standard real
Gaussian and independent of the earlier sample blocks, and hence of
$U_{i,\mathrm{old}}$.
\end{theorem}

\begin{proof}
The matrix $Q_{i,\mathrm{new}}$ has orthonormal columns and is constructed solely
from $\Omega_{i,\mathrm{new}}$. It is therefore independent of
$\Omega_{i^c,\mathrm{new}}$. Right multiplication of an independent
standard Gaussian matrix by a matrix with orthonormal columns preserves
its distribution, applied here to $\Omega_{i^c,\mathrm{new}}Q_{i,\mathrm{new}}$
as a whole. The two blocks
$\Omega_{i^c,\mathrm{new}}Q_{i,\mathrm{old}}^{\uparrow}$ and
$\Psi_{i,\mathrm{add}}=\Omega_{i^c,\mathrm{new}}Q_{i,\mathrm{add}}$ of this
Gaussian matrix are therefore jointly Gaussian, and their cross-covariance
is proportional to
$(Q_{i,\mathrm{old}}^{\uparrow})^\top Q_{i,\mathrm{add}}=0$, since
$Q_{i,\mathrm{old}}^{\uparrow}$ and $Q_{i,\mathrm{add}}$ are the two
column blocks of $Q_{i,\mathrm{new}}$, which has orthonormal columns.
Uncorrelated jointly Gaussian
vectors are independent, so
$\Omega_{i^c,\mathrm{new}}Q_{i,\mathrm{old}}^{\uparrow}$ and
$\Psi_{i,\mathrm{add}}$ are independent. Since $U_{i,\mathrm{old}}$ is a
function of $A_{ii^c}\Omega_{i^c,\mathrm{old}}Q_{i,\mathrm{old}}
=A_{ii^c}\Omega_{i^c,\mathrm{new}}Q_{i,\mathrm{old}}^{\uparrow}$ alone, it
is independent of $\Psi_{i,\mathrm{add}}$.
\qed
\end{proof}

The hypothesis of \Cref{t:psi-gaussian-iid} holds at the leaf level, where
the test matrix consists of the Gaussian columns drawn at the start of
\Cref{alg:adaptive-rank} and the fresh batches appended to them.
\Cref{ss:the-adaptive-rank-algorithm} discusses the coarser levels.

\subsection{Rank Pruning and Level-Wide Rank Selection}\label{s:rank-refinement}

The block size $b$ trades efficiency against rank resolution. A large
block makes the matvecs and the local updates more efficient, because they
act on $b$ vectors at once and can use level-3 BLAS operations. It also lowers the variance of the
Frobenius-norm estimates in \Cref{eq:adaptive-practical-test}. A failed
test, however, appends $b$ directions at once, so the accepted rank can
overshoot the numerical rank by up to $b-1$. We therefore use a relatively
large block size and add a lightweight pruning step that removes the
redundant directions from the last appended block.

For the basic single-node procedure of
\Cref{ss:build-u-v-from-nullspace}, the pruning step uses a look-ahead
sample. When the stopping test passes, the current sample
$G_{i,\mathrm{add}}$ was not used to build $U_i$, and at the leaf level
\Cref{t:psi-gaussian-iid} shows that its test matrix
$\Psi_{i,\mathrm{add}}$ is Gaussian and independent of $U_i$. We call $G_{i,\mathrm{add}}$ the \emph{look-ahead sample}. With it, each direction
appended by the last growth step can be tested individually, and those
that do not contribute to the rank are discarded. We adapt the
sort-and-prune strategy of our earlier
work~\cite[Algorithm~3.2]{smith2026adaptive} to the hierarchical setting.
We first describe pruning in this look-ahead setting. The level-wide
rules below explain the modification when the freeze margin also uses a
passing sample for growth; \Cref{ss:the-adaptive-rank-algorithm} discusses
the statistical limitations at coarser levels.

Let $U_i\in\mathbb R^{m\times k}$ be the accepted basis and let
$U_{i,\mathrm{last}}\in\mathbb R^{m\times b}$ be its last $b$ columns,
appended by the last growth step. For an index set
$\mathcal R\subseteq\{1,\ldots,b\}$ of candidate directions to remove, let
$U_{i,\mathrm{keep}}$ be $U_i$ without the columns
$U_{i,\mathrm{last}}(:,\mathcal R)$. Since the columns of $U_i$ are
orthonormal, the Pythagorean identity gives
\begin{equation}\label{eq:rank-refine-block-q}
\begin{aligned}
&\tfrac{1}{b}\bigl\|(I-U_{i,\mathrm{keep}}U_{i,\mathrm{keep}}^\top)
G_{i,\mathrm{add}}\bigr\|_{\mathrm{F}}^2\\
&\qquad =\underbrace{\tfrac{1}{b}\bigl\|(I-U_iU_i^\top)
G_{i,\mathrm{add}}\bigr\|_{\mathrm{F}}^2}_{\rho_i^2}
+\tfrac{1}{b}\bigl\|U_{i,\mathrm{last}}(:,\mathcal R)^\top
G_{i,\mathrm{add}}\bigr\|_{\mathrm{F}}^2 .
\end{aligned}
\end{equation}
The left-hand side of \Cref{eq:rank-refine-block-q} is the residual
estimate of the basis pruned by removing $\mathcal R$, so pruning as many
directions as possible while keeping that estimate within the tolerance amounts to
the optimization problem
\[
  \max_{\mathcal R\subseteq\{1,\ldots,b\}}|\mathcal R|
  \quad\text{subject to}\quad
  \rho_i^2+\tfrac{1}{b}\bigl\|U_{i,\mathrm{last}}(:,\mathcal R)^\top
  G_{i,\mathrm{add}}\bigr\|_{\mathrm{F}}^2\leq\eta_i^2 .
\]
The row energies of
\[
T_i\equiv\frac{1}{\sqrt b}
U_{i,\mathrm{last}}^\top G_{i,\mathrm{add}},
\qquad
w_{i,j}^2\equiv\|T_i(j,:)\|_2^2,
\qquad j=1,\ldots,b,
\]
estimate the contribution of each direction: $\mathbb E[w_{i,j}^2]$ is the
exact energy that direction $j$ alone would remove, and $w_{i,j}^2$ is its
 Frobenius-norm estimate from the look-ahead sample. Since
$\tfrac{1}{b}\|U_{i,\mathrm{last}}(:,\mathcal R)^\top
G_{i,\mathrm{add}}\|_{\mathrm{F}}^2=\sum_{j\in\mathcal R}w_{i,j}^2$
exactly, the constraint becomes $\rho_i^2+\sum_{j\in\mathcal R}w_{i,j}^2\leq\eta_i^2$.
Sorting the $w_{i,j}$ in ascending order and taking the longest prefix
$\mathcal R$ that satisfies this constraint solves the optimization:
\begin{equation}\label{eq:rank-refine-budget}
  \rho_i^2+\sum_{j\in\mathcal R}w_{i,j}^2
  \;\leq\;\eta_i^2 .
\end{equation}
Because only the last block is examined, forming the energies costs
$\mathcal O(mb^2)$, followed by an $\mathcal O(b\log b)$ sort; the cost
depends on the block size, not on the retained rank. In our implementation,
the slack $\eta_i^2-\rho_i^2$ in \Cref{eq:rank-refine-budget} is reduced by
a factor $1/L^2$ that depends on the tree depth (\crefrem{rem:heuristics}).

With this step applied at every
level, the factor storage on the low-rank problem of \Cref{ss:block-size}
varies by less than one percent over a 32-fold range of $b$
(\Cref{tab:prune-stability}). Without it, the variation is forty percent.
The block size can therefore be chosen for speed.

\paragraph{Level-wide rules.}
The rules below coordinate the nodes of one level $\ell$, whose nodes
have size $m_\ell$; at the leaves, $\ell=L$ and $m_L=m$. The rules modify
the single-node procedure of \Cref{ss:build-u-v-from-nullspace} and the
pruning step above at every level, including the leaves. Because they
depend on how the tolerance is allocated across levels, we first describe
that allocation.

\paragraph{Error propagation across levels.}
The orthogonal transformations of the compression stage preserve Frobenius
norms, but the complete factorization also contains the triangular
elimination factors of \Cref{ss:x_rr_rs_sr}. A compression error introduced
into a Schur complement can be amplified when it is carried back through
those factors to the original coordinates, and for a nonsymmetric matrix
both the column-space and the row-space errors enter. Thus, the local
tolerance of \Cref{eq:error-stopping-criterion} is not, by itself, an \emph{a priori} bound on the error of the complete representation or on the residual
of a solve with it. A rigorous end-to-end bound would also have to account
for the conditioning of the intervening triangular factors. \Cref{s:result}
reports the resulting solve accuracy; the Supplementary Material, Section~S2,
examines its relation to local amplification factors.
The allocation of the tolerance across levels
in our implementation is therefore a conservative heuristic, described in
\crefrem{rem:heuristics}, and not a global error guarantee.

\begin{remark}[Practical heuristics]\label{rem:heuristics}
The following two tolerance choices in our implementation are heuristics
motivated by this error propagation.

\noindent(i)~\emph{Level-dependent tolerance.} Each coarse-level sample is
assembled from the already-compressed skeletons of the finer levels and
carries their accumulated error, so the coarse-level residual cannot fall
below the error already committed at the finer levels. Demanding the same tolerance at
every level is therefore neither attainable nor worthwhile. We instead
relax the tolerance toward the root and use
\[
  \epsilon_\ell=\epsilon\,\tau^{L-\ell}
\]
in place of $\epsilon$ in $\eta_i$ at level $\ell$, where $L$ is the
number of levels, the leaves are at level $L$, and $\tau\ge1$. The
relaxation also caps the ranks at the coarse levels, where the blocks are
largest and most expensive. In the experiments of \Cref{s:result}, we use
$\tau=\sqrt2$ for 1D problems and $\tau=2$ for 2D problems, reflecting the
faster rank growth in two dimensions.

\noindent(ii)~\emph{Pruning budget.} Pruning at one level raises the error entering
the next level, and the per-level stopping test does not see this
increase. We therefore allow each level to spend only a $1/L$ fraction of
the available tolerance on pruning: in \Cref{eq:rank-refine-budget}, the
slack $\eta_i^2-\rho_i^2$ is replaced by $(\eta_i^2-\rho_i^2)/L^2$, at
every level including the leaves. Here $\eta_i$ is formed with the
unrelaxed $\epsilon$, so that the relaxation in (i) caps coarse-level rank
growth without allowing additional pruning where the accumulated error is
largest.
\end{remark}

\paragraph{Level-wide stopping threshold.}
A purely relative test never passes on a node whose off-diagonal block
is numerically zero, such as a leaf on which the compactly supported
scattering potential of \Cref{app:lippmann} vanishes. There, $\eta_i$ is
$\epsilon$ times roundoff, whereas $\rho_i$ is roundoff that decreases
only as the basis absorbs the noise, one batch of $b$ directions at a
time. Such a node reaches full rank without ever passing. We therefore
give a node a second way to stop. Let
\[
  \bar\eta_\ell\equiv\frac{\epsilon}{\sqrt b}\,
  \Bigl(\frac{1}{n_\ell}\sum_{j}\|G_{j,\mathrm{add}}\|_{\mathrm{F}}^2\Bigr)^{1/2}
\]
be the absolute tolerance of level $\ell$, formed from the mean of $\|G_{j,\mathrm{add}}\|_{\mathrm{F}}^2$ over its
$n_\ell$ nodes. A node also passes the stopping test once
$\rho_i\le\bar\eta_\ell$: what remains of it is negligible at the
scale of its level. In particular, a node whose block is itself
negligible at that scale passes its first test and is stored without
compression. We therefore replace the threshold $\eta_i$ of
\Cref{ss:build-u-v-from-nullspace} and of the pruning step above by $\max(\eta_i,\bar\eta_\ell)$, both for stopping and for pruning.
Since $\sum_i\max(\eta_i,\bar\eta_\ell)^2\le2\sum_i\eta_i^2$, the bound on
the level's total error grows by at most a factor $\sqrt2$. On the
problems of \Cref{s:result}, whose blocks have comparable norms, the rule
changes the selected ranks by at most one unit, and only at isolated levels.
For stopping at level $\ell$, both thresholds use $\epsilon_\ell$ in
place of $\epsilon$; pruning uses the unrelaxed tolerance, as specified in
\crefrem{rem:heuristics}. In the following discussion, $\eta_i$ denotes
the effective threshold $\max(\eta_i,\bar\eta_\ell)$ unless the maximum
is written explicitly.

\paragraph{Freeze margin.}
The level stops when $\rho_i\le\eta_i$ for every node, but a node grows
its basis only while $\rho_i>\gamma\eta_i$, where $\gamma=1/4$. A node
whose estimate has fallen below $\gamma\eta_i$ is frozen: it is still
tested with every later batch but no longer grows. The margin matters
because $\rho_i$ is computed from $b$ Gaussian columns and has a relative
standard error of about $(2b)^{-1/2}$. With a single threshold, a node
whose true error is close to it passes or fails depending on the sample it
draws, and the simultaneous convergence of all nodes at a large level
becomes unlikely. In preliminary runs without the margin, the matvec count on the
one-dimensional Poisson problem grew with $N$ for this reason. With the
margin, it was independent of $N$.

Unlike the basic rule of \Cref{ss:build-u-v-from-nullspace}, this
level-wide rule can grow a basis even when its node passes the stopping
test. If $\rho_i\in(\gamma\eta_i,\eta_i]$, the
passing sample is also used to grow the basis, so it is no longer a
look-ahead sample for that node. In that case, the energies $w_{i,j}$ are
measured on the sample the last block was fitted to, and $\rho_i$ keeps its
value from before the growth, so the left-hand side of
\Cref{eq:rank-refine-budget} counts the energy of the removed directions
twice. This can only make the criterion more conservative: it may prune
fewer directions, but not more.

\paragraph{One rank per level.}
During the rank search, the nodes at level $\ell$ are tested
independently. Node $i$ stops growing at its provisional rank $\hat k_i$,
the width of its accepted basis, and the search ends once every node
passes. Our implementation then stores one rank per level, which keeps the
storage rectangular. Every node is first completed to the width
\[
  k_\ell^{\mathrm{pre}}\equiv\max_i\hat k_i .
\]
A node that stopped early is extended with directions found in the samples
drawn for the nodes that were still growing, excluding the look-ahead
block. If those samples do not supply enough directions, the remaining
columns are taken from the orthogonal complement of its current basis. This completion requires
no additional matvecs. Each node then determines the set $\mathcal R_i$ of
directions that \Cref{eq:rank-refine-budget} allows it to remove, and the
level removes
\[
  r_\ell\equiv\min_i|\mathcal R_i|
\]
directions from every node, each node removing the directions with its lowest energy. The stored rank is the common value
$k_\ell\equiv k_\ell^{\mathrm{pre}}-r_\ell$. If a level reaches the node
dimension $m_\ell$ with the test still failing, its nodes are stored
dense, as happens for the leaves at $\epsilon=10^{-12}$ in
\Cref{ss:poisson-d3}.

Extending a node that stopped early to the common width makes its error
smaller than $\epsilon_\ell$. This extra accuracy is useful
rather than wasteful, because the elimination of \Cref{ss:x_rr_rs_sr}
carries a child's compression error into its parent's samples. On the two-dimensional problems of
\Cref{s:result}, compressing every node only to $\epsilon_\ell$ and
storing it at its own rank lowered storage by 10--20\%, but it raised the
coarse-level ranks and increased the final error by a factor of three to
twenty.

\subsection{Complete Algorithm and Complexity}
\label{ss:the-adaptive-rank-algorithm}

The construction of
\Cref{s:extend-nullspace-basis,ss:build-u-v-from-nullspace,s:rank-refinement}
applies at every coarser level $\ell$ after replacing $m$ by the node size
$m_\ell$ and $A$ by the Schur complement $A^{(\ell)}$ of \Cref{ss:p}. What
changes is where the test columns come from.

\paragraph{Inherited columns.}
When the algorithm moves to level $\ell$, the recursion of \Cref{ss:p} has
already produced test and sample columns for $A^{(\ell)}$ from the child
skeletons, as in \Cref{eq:coarse}. These inherited columns form the initial
retained matrices $\Omega_{\mathrm{old}}$ and $Y_{\mathrm{old}}$ of every
node. If fewer than $m_\ell$ columns are inherited, we draw just enough
additional columns to complete the square initial block. If more are
inherited, the first $m_\ell$ form that block, and the rest are fed through
the update of \Cref{s:extend-nullspace-basis} one batch at a time, exactly
as freshly drawn columns would be, before any new matvecs are computed.
If the stopping test still fails, fresh columns are drawn at the leaf
level and carried to level $\ell$ as in \crefrem{rem:additional-columns}.
The root has no rank search of its own, so we draw columns until the
root holds at least $1.1$ times its dimension; the recovery of $A^{(0)}$
in \Cref{ss:handle-the-root} is then overdetermined.
When the ranks stay bounded toward the root, as in the one-dimensional
problems of \Cref{s:result}, the leaf-level columns suffice at every level
and no fresh columns are drawn. The implementation also caps the batch size
at level $\ell$ to $\min(b,\lfloor m_\ell/4\rfloor)$, so that the search at a
small coarse node keeps at least four resolvable steps (\Cref{eq:b-eff});
we write $b$ for the batch size throughout.

\paragraph{Test matrices at coarser levels.}
By \crefrem{rem:additional-columns}, a fresh batch arrives at level
$\ell$ still Gaussian and independent of $A^{(\ell)}$. The inherited
columns, however, were used to build the finer-level bases, and the cross-dependence of \crefrem{rem:cross-dependence}
applies to them. Because the appended nullspace block $Q_{i,\mathrm{add}}$
combines fresh and inherited columns, the sample $\Psi_{i,\mathrm{add}}$ in
\Cref{eq:psi-shows-up} is not, in general, independent of $A^{(\ell)}$ at a
coarser level. Relaxing Gaussianity to isotropy does not remove this
restriction, because unbiasedness also requires the test matrix to be
independent of the block it tests, and it is this independence that the
inherited columns lose. There, we use $\rho_i$ and $\eta_i$ as practical error
indicators rather than as estimators with the guarantee of
\Cref{t:psi-gaussian-iid}, and we apply the pruning criterion
\Cref{eq:rank-refine-budget} as a heuristic. \Cref{ss:ls-conditioning}
summarizes the measured level-median condition numbers, which track the
Gaussian benchmark for matrices of the same aspect ratio; the Supplementary Material, Section~S4, gives the level-by-level comparison.
This observation supports the practical use of pruning at
every level but does not establish the independence required by
\Cref{t:psi-gaussian-iid}.

\paragraph{Complete algorithm.}
\Cref{alg:adaptive-rank} gives the overall adaptive-rank method, and the
helper \Cref{alg:adaptive-basis} computes the bases at one level. To obtain
$V_i$, we call it with $Z$, $\Phi$, and $x\mapsto A^\top x$ in place of
$Y$, $\Omega$, and $x\mapsto Ax$. For simplicity of presentation, the helper is written with
$m_\ell$ initial columns, so each local $\Omega_{i,0}$ is square, and the
inherited columns are processed one batch at a time before any fresh
matvecs are computed.

\begin{algorithm}[tbp]
\caption{Level $\ell$ Adaptive Basis Computation for \Cref{alg:adaptive-rank}}
\label{alg:adaptive-basis}
\begin{algorithmic}[1]
\Require Level $\ell$; batch size $b$; stopping tolerance $\epsilon_\ell$ and
pruning tolerance $\epsilon$ (\crefrem{rem:heuristics}); freeze margin
$\gamma$; number of levels $L$;
retained $\Omega$ and $Y=A\Omega$; access to $x\mapsto Ax$
\Ensure Full orthogonal $\{U_i\}$ whose first $k_\ell$ columns are the
accepted bases; augmented $Y,\Omega$
\Function{AdaptiveBasisAtLevel}{$\ell,b,\epsilon_\ell,\epsilon,\gamma,L,
Y,\Omega,(x\mapsto Ax)$}
\State Initialize the retained nullspace and range bases of every node $i$
from the first $m_\ell$ columns of $\Omega$ and $Y$
\Repeat
    \State Take the next inherited $(\Omega_{\mathrm{add}},Y_{\mathrm{add}})$
    pair, or draw a Gaussian $\Omega_{\mathrm{add}}$, compute
    $Y_{\mathrm{add}}\gets A\Omega_{\mathrm{add}}$, and carry both to
    level $\ell$; append the pair to $\Omega$ and $Y$
    \For{node $i$ in level $\ell$}
        \State Extend $Q_i$ by \Cref{alg:nullspace-row-update} and form
        $G_{i,\mathrm{add}}=Y_{i,\mathrm{new}}Q_{i,\mathrm{add}}$
        \State Form $E_{i,\mathrm{add}}=(I-U_iU_i^\top)G_{i,\mathrm{add}}$
        and update $\rho_i$, $\eta_i$, and $\bar\eta_\ell$
        \If{$\rho_i>\gamma\max(\eta_i,\bar\eta_\ell)$}
            \State Append $\texttt{orth}(E_{i,\mathrm{add}})$ to $U_i$
        \EndIf
    \EndFor
\Until{$\rho_i\leq\max(\eta_i,\bar\eta_\ell)$ for every node $i$}
\State Complete every $U_i$ to $k_\ell^{\mathrm{pre}}=\max_i\hat k_i$
columns from the retained samples
\State At each node, sort the last-block energies and find the largest
$\mathcal R_i$ admissible under \Cref{eq:rank-refine-budget} with the
budget of \crefrem{rem:heuristics}
\State Remove $r_\ell=\min_i|\mathcal R_i|$ directions from every node, set
$k_\ell=k_\ell^{\mathrm{pre}}-r_\ell$, and complete each $U_i$ to a full
orthogonal matrix by $\texttt{qr}$
\State \Return $\{U_i\}$, $Y$, $\Omega$
\EndFunction
\end{algorithmic}
\end{algorithm}

\begin{algorithm}[tbp]
\caption{Adaptive-Rank HSS Compression and Factorization}
\label{alg:adaptive-rank}
\begin{algorithmic}[1]
\Require Sampling block size $b$, tolerance $\epsilon$, relaxation factor $\tau \ge 1$, freeze margin $0<\gamma\le 1$, leaf size $m$, number of levels for HSS tree $L$, access to matvecs $x \mapsto Ax$
and $x \mapsto A^\top x$, where $A \in \mathbb{R}^{N \times N}$
\Ensure Compressed and factorized HSS representation of $A$
\Function{AdaptiveRank}{$b,\epsilon,\tau,\gamma,m,L,(x\mapsto Ax),(x\mapsto A^\top x)$}
\State Draw independent standard Gaussian matrices
$\Omega,\Phi\in\mathbb{R}^{N\times m}$
\State $Y \gets (x \mapsto Ax)(\Omega), \quad Z \gets (x \mapsto A^\top x)(\Phi)$
\For{level $\ell = L,L-1,\dots, 1$}
    \State $\epsilon_\ell \gets \epsilon\,\tau^{L-\ell}$ \hfill \Comment{\gray{$\triangleright$~level-dependent tolerance (\crefrem{rem:heuristics})}}
    \If{$\ell < L$}
        \State Assemble $\Omega,\Phi,Y,Z$ from the child skeleton data
        \hfill \Comment{\gray{$\triangleright$~\Cref{ss:p}}}
    \EndIf
    \State $[\{U_i\},Y,\Omega]\gets\textsc{AdaptiveBasisAtLevel}(\cdots)$
    \hfill \Comment{\gray{$\triangleright$~using $x\mapsto Ax$}}
    \State $[\{V_i\},Z,\Phi]\gets\textsc{AdaptiveBasisAtLevel}(\cdots)$
    \hfill \Comment{\gray{$\triangleright$~using $x\mapsto A^\top x$}}
    \State Perform the remaining compression and elimination steps at level $\ell$
    \hfill \Comment{\gray{$\triangleright$~\Cref{ss:form-u,ss:x_rr_rs_sr}}}
\EndFor
\State Solve for $A^{(0)}$ via $Y$ and $\Omega$ from level 1 \hfill \Comment{\gray{$\triangleright$~At root: \Cref{ss:handle-the-root}}}
\State \Return the HSS factors
\EndFunction
\end{algorithmic}
\end{algorithm}

\paragraph{Complexity analysis.}
The number of batches $J_\ell$ processed at level $\ell$ satisfies
\begin{equation}\label{eq:batches-vs-rank}
  J_\ell\,b\;\le\;k_\ell^{\mathrm{pre}}+b\;\le\;k_\ell+2b .
\end{equation}
The first inequality holds because the last node to pass grew at every batch before it
passed. It can fail only if a frozen node fails a later test, which
requires its estimate to rise by a factor of $1/\gamma$ with an unchanged
basis. The second inequality holds because pruning removes at
most one block. Setting $m=m_\ell$ and $J=J_\ell$ in
\creflem{lemma:null-space-costs} bounds the nullspace extension at one
node of level $\ell$ by $\mathcal O((m_\ell+J_\ell b)^3)$. By
\Cref{eq:batches-vs-rank} and $m_\ell=\mathcal O(k_\ell)$, this is
$\mathcal O((k_\ell+b)^3)$, the same order as the fixed-rank construction
at rank $k_\ell$ (\creflem{lemma:compute-basis-cost}).

\begin{theorem}[Complexity of the adaptive factorization]\label{th:complexity-adaptive}
Let the HSS tree have $L$ levels with $N_\ell$ nodes of size $m_\ell$ at
level $\ell$, let $k_\ell$ be the stored rank at level $\ell$, and assume
$m_\ell=\mathcal O(k_\ell)$, as in \Cref{s:fixed-rank}. Let $p$ be the
number of columns of $\Omega$ at the end of the construction, which is the
number of matvecs with $A$; the analogous count applies to $\Phi$ and
$A^\top$.
Excluding the matvecs, \Cref{alg:adaptive-rank} costs
\[
  \mathcal O\Bigl(\sum_{\ell=1}^{L}N_\ell\,(k_\ell+b)^3\Bigr)
  +\mathcal O\Bigl(p\sum_{\ell=1}^{L}N_\ell\,k_\ell^2\Bigr)
\]
flops, where the first term is the work on the local bases and the second
bounds the work proportional to the number of sample columns: recovering
the compressed diagonal blocks from the $p_\ell\le p$ columns present at
level $\ell$, and carrying the columns through the tree.
The factorization requires $\mathcal O(\sum_{\ell}N_\ell k_\ell^2)$ storage,
and one solve with it costs the same. For fixed $b$, this gives the
following bounds.
\begin{enumerate}[label=(\alph*)]
\item For the one-dimensional model, the binary tree with $N_\ell=2^\ell$
and bounded ranks $k_\ell\le k$: $p=\mathcal O(k)$ matvecs with each
operator, $\mathcal O(Nk^2)$ flops, and $\mathcal O(Nk)$ storage and solve
cost. For bounded $k$, the flop count, the storage, and the solve cost are
all $\mathcal O(N)$.
\item For the two-dimensional model, in which a node with $n$ indices has
rank $\mathcal O(\sqrt n)$, so that $k_\ell=\mathcal O(\sqrt{N/2^\ell})$ on
the binary tree with $N_\ell=2^\ell$: $p=\mathcal O(\sqrt N)$ matvecs with
each operator, $\mathcal O(N^{1.5}\log N)$ flops, and
$\mathcal O(N\log N)$ storage and solve cost. The logarithmic factor in the
flop count comes only from carrying the columns through the tree; all
other work is $\mathcal O(N^{1.5})$.
\end{enumerate}
\end{theorem}

\begin{proof}
At a node of level $\ell$, the nullspace extension costs
$\mathcal O((m_\ell+J_\ell b)^3)$ by \creflem{lemma:null-space-costs}, and
completing the basis by a full QR factorization costs
$\mathcal O(m_\ell^3)$. With \Cref{eq:batches-vs-rank} and
$m_\ell=\mathcal O(k_\ell)$, this work is $\mathcal O((k_\ell+b)^3)$ per
node; summing over the nodes and levels gives the first term. Recovering
the compressed diagonal blocks solves the least-squares problems of
\Cref{eq:ls2}, with $p_\ell$ equations and $m_\ell$ unknowns, at
$\mathcal O(p_\ell m_\ell^2)=\mathcal O(p_\ell k_\ell^2)$ per node. Every
column of $\Omega$ and $Y$ is drawn at the
leaf level and carried toward the root by applying the stored $U_j^\top$, $V_j^\top$, and
elimination factors once at each level it passes. At a node, applying the
$k_\ell$ reflectors costs $\mathcal O(m_\ell k_\ell)$ and the solve with
$X_i^{rr}$ costs $\mathcal O((m_\ell-k_\ell)^2)$, together
$\mathcal O(k_\ell^2)$, as in one level of the factorized solve. Since
$p_\ell\le p$, the recovery and the carrying together cost at most the
second term. The stored factors are $k_\ell$ Householder vectors of length
$m_\ell$ for each $U_i$ and $V_i$ and the $\mathcal O(k_\ell^2)$ entries of
the blocks $X_i^{rr}$, $X_i^{rs}$, and $X_i^{sr}$ at each node, and the
solve applies each factor once~\cite{chandrasekaran2006afast,xia2010fast}.

Columns are never discarded and their number is nondecreasing toward the
root, so $p=\max_\ell(m_\ell+J_\ell b)$, the largest number of columns any
level requires, which is $\mathcal O(\max_\ell k_\ell+b)$ by
\Cref{eq:batches-vs-rank}. Under both models, the number of columns present
at level $\ell$ is $p_\ell=\mathcal O(k_\ell+b)$: in the one-dimensional
model every column is drawn at the leaves and $p=\mathcal O(k)$, and in the
two-dimensional model the count grows toward the root, so
$p_\ell=m_\ell+J_\ell b$. The recovery therefore costs
$\mathcal O((k_\ell+b)^3)$ per node, the same order as the work on the
bases, and only carrying the columns can exceed the first term. For the
one-dimensional model, $N=2^Lm_L$ with $m_L=\Theta(k)$ (the leaf size is
$2k$ in \Cref{s:fixed-rank}), so $\sum_\ell 2^\ell=\mathcal O(N/k)$ and
$p=\mathcal O(k)$; substituting gives $\mathcal O(Nk^2)$ flops and
$\mathcal O(Nk)$ storage. For the two-dimensional model,
$N_\ell k_\ell^2=\mathcal O(N)$ at every level,
$N_\ell k_\ell^3=\mathcal O(N^{1.5}/2^{\ell/2})$, and
$p=\mathcal O(k_1)=\mathcal O(\sqrt N)$. The work on the bases and the
recovery therefore sum to $\mathcal O(N^{1.5})$, the storage to
$\mathcal O(NL)=\mathcal O(N\log N)$, and carrying the columns costs
$\mathcal O(\sqrt N\cdot N\log N)$ $=\mathcal O(N^{1.5}\log N)$.
\qed
\end{proof}

The second term of the flop count is not specific to adaptivity: any
construction from $p$ global sketches carries every sketch column through
the tree, and the fixed-rank construction of \Cref{s:fixed-rank} does the
same with all $p$ columns at every level. In the two-dimensional model,
this term exceeds the dense work by a logarithmic factor, because $p$ grows
like $\sqrt N$ while every level contributes $\mathcal O(N)$ per column.
When the ranks are bounded, the leaf-level columns suffice at every level,
and the flop count reduces to that of the fixed-rank construction in
\Cref{th:complexity-fixed}. The storage bound covers the factorization.
During the construction, the algorithm also holds the sample matrices
$\Omega,Y,\Phi,Z\in\mathbb R^{N\times p}$, which is $\mathcal O(Np)$ working
memory; \Cref{ss:setup} excludes them from the reported storage.

\section{Numerical Results} \label{s:result}

We assess the adaptive factorization on four application classes: Poisson
and Helmholtz Schur complements, Mat\'{e}rn covariance matrices, and a
complex Lippmann--Schwinger equation. The experiments examine how the
detected ranks affect computational cost, how the local compression
tolerance relates to the solve residual, and how the factorization
performs as a preconditioner. The Poisson problems test bounded ranks and
rank growth; the Helmholtz, Mat\'{e}rn, and Lippmann--Schwinger problems
test indefinite, ill-conditioned, and complex non-Hermitian operators,
respectively. \Cref{app:parameters} summarizes the evidence supporting
the choices of sampling block size, pruning scope, and unregularized
least-squares recovery. The full diagnostic studies are provided in
the Supplementary Material.

\subsection{Test Problems and Protocol}\label{ss:setup}

\Cref{t:problems} lists the five test problems. Throughout this section,
$N$ is the order of the dense matrix $A$, and the dimension of a matrix
problem refers to its indexing point set. Thus, a separator Schur
complement from a 2D PDE is a 1D problem, whereas one from a 3D PDE is a
2D problem. For the performance tables, we prescribe the solution
$x_\star(t)=\sin(\pi t)$ in 1D and $x_\star(t_1,t_2)=\sin(\pi t_1)\sin(\pi t_2)$
in 2D, sampled on the normalized indexing grid, and set $f=Ax_\star$.
The HSS tree is binary in 1D and has four children per node in 2D,
obtained by bisecting the point set in each coordinate direction. In 2D,
the coarsening step of \Cref{ss:p} stacks four child skeletons; this
changes constants but not the bounds of \Cref{th:complexity-adaptive}.
For each problem, we construct $A_{\mathrm{approx}}$ and compute
$\tilde{x}=A_{\mathrm{approx}}^{-1}f$. The tables report subsets of the
following quantities:
\begin{itemize}
  \item $\epsilon$: the input tolerance controlling estimated local
  compression errors; it is not a prescribed solve residual.
  \item $n_{\mathrm{mv}}$: the total number of matvecs with $A$ and, when
  needed, $A^\top$ or $A^*$ during construction.
  \item $t_{\mathrm{fact}}$: the local compression and elimination time,
  excluding all matvec time.
  \item $t_{\mathrm{solve}}$: the time for one solve with the computed
  factorization.
  \item relres: the relative residual
  $\|f-A\tilde x\|_2/\|f\|_2$ of the direct solve, before Krylov refinement.
  \item mem$_{\mathrm{MB}}$: storage for the computed factors in MB,
  excluding sample matrices and other construction workspace.
  \item $n_{\mathrm{it}}$ and $\tilde n_{\mathrm{it}}$: the Krylov
  iteration counts with and without the factorization as a preconditioner,
  using the same matrix and right-hand side.
\end{itemize}
We use MATLAB's \texttt{pcg} for SPD matrices and unrestarted
\texttt{gmres} for the indefinite and non-Hermitian problems, with solver
tolerance $10^{-12}$ in both the preconditioned and unpreconditioned runs.
For preconditioned GMRES, this tolerance applies to the relative
preconditioned residual, which differs from the direct-solve residual
reported as relres. Stagnated runs are marked and explained in
\Cref{app:lippmann}.

\begin{table}
    \centering
    \small
    \caption{\em Test problems and their roles in the numerical study.
    The factorization accesses each matrix only through matvecs.}
    \label{t:problems}
    \begin{tabular}{@{}>{\raggedright\arraybackslash}p{0.46\textwidth}>{\raggedright\arraybackslash}p{0.46\textwidth}@{}}
        \toprule
        Problem & What it tests \\
        \midrule
        Schur complement of the 2D Poisson equation, SPD &
        bounded ranks and linear scaling \\
        \addlinespace
        Schur complement of the 3D Poisson equation, SPD &
        rank growth with matrix size and tree level \\
        \addlinespace
        Schur complement of the 2D Helmholtz equation, symmetric
        indefinite &
        oscillatory interactions and sensitivity to resonance \\
        \addlinespace
        2D Mat\'{e}rn covariance, two parameter sets, ill-conditioned SPD &
        preconditioning of ill-conditioned matrices \\
        \addlinespace
        2D Lippmann--Schwinger equation, complex symmetric &
        complex arithmetic and preconditioning of a near-resonant cavity \\
        \bottomrule
    \end{tabular}
\end{table}

All test problems are set up on regular grids, which permit efficient operator applications: sine transforms
and tridiagonal solves for the Schur complements, and two-dimensional FFTs
for the kernel matrices. The factorization uses only these matvec routines
and the prescribed tree, so it does not depend on the transform-based
implementation. On other geometries, suitable sparse solves or fast kernel
summation routines can supply the matvecs.

All reported times are wall-clock seconds.
Unless otherwise stated, we use leaf size $m=256$, sampling block size
$b=16$ for Schur complements of 2D PDEs and $b=32$ otherwise, freeze margin
$\gamma=1/4$
(\Cref{s:rank-refinement}), and the relaxation factors of
\crefrem{rem:heuristics}. \Cref{ss:block-size} summarizes the block-size
choice; the Supplementary Material, Section~S5, gives the full parameter sweep.

A sample stream is the sequence of Gaussian test matrices drawn from one
seed. The application tables report single runs with MATLAB's Twister generator
seeded at $0$, except \Cref{t:helmholtz-schur}, which reports medians over
three sample streams. In representative Poisson and Mat\'{e}rn tests,
three streams produced less than $1\%$ variation in matvec count and
storage and at most $15\%$ variation in residual. The Helmholtz Schur
complement shows a larger residual variation, so we report medians for
it. Parameter comparisons also use matched streams, drawn from the same
seeds for every setting compared. This matters when the differences
between settings are small relative to the sampling variability.
The block-size comparisons in \Cref{ss:block-size} use medians over four
matched streams; the conditioning diagnostic in \Cref{ss:ls-conditioning}
uses one fixed stream per problem and aggregates over nodes at each level.
The Supplementary Material, Section~S1, gives further variability diagnostics.

All experiments were run on a single Ice Lake compute node of the Stampede3
system at the Texas Advanced Computing Center. Each node has two Intel Xeon
Platinum 8380 processors, giving $80$ cores at a nominal $2.3$\,GHz, and
$256$\,GB of DDR4-3200 memory. Our implementation is written in
MATLAB~R2025b and contains no explicit parallelism; multithreading is
provided by MATLAB's built-in numerical routines.

Unpreconditioned Krylov iteration on the same operator provides a natural
baseline under the matrix-free access model. The reported counts
$\tilde n_{\mathrm{it}}$ isolate the effect of the HSS preconditioner. They reach thousands for the
Helmholtz Schur complement and the Mat\'{e}rn matrices, and they grow
with the frequency for the Lippmann--Schwinger cavity. We do not benchmark against
a separate HSS construction followed by factorization, such as the
construction of~\cite{levitt2024linear} coupled to a conventional HSS
solver. Such a comparison would require another implementation with
matched tolerances and accuracy. The present experiments therefore assess
scaling, accuracy, and preconditioning effectiveness rather than establish
a speed advantage over other hierarchical solvers.

\subsection{Schur Complements of the Poisson Equation}\label{ss:poisson}

Consider Poisson's equation $-\Delta u=f$ with Dirichlet boundary
conditions on a rectangular domain in $d=2$ or $3$ dimensions. A uniform
grid with the five-point stencil in 2D or the seven-point stencil in 3D
gives a sparse matrix $\A$. One level of nested
dissection~\cite{george1973nested} partitions the grid into subdomains
$\mathcal{I}_1,\mathcal{I}_2$ and a separator $\mathcal{I}_3$, yielding
the block structure in \Cref{f:nd}.

Eliminating $\mathcal{I}_1$ and $\mathcal{I}_2$ yields the Schur complement on the
separator:
\begin{equation}\label{eq:poisson-equation}
A = \A_{33} - \A_{31}\A_{11}^{-1}\A_{13} - \A_{32}\A_{22}^{-1}\A_{23}.
\end{equation}
Although $\A$ is sparse, its $N\times N$ Schur complement $A$ is dense.
A conventional dense factorization requires $\mathcal O(N^3)$ arithmetic
and $\mathcal O(N^2)$ storage, illustrating the cost of large separator
blocks in sparse direct methods~\cite{davis2016survey}. The Schur
complement has hierarchical off-diagonal low-rank
structure~\cite{chandrasekaran2010numerical,xia2010superfast,xia2013efficient,ghysels2016efficient,aminfar2016fast,ho2016hierarchical,amestoy2017complexity},
making it a natural test for the proposed factorization. The construction
cost can be amortized over multiple right-hand sides when the underlying
operator is unchanged. A matvec with $A$ requires sparse coupling products
and solves with $\A_{11}$ and $\A_{22}$; their inverses are never formed.
The regular subdomain geometry allows these solves to be performed with
the fast transforms described in \Cref{ss:setup}.

\begin{figure}
\centering
\begin{minipage}[c]{0.40\textwidth}
\centering
\includegraphics[width=0.85\linewidth]{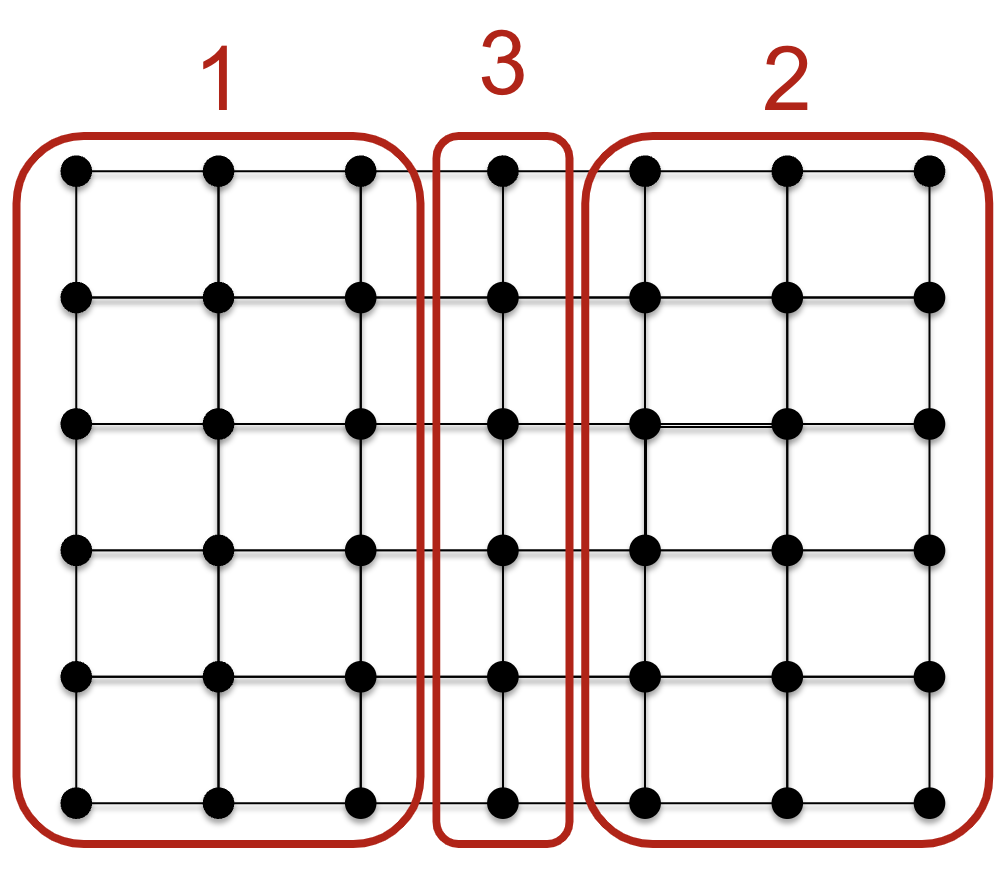}
\end{minipage}\hfill
\begin{minipage}[c]{0.55\textwidth}
\centering
$\displaystyle
\A = \begin{bmatrix}
\A_{11} & 0 & \A_{13} \\
0 & \A_{22} & \A_{23} \\
\A_{31} & \A_{32} & \A_{33}
\end{bmatrix}$
\end{minipage}
\caption{Left: partition of the grid into subdomains $\mathcal{I}_1$, $\mathcal{I}_2$ and
separator $\mathcal{I}_3$. Right: the corresponding block structure of
$\A$. In both dimensions, the transverse width is fixed at $51$ grid
points, giving two subdomains of thickness $25$ separated by one grid
layer. The separator grows with $N$ while this thickness remains fixed.
}
\label{f:nd}
\end{figure}

\subsubsection{Two Dimensions ($d=2$)}\label{ss:poisson-d2}

We discretize the 2D Laplacian with the five-point stencil on a regular
$n_x \times n_y$ grid. The transverse width is $n_x = 51$ points, split as
$25\mid1\mid25$ across the separator. We vary the separator length $n_y = N$.
Fixing the transverse width gives the three Schur-complement families a
common geometry. Additional width tests show a modest effect on rank:
at $N=8192$ and $\epsilon=10^{-6}$, widening the strip from $n_x=51$ to
$801$ increases the largest rank over all levels from $15$ to $33$. In
an additional 3D run (not tabulated) at $N=128^2$ and
$\epsilon=10^{-6}$, doubling the slab thickness leaves the leaf and
coarsest ranks unchanged and raises the middle-level rank from $384$ to
$409$.

\Cref{t:poisson-schur} presents the results, and
\Cref{f:poisson-1D-scaling-plot} shows the scaling. At each tolerance, the
detected ranks remain bounded over the tested sizes, and factorization
time, solve time, and storage grow approximately linearly with $N$,
consistent with \Cref{th:complexity-adaptive}. The selected leaf ranks are
approximately $12$, $16$, $20$, and $28$ at $\epsilon=10^{-3}$,
$10^{-6}$, $10^{-9}$, and $10^{-12}$, respectively; the largest rank at
any level is $48$.

Only a few sampling batches are needed. Starting with leaf size $m=256$ columns and
appending $b=16$ columns per batch gives $288$ matvecs at the two looser
tolerances and $304$ at the two tighter tolerances, independently of
$N$. The relative residual also depends little on $N$ and remains
within a factor of three of $\epsilon$. The condition number is
approximately $73.5$ throughout. This example shows that adaptivity
recovers the expected bounded-rank behavior without prior rank estimates.

\begin{table}
    \centering
    \footnotesize
    \caption{\em Schur complement of the 2D Poisson equation
    (\Cref{ss:poisson-d2}): results as $N$ and $\epsilon$ vary.}
    \label{t:poisson-schur}
    \begin{tabular}{ccccccc}
    \toprule
    $\epsilon$ & $N$ & $n_{\text{mv}}$ & $t_{\text{fact}}$ & $t_{\text{solve}}$ & $\text{relres}$ & $\text{mem}_{\text{MB}}$  \\
    \midrule
    \multirow{5}{*}{$10^{-3}$}
     & $2^{17}$ & 288 & 9.4e+0 & \num{4.6e-1} & \num{1.3e-3} & 2.8e+2  \\
     & $2^{18}$ & 288 & 1.9e+1 & \num{8.1e-1} & \num{1.6e-3} & 5.7e+2  \\
     & $2^{19}$ & 288 & 3.7e+1 & 1.5e+0 & \num{8.2e-4} & 1.1e+3  \\
     & $2^{20}$ & 288 & 7.6e+1 & 3.3e+0 & \num{8.0e-4} & 2.3e+3  \\
     & $2^{21}$ & 288 & 1.5e+2 & 6.1e+0 & \num{1.4e-3} & 4.6e+3  \\
    \midrule
    \multirow{5}{*}{$10^{-6}$}
     & $2^{17}$ & 288 & 1.1e+1 & \num{3.9e-1} & \num{9.5e-7} & 3.0e+2  \\
     & $2^{18}$ & 288 & 2.3e+1 & \num{6.6e-1} & \num{8.4e-7} & 6.0e+2  \\
     & $2^{19}$ & 288 & 4.1e+1 & 1.5e+0 & \num{1.1e-6} & 1.2e+3  \\
     & $2^{20}$ & 288 & 8.1e+1 & 3.0e+0 & \num{1.3e-6} & 2.4e+3  \\
     & $2^{21}$ & 288 & 1.6e+2 & 6.3e+0 & \num{1.4e-6} & 4.8e+3  \\
    \midrule
    \multirow{5}{*}{$10^{-9}$}
     & $2^{17}$ & 304 & 1.3e+1 & \num{3.9e-1} & \num{1.0e-9} & 3.1e+2  \\
     & $2^{18}$ & 304 & 2.6e+1 & \num{8.2e-1} & \num{1.3e-9} & 6.2e+2  \\
     & $2^{19}$ & 304 & 4.9e+1 & 1.4e+0 & \num{9.8e-10} & 1.2e+3  \\
     & $2^{20}$ & 304 & 9.3e+1 & 2.9e+0 & \num{1.2e-9} & 2.5e+3  \\
     & $2^{21}$ & 304 & 1.8e+2 & 6.3e+0 & \num{1.3e-9} & 5.0e+3  \\
    \midrule
    \multirow{5}{*}{$10^{-12}$}
     & $2^{17}$ & 304 & 1.4e+1 & \num{4.0e-1} & \num{2.2e-12} & 3.4e+2  \\
     & $2^{18}$ & 304 & 2.9e+1 & \num{7.8e-1} & \num{2.4e-12} & 6.8e+2  \\
     & $2^{19}$ & 304 & 5.5e+1 & 1.6e+0 & \num{2.6e-12} & 1.4e+3  \\
     & $2^{20}$ & 304 & 1.1e+2 & 2.9e+0 & \num{2.9e-12} & 2.7e+3  \\
     & $2^{21}$ & 304 & 2.0e+2 & 6.4e+0 & \num{2.8e-12} & 5.5e+3  \\
    \bottomrule
    \end{tabular}
\end{table}

\bigskip

\begin{figure}
    \centering
    \begin{minipage}[t]{0.31\textwidth}
        \centering
        \includegraphics[width=\linewidth]{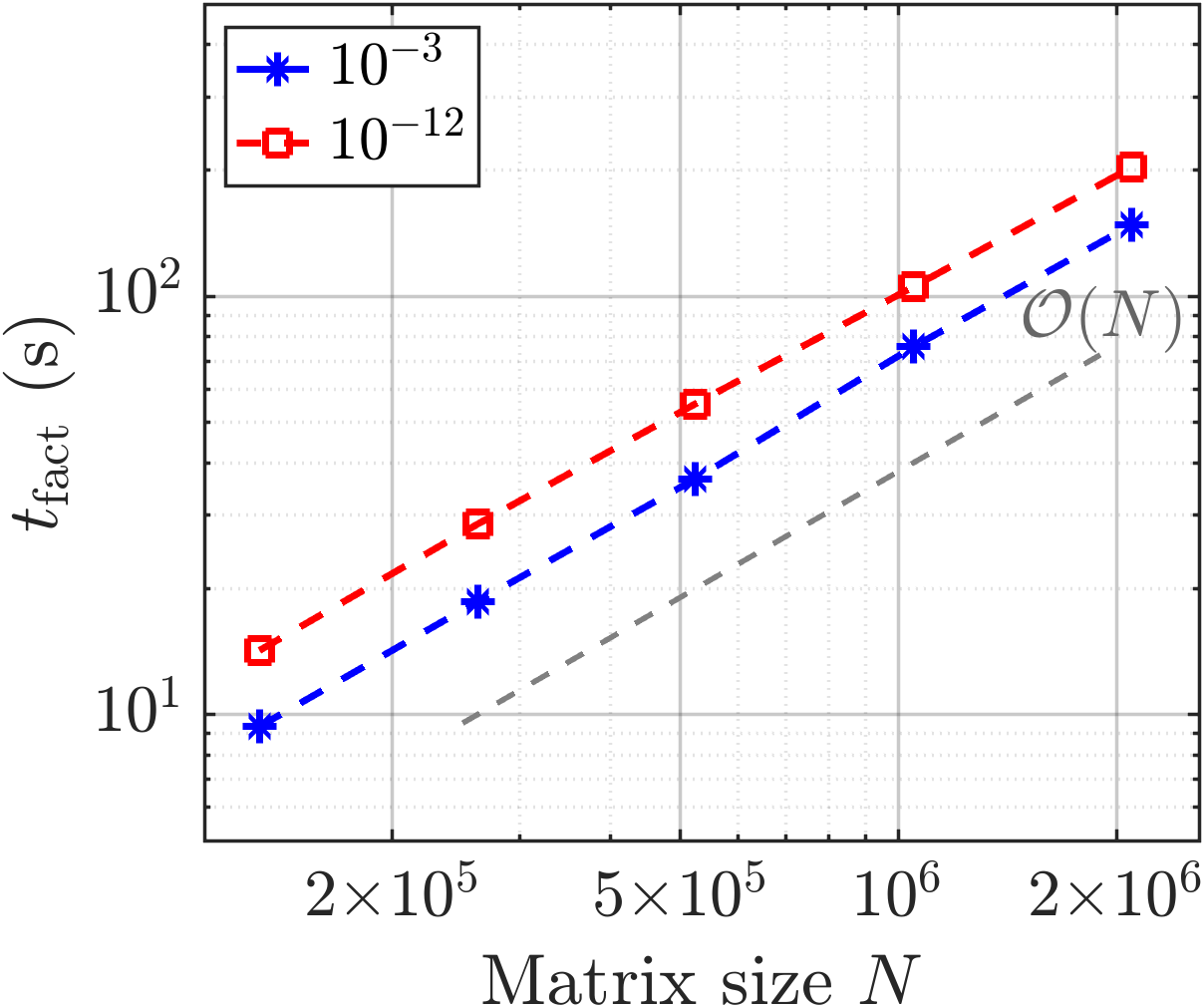}\\[-0.3em]
        \small (a) $t_{\text{fact}}$
    \end{minipage}
    \hfill
    \begin{minipage}[t]{0.31\textwidth}
        \centering
        \includegraphics[width=\linewidth]{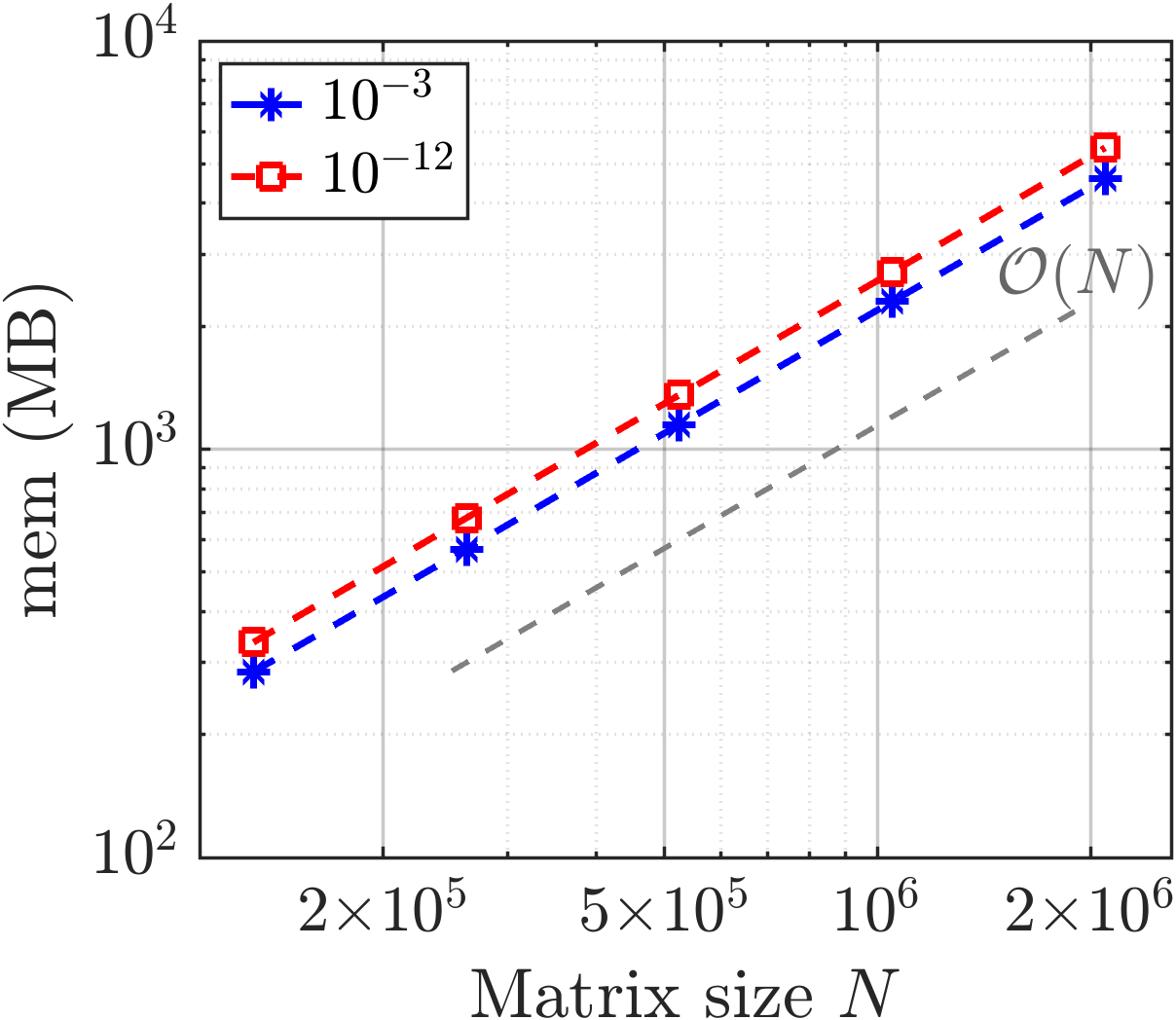}\\[-0.3em] \small (b) $\text{mem}_{\text{MB}}$
    \end{minipage}
    \hfill
    \begin{minipage}[t]{0.31\textwidth}
        \centering
        \includegraphics[width=\linewidth]{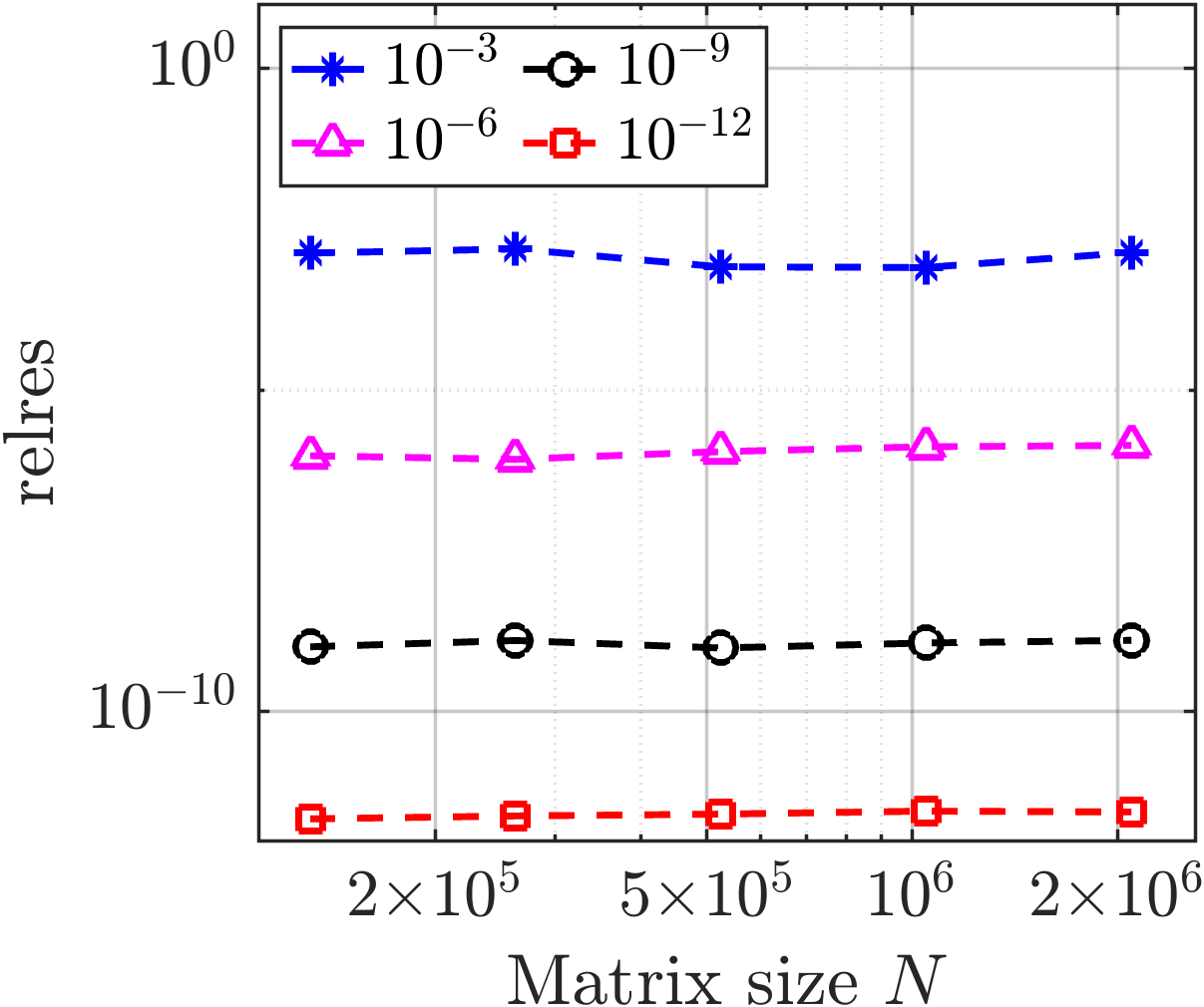}\\[-0.3em] \small (c) relres
    \end{minipage}

    \caption{Scaling plots for the Schur complement of the 2D Poisson
    equation (\Cref{ss:poisson-d2}), from the runs of
    \Cref{t:poisson-schur}: (a) factorization time $t_{\text{fact}}$ and
    (b) factor storage, both at the two extreme tolerances against an
    $\mathcal{O}(N)$ reference, and (c) relres at all four tolerances. All panels
    share the matrix size $N$ on the $x$-axis and use log--log scaling.
    }
    \label{f:poisson-1D-scaling-plot}
\end{figure}

\subsubsection{Three Dimensions ($d=3$)}\label{ss:poisson-d3}

We next consider Poisson's equation in 3D, where the separator is a 2D
plane. We discretize with the seven-point stencil on an
$n_x \times n_y \times n_z$ grid, fix $n_x = 51$ with the same
$25\mid1\mid25$ split, and vary $N = n_y n_z$ with $n_y=n_z$.

\Cref{t:poisson-2D-results} shows rank growth both with $N$ and toward the
root. As $N$ quadruples, the coarsest stored ranks increase by a factor of
approximately $1.9$--$2.1$, consistent with the $\mathcal O(\sqrt N)$
rank model in \Cref{th:complexity-adaptive}. The matvec count, equal here
to the final number of test columns $p$, increases by a factor of
$1.9$--$2.1$. These ranks are detected during construction rather than
supplied as input.

Storage grows by a factor of approximately five to nine per size step,
faster than $N\log N$ over this range. Factorization time increases by a
factor of roughly nine to twelve per size step, corresponding to
empirical power-law exponents of $1.5$--$1.8$. These finite-size measurements do not resolve the
asymptotic $\mathcal O(N^{1.5}\log N)$ time and $\mathcal O(N\log N)$
storage bounds. Nevertheless, the storage savings relative to a dense
matrix are substantial: at $N=512^2$ and $\epsilon=10^{-12}$, the factors
occupy $22$ GB, compared with approximately $550$ GB for $A$ in dense
double precision. One factorized solve takes at most $7.5$ s in the
reported tests.

The relative residual increases with $N$ even though the condition number
grows only moderately, from approximately $67.3$ at $N=64^2$ to $119$
at $256^2$. Across the reported tolerances, relres$/\epsilon$ rises from
approximately $4$--$6$ at $N=64^2$ to $180$--$220$ at $512^2$,
consistent with error accumulation and tolerance relaxation across levels.
At $\epsilon=10^{-12}$, the largest problem has residual
$1.8\times10^{-10}$. We report $\epsilon=10^{-6},10^{-9},10^{-12}$;
the looser choice $10^{-3}$ produces a residual as large as $0.21$.
At the tightest tolerance, the leaf blocks are retained at full rank and
stored densely. This fallback allows the construction to proceed when
leaf-level compression is ineffective.

\begin{table}
    \centering
    \footnotesize
    \setlength{\tabcolsep}{3pt}
    \caption{\em Schur complement of the 3D Poisson equation
    (\Cref{ss:poisson-d3}): maximum stored rank at each level and
    performance as $N$ and $\epsilon$ vary. At $\epsilon=10^{-12}$,
    the leaves are stored densely, so the reported leaf rank equals the
    leaf size. A dash denotes a level absent from the tree.}
    \label{t:poisson-2D-results}
    \begin{tabular}{cccccccccccc}
        \toprule
        & & \multicolumn{5}{c}{Max rank at level} & & & & & \\
        \cmidrule(lr){3-7}
        $\epsilon$ & $N$ & 5 & 4 & 3 & 2 & 1
        & $n_{\text{mv}}$ & $t_{\text{fact}}$
        & $t_{\text{solve}}$ & relres & mem$_{\text{MB}}$
        \\
        \midrule
        \multirow{4}{*}{$10^{-6}$}
         & $64^2$ & -- & -- & -- & 174 & 224 & 1016 & 1.9e+0 & \num{2.0e-2} & \num{5.6e-6} & 4.4e+1  \\
         & $128^2$ & -- & -- & 177 & 384 & 416 & 1952 & 1.7e+1 & \num{1.0e-1} & \num{3.4e-5} & 3.0e+2  \\
         & $256^2$ & -- & 179 & 384 & 768 & 800 & 3872 & 1.5e+2 & \num{5.7e-1} & \num{9.9e-5} & 1.7e+3  \\
         & $512^2$ & 180 & 384 & 768 & 1536 & 1600 & 7744 & 1.3e+3 & 3.2e+0 & \num{2.2e-4} & 8.8e+3  \\
        \midrule
        \multirow{4}{*}{$10^{-9}$}
         & $64^2$ & -- & -- & -- & 212 & 288 & 1296 & 2.6e+0 & \num{2.0e-2} & \num{4.4e-9} & 6.0e+1  \\
         & $128^2$ & -- & -- & 216 & 512 & 576 & 2624 & 3.0e+1 & \num{1.3e-1} & \num{2.4e-8} & 4.6e+2  \\
         & $256^2$ & -- & 218 & 512 & 1056 & 1088 & 5312 & 3.1e+2 & \num{8.5e-1} & \num{9.5e-8} & 2.8e+3  \\
         & $512^2$ & 219 & 512 & 1056 & 2144 & 2272 & 10848 & 3.3e+3 & 4.5e+0 & \num{2.1e-7} & 1.5e+4  \\
        \midrule
        \multirow{4}{*}{$10^{-12}$}
         & $64^2$ & -- & -- & -- & 256 & 337 & 1504 & 3.6e+0 & \num{4.0e-2} & \num{4.5e-12} & 6.4e+1  \\
         & $128^2$ & -- & -- & 256 & 606 & 704 & 3128 & 4.3e+1 & \num{1.6e-1} & \num{1.9e-11} & 5.6e+2  \\
         & $256^2$ & -- & 256 & 607 & 1312 & 1408 & 6656 & 4.8e+2 & 1.3e+0 & \num{9.1e-11} & 3.7e+3  \\
         & $512^2$ & 256 & 632 & 1312 & 2784 & 3008 & 14144 & 5.9e+3 & 7.5e+0 & \num{1.8e-10} & 2.2e+4  \\
        \bottomrule
    \end{tabular}
\end{table}
\subsection{Schur Complement of the Helmholtz Equation}
\label{ss:helmholtz-schur}

To test an oscillatory, indefinite problem, we replace the Poisson
operator by the Helmholtz operator $-\Delta-\omega^2$ while retaining the
same strip geometry. The resulting matrices are real symmetric but
indefinite, and their off-diagonal ranks depend on frequency. We discretize
\begin{equation}\label{eq:helmholtz-strip}
  -\Delta u - \omega^2 u = f
\end{equation}
on the same $51\times N$ grid and with the same $25\mid1\mid25$ split as
\Cref{ss:poisson-d2}, so that the Schur complement is again $N\times N$ on
the separator. With the five-point stencil scaled by $h^2$, where $h=1/(N+1)$ is the
grid spacing, the discrete
operator is the five-point Laplacian shifted by $-(\omega h)^2$ on the
diagonal. We hold the resolution fixed at $32$ grid points per
wavelength, so that $\omega h$ is constant and the separator spans
approximately $N/32$ wavelengths. Increasing $N$ therefore tests the
adaptive construction as the frequency increases.

The strip acts as a waveguide with reflecting walls. At approximately
$1.6$ wavelengths across, it supports three propagating transverse sine
modes; the remaining modes are evanescent. Widening the strip increases
the number of propagating modes and the detected ranks. For $N=4096$
and $\epsilon=10^{-6}$, the largest rank in a single sample stream is
$14$, $25$, $39$, and $46$ at $n_x=51$, $101$, $201$, and $401$,
respectively (the three-stream median at $n_x=51$ is $15$). For the
fixed width used in \Cref{t:helmholtz-schur}, each $256$-point leaf
spans eight wavelengths, independently of $N$. Its rank therefore does
not grow with $N$; the growth of $k_{\max}$ in \Cref{t:helmholtz-schur}
occurs at the coarser nodes, which cover progressively more wavelengths.

The matvec count remains $288$ at $\epsilon=10^{-3}$ and $10^{-6}$ and
$304$ at $10^{-9}$, as for the Poisson Schur complement. The maximum rank grows from $10$ to $15$
at $\epsilon=10^{-3}$, from $15$ to $24$ at $10^{-6}$, and from $18$ to
$30$ at $10^{-9}$, while storage and factorization time approximately
double with $N$. Unpreconditioned
GMRES requires $1370$--$5100$ iterations at tolerance $10^{-12}$.
The HSS preconditioner reduces this count to at most $10$, a reduction
of more than two orders of magnitude throughout the table.

The solve residual is more sensitive than in the Poisson case: its median
ranges from $17\epsilon$ to $182\epsilon$. To distinguish approximation
quality from solve accuracy, we estimate the compression error after
construction. The estimate is
$\|(A-A_{\mathrm{approx}})\Xi\|_{\mathrm F}/\|A\Xi\|_{\mathrm F}$,
where $\Xi$ is an independent Gaussian test matrix with eight columns. It
is reported as cmp err and measures the operator approximation, whereas
relres measures the direct solve. The
median compression error is $0.4\epsilon$--$6.3\epsilon$. The larger
solve residual indicates sensitivity to these perturbations; the two
metrics alone do not separate compression error from finite-precision
error in elimination or inversion.

Local diagnostics identify a potential amplification mechanism. The
elimination factors contain $X^{sr}(X^{rr})^{-1}$, which can magnify
errors passed between levels. At $N=\num{16384}$ and
$\epsilon=10^{-6}$, the largest observed $\kappa(X^{rr})$ is
$1.2\times10^4$ and the largest $\|X^{sr}(X^{rr})^{-1}\|$ is
$2.6\times10^3$. For comparison, the Schur complement of the 3D Poisson
equation at $N=256^2$ and $\epsilon=10^{-9}$ has coarsest-level
condition numbers between $10$ and $55$ and amplification factors at
most $2.5$.

Across the three sample streams, the matvec count is identical and
storage varies by less than one percent, but the residual varies by up
to a factor of twelve (Supplementary Material, Section~S1). At
$\epsilon=10^{-9}$, the spread is at most a factor of six and no
preconditioned solve stagnates.

The largest tested size is $N=\num{16384}$. The closed strip has
increasingly dense resonances as $N$ grows. The frequency tests in
the Supplementary Material, Section~S3, show that near-resonance limits the attainable
accuracy for larger sizes in this undamped configuration.

\begin{table}
    \centering
    \scriptsize
    \setlength{\tabcolsep}{3pt}
    \caption{\em 2D Helmholtz Schur complement at $32$ grid points per
    wavelength (\Cref{ss:helmholtz-schur}), medians over three matched
    sample streams. $k_{\max}$ is the largest rank over all levels, and
    cmp err is the compression error of
    \Cref{ss:helmholtz-schur}. $\tilde n_{\text{it}}$ depends on $N$ only.}
    \label{t:helmholtz-schur}
    \begin{tabular}{@{}ccccccccccc@{}}
    \toprule
    $\epsilon$ & $N$ & $k_{\max}$ & $n_{\text{mv}}$
    & $t_{\text{fact}}$ & $t_{\text{solve}}$ & $\text{cmp err}$ & $\text{relres}$
    & $\text{mem}_{\text{MB}}$ & $n_\text{it}$ & $\tilde n_\text{it}$ \\
    \midrule
    \multirow{3}{*}{$10^{-3}$}
     & $2^{12}$  & 10 & 288 & \num{4.6e-1} & \num{2.0e-2} & \num{3.9e-4} & \num{1.7e-2} & 8.8e+0 & 8 & 1370 \\
     & $2^{13}$  & 12 & 288 & \num{8.4e-1} & \num{4.0e-2} & \num{1.5e-3} & \num{1.8e-1} & 1.8e+1 & 10 & 2604 \\
     & $2^{14}$  & 15 & 288 & 1.6e+0 & \num{8.0e-2} & \num{1.4e-3} & \num{6.1e-2} & 3.5e+1 & 10 & 5100 \\
    \midrule
    \multirow{3}{*}{$10^{-6}$}
     & $2^{12}$  & 15 & 288 & \num{4.6e-1} & \num{2.0e-2} & \num{9.3e-7} & \num{3.5e-5} & 9.1e+0 & 4 & 1370 \\
     & $2^{13}$  & 20 & 288 & \num{8.5e-1} & \num{4.0e-2} & \num{1.3e-6} & \num{4.5e-5} & 1.8e+1 & 3 & 2604 \\
     & $2^{14}$  & 24 & 288 & 1.6e+0 & \num{7.0e-2} & \num{6.3e-6} & \num{1.1e-4} & 3.7e+1 & 4 & 5100 \\
    \midrule
    \multirow{3}{*}{$10^{-9}$}
     & $2^{12}$  & 18 & 304 & \num{4.6e-1} & \num{2.0e-2} & \num{6.8e-10} & \num{3.8e-8} & 9.5e+0 & 2 & 1370 \\
     & $2^{13}$  & 27 & 304 & \num{8.7e-1} & \num{4.0e-2} & \num{6.4e-10} & \num{3.1e-8} & 1.9e+1 & 2 & 2604 \\
     & $2^{14}$  & 30 & 304 & 1.7e+0 & \num{7.0e-2} & \num{2.6e-9} & \num{8.9e-8} & 3.8e+1 & 2 & 5100 \\
    \bottomrule
    \end{tabular}
\end{table}

\subsection{Mat\'{e}rn Covariance Matrices}\label{ss:kernels}\label{ss:matern-kernel-2D}

Mat\'{e}rn covariance matrices, widely used to model spatial
correlations~\cite{genton2002classes,rasmussen2005gaussian}, test the
method on ill-conditioned SPD problems. Forming the dense matrix requires
$\mathcal O(N^2)$ kernel evaluations and entries. On a uniform grid,
translation invariance gives block Toeplitz structure with Toeplitz blocks,
allowing matvecs in $\mathcal O(N\log N)$ time through a two-dimensional
FFT. Reusing a factorization can be useful for multiple right-hand sides
at fixed covariance parameters, as in kriging or conditional simulation.
The kernel is
\begin{equation}\label{eq:matern-kernel}
K(r) = \frac{2^{1-\nu}}{\Gamma(\nu)}
\left( \frac{\sqrt{2\nu}\,\|r\|_2}{\ell} \right)^\nu
K_\nu\!\left( \frac{\sqrt{2\nu}\,\|r\|_2}{\ell} \right),
\end{equation}
where $\nu$ controls smoothness, $\ell$ is the correlation length,
$\Gamma$ is the gamma function, and $K_\nu$ is the modified Bessel
function of the second kind. We discretize on a uniform
$n \times n$ grid over $[0,1]^2$ with
$N = n^2$ points, giving
\begin{equation}\label{eq:matern-matrix-entries}
A_{ij} =
\begin{cases}
K(x_i - x_j), & i \neq j, \\
1 + \sigma^2, & i = j,
\end{cases}
\end{equation}
where the nugget $\sigma^2 = 10^{-4}$ is added for numerical stability.
We compare $(\nu,\ell)=(0.5,0.5)$ with $(1.5,0.1)$; the latter has
higher smoothness and a shorter correlation length. Both matrices are
highly ill-conditioned. At $N=256^2$, $\kappa(A)\approx8.1\times10^6$
for $(0.5,0.5)$ and $3.2\times10^7$ for $(1.5,0.1)$. Over the tested
sizes, unpreconditioned CG needs $538$ to $4996$ iterations.
\Cref{t:matern-2D} presents the results.

\begin{table}
    \centering
    \scriptsize
    \setlength{\tabcolsep}{2pt}
    \caption{\em Mat\'{e}rn covariance matrices (\Cref{ss:matern-kernel-2D})
    for $(\nu,\ell)=(0.5,0.5)$ and $(1.5,0.1)$: maximum rank at each
    level and performance as $N$ and $\epsilon$ vary. Unpreconditioned
    counts were computed separately with the same matrix, right-hand side,
    and solver tolerance; they are independent of $\epsilon$.}
    \label{t:matern-2D}
    \begin{tabular}{cccccccccccccc}
        \toprule
        & & \multicolumn{5}{c}{Max rank at level} & & & & & & \\
        \cmidrule(lr){3-7}
        $\epsilon$ & $N$ & 5 & 4 & 3 & 2 & 1
        & $n_{\text{mv}}$ & $t_{\text{fact}}$
        & $t_{\text{solve}}$ & relres & mem$_{\text{MB}}$
        & $n_\text{it}$ & $\tilde{n}_\text{it}$ \\ \midrule \multicolumn{14}{l}{{\boldmath$(\nu,\ell)=(0.5,0.5)$}}\\ \midrule \multirow{3}{*}{$10^{-9}$}
         & $128^2$ & -- & -- & 174 & 378 & 384 & 1896 & 1.5e+1 & \num{1.0e-1} & \num{1.1e-10} & 2.9e+2 & 2 & 538 \\
         & $256^2$ & -- & 162 & 350 & 608 & 640 & 3072 & 1.1e+2 & \num{4.5e-1} & \num{3.8e-10} & 1.3e+3 & 2 & 1053 \\
         & $512^2$ & 153 & 319 & 544 & 1056 & 1056 & 5280 & 6.3e+2 & 1.5e+0 & \num{1.9e-9} & 5.5e+3 & 4 & 1689 \\
        \midrule
        \multirow{3}{*}{$10^{-12}$}
         & $128^2$ & -- & -- & 225 & 512 & 576 & 2624 & 2.9e+1 & \num{1.2e-1} & \num{8.5e-14} & 4.7e+2 & 1 & 538 \\
         & $256^2$ & -- & 218 & 480 & 928 & 960 & 4672 & 2.3e+2 & \num{6.6e-1} & \num{3.4e-13} & 2.4e+3 & 1 & 1053 \\
         & $512^2$ & 211 & 473 & 864 & 1728 & 1760 & 8672 & 1.9e+3 & 3.6e+0 & \num{7.3e-13} & 1.2e+4 & 1 & 1689 \\
        \midrule[\heavyrulewidth] \multicolumn{14}{l}{{\boldmath$(\nu,\ell)=(1.5,0.1)$}}\\ \midrule \multirow{3}{*}{$10^{-9}$}
         & $128^2$ & -- & -- & 182 & 414 & 448 & 2104 & 1.8e+1 & \num{1.2e-1} & \num{1.8e-9} & 3.3e+2 & 3 & 2392 \\
         & $256^2$ & -- & 161 & 352 & 640 & 736 & 3296 & 1.1e+2 & \num{4.7e-1} & \num{2.7e-8} & 1.4e+3 & 4 & 3925 \\
         & $512^2$ & 140 & 288 & 480 & 960 & 1216 & 5376 & 4.9e+2 & 1.3e+0 & \num{1.4e-6} & 4.7e+3 & 8 & 4996 \\
        \midrule
        \multirow{3}{*}{$10^{-12}$}
         & $128^2$ & -- & -- & 229 & 541 & 608 & 2772 & 3.0e+1 & \num{1.4e-1} & \num{5.1e-13} & 5.1e+2 & 1 & 2392 \\
         & $256^2$ & -- & 213 & 480 & 960 & 1056 & 4896 & 2.2e+2 & \num{6.5e-1} & \num{7.0e-12} & 2.5e+3 & 2 & 3925 \\
         & $512^2$ & 198 & 443 & 832 & 1696 & 1888 & 8672 & 1.7e+3 & 2.9e+0 & \num{4.1e-11} & 1.1e+4 & 2 & 4996 \\
        \bottomrule
    \end{tabular}
\end{table}

The HSS factors substantially reduce the Krylov iteration counts. At
$\epsilon=10^{-12}$, PCG converges in one or two iterations for every
reported size and parameter set; at $10^{-9}$, it requires two to eight.
We focus on these tolerances because, in runs not tabulated here, looser
compression gave weaker preconditioners for these ill-conditioned
matrices.

Direct-solve accuracy differs markedly between the parameter sets.
For $(\nu,\ell)=(0.5,0.5)$, relres remains at or below $2\epsilon$,
and PCG requires at most four iterations. For $(1.5,0.1)$ at
$\epsilon=10^{-9}$, relres$/\epsilon$ grows from $1.8$ at $N=128^2$
to $27$ at $256^2$ and $1350$ at $512^2$. In the largest case, the
direct-solve residual is $1.4\times10^{-6}$ and PCG requires $8$
iterations. Tightening to $\epsilon=10^{-12}$ reduces the direct-solve
residual to $4.1\times10^{-11}$ and the PCG count to two. These are
distinct accuracy measures: the tabulated relres precedes PCG refinement.

The cost must therefore be interpreted together with the delivered
accuracy. At $N=512^2$ and $\epsilon=10^{-9}$, the short-correlation
case $(1.5,0.1)$ needs less factorization time than $(0.5,0.5)$, $490$
versus $630$ s, and less storage, $4.7$ versus $5.5$ GB. Its
direct-solve residual, however, is much larger. The same local tolerance does not yield
comparable global solve accuracy for these two matrices.

The factorization time $t_{\mathrm{fact}}$, which excludes the matvecs,
grows with empirical exponents between $1.2$ and $1.5$ over the tested
range, with the larger exponent at the tighter tolerance. Memory grows by
a factor of $3.4$ to $5.1$ per quadrupling of $N$ in the four series,
bracketing the factor of about $4.5$ of $\mathcal{O}(N\log N)$ and
consistent with the storage bound of \Cref{th:complexity-adaptive}.

\subsection{The Lippmann--Schwinger Equation}\label{app:lippmann}

The preceding examples use real matrices. Our final test is an oscillatory,
complex-symmetric matrix satisfying $A^\top=A$ but $A^*\ne A$. The
specialization of \Cref{s:fixed-rank} sets $V_i=\overline{U_i}$ and uses
$U_i^*A_{ii}\overline{U_i}$, so only matvecs with $A$ are needed.
The construction uses circular complex Gaussian samples.

Consider the Lippmann--Schwinger equation for time-harmonic scattering:
\begin{equation} \label{e:ls}
\sigma(x) + \omega^2 \eta(x) \int_D G(x-y)\,\sigma(y)\,dy
= -\omega^2 \eta(x)\,u_0(x), \quad x \in D,
\end{equation}
where $\sigma$ is unknown, $u_0$ is the incoming wave,
$0\le\eta\le1$ is the scattering potential, and
$G(r)=\tfrac{\mathrm{i}}{4}H_0^{(1)}(\omega\|r\|_2)$ is the free-space Green's
function for $\Delta+\omega^2$. Here $\omega=2\pi q$, with $q$ wavelengths
per unit length.

We set $D=[0,1]^2$ with a uniform $\sqrt{N}\times\sqrt{N}$
grid. Writing $\sigma(x)=\sqrt{\eta(x)}\mu(x)$ gives a symmetrized
integral equation, whose piecewise-constant collocation yields the
complex-symmetric matrix
\begin{equation}\label{e:helmholtz-2D-gaussian-discretized}
A_{ij} = \begin{cases}
\displaystyle h^2 \omega^2 \sqrt{\eta(x_i)\,\eta(x_j)}\,G(x_i - x_j),
& i \ne j, \\[6pt]
\displaystyle 1 + \omega^2 \eta(x_i) \!\int_{-h/2}^{h/2}\!\int_{-h/2}^{h/2}
\frac{\mathrm{i}}{4}\,H_0^{(1)}(\omega\sqrt{u^2+v^2})\,du\,dv,
& i = j,
\end{cases}
\end{equation}
where $h = 1/\sqrt{N}$, and the diagonal entries are computed by adaptive
quadrature using MATLAB's \texttt{integral2}. The Green's convolution
matrix is block Toeplitz with Toeplitz blocks on this grid. Applying $A$
combines this convolution with diagonal scaling by $\sqrt{\eta}$ and
the identity, so a matvec still costs $\mathcal O(N\log N)$ through a
two-dimensional FFT. The spatially varying potential means that $A$
itself is not Toeplitz. The symmetrized system remains defined where
$\eta=0$, with the corresponding rows giving $\mu=0$.

The scattering potential is the cavity of~\cite[Section~6.2]{gopal2022accelerated}:
in polar coordinates $(r,\theta)$ about $c=(1/2,1/2)$,
\begin{equation}\label{e:cavity-potential}
  \eta(x)=\bigl(1-\sin^{500}(\theta/2)\bigr)\,
  \exp\bigl(-2000\,(0.1-r^2)^2\bigr),
\end{equation}
a thin ring of radius $\sqrt{0.1}$ with a narrow opening on the side
facing $-x_1$. Field trapping in the ring produces near-resonant behavior
that is challenging for iterative solvers. The potential is negligible
over much of the domain, leaving many leaf off-diagonal blocks numerically
zero. This motivates the levelwise absolute tolerance in the stopping
test of \Cref{ss:build-u-v-from-nullspace}. We maintain $32$ grid points per
wavelength, so $N=(32q)^2$ and the frequency increases with $N$.

To illustrate the scattering geometry, \Cref{f:cavity-fields} shows the
potential and total field $u=u_0+\int_D G\sigma$ for a plane wave
$u_0(x)=e^{\mathrm{i}2\pi qx_1}$ incident from the left. The maximum field
magnitudes are $2.8$ and $6.1$ times the incident amplitude at $q=8$ and
$16$, respectively. These field plots use the physical scattering
right-hand side; the performance table uses the manufactured right-hand
side $f=Ax_\star$ specified in \Cref{ss:setup}.

\begin{figure}
    \centering
    \begin{minipage}[t]{0.235\textwidth}
        \centering
        \includegraphics[width=\linewidth]{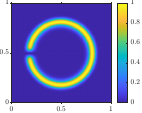}\\[-0.3em]
        \small (a) potential $\eta$
    \end{minipage}
    \hfill
    \begin{minipage}[t]{0.235\textwidth}
        \centering
        \includegraphics[width=\linewidth]{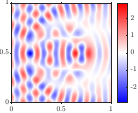}\\[-0.3em]
        \small (b) $\operatorname{Re}u$, $q=8$
    \end{minipage}
    \hfill
    \begin{minipage}[t]{0.235\textwidth}
        \centering
        \includegraphics[width=\linewidth]{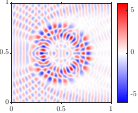}\\[-0.3em]
        \small (c) $\operatorname{Re}u$, $q=16$
    \end{minipage}
    \hfill
    \begin{minipage}[t]{0.235\textwidth}
        \centering
        \includegraphics[width=\linewidth]{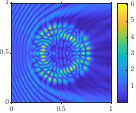}\\[-0.3em]
        \small (d) $|u|$, $q=16$
    \end{minipage}
    \caption{The cavity problem of \Cref{app:lippmann}: (a) the scattering
    potential $\eta$ of \Cref{e:cavity-potential}; (b), (c) the real part
    of the total field $u$ at $q=8$ and $q=16$; (d) its magnitude at
    $q=16$. The axes are $x_1$ (horizontal) and $x_2$ (vertical) on
    $[0,1]^2$, and the wave is incident from the left. The field is
    concentrated inside the ring, with greater amplitude at $q=16$ in
    these two examples. Panel (b) uses the cavity and frequency of~\cite[Figure~10]{gopal2022accelerated}.}
    \label{f:cavity-fields}
\end{figure}

\Cref{t:helmholtz-2D-step-32-gaussian} reports the results. Tightening
$\epsilon$ reduces the direct-solve residual at every size. At $N=128^2$,
for example, relres falls from $3.2\times10^{-4}$ to
$4.6\times10^{-12}$ as $\epsilon$ decreases from $10^{-3}$ to
$10^{-12}$. At fixed $\epsilon$, accuracy deteriorates as the frequency
and matrix size increase. At $N=1024^2$ and $\epsilon=10^{-3}$, the
residual reaches $0.11$, so the factorization is useful primarily as a
preconditioner.

On these runs, the compression-error estimate of
\Cref{ss:helmholtz-schur} (not tabulated) lies between $0.3$ and $3.4$
times relres.
Thus, unlike the Helmholtz Schur complement, the operator approximation
and the solve residual lose accuracy at comparable rates. Each size step adds a tree
level and doubles the frequency. The level-dependent tolerance of
\crefrem{rem:heuristics} relaxes by a factor of two per level; error
propagation through elimination and increasingly oscillatory interactions
can contribute further. These effects are consistent with the observed
growth in residual, but the measurements do not isolate their individual
contributions.

Increasing frequency and tighter tolerances also increase the retained
ranks. At $N=512^2$, the root dimensions are $384$, $1152$, $2176$, and
$9728$ at the four tolerances; no run was made at $N=1024^2$ and
$\epsilon=10^{-12}$. For $\epsilon=10^{-3},10^{-6},10^{-9}$,
factorization time increases by a factor of approximately $5$ to $9$
when $N$ quadruples, and storage by a factor of $4$ to $5.5$. At
$\epsilon=10^{-12}$, the time increases by a factor of approximately
$11$ at each of the two size steps. The data show the
cost of growing ranks; they do not establish an asymptotic rate.

For the fixed right-hand side and solver tolerance, unpreconditioned
GMRES requires $23$, $58$, $214$, and $785$ iterations as $N$ grows from
$128^2$ to $1024^2$. If the count continued to grow proportionally to
$N$, its matvec work alone would scale as $\mathcal O(N^2\log N)$.
Under the two-dimensional rank model of \Cref{th:complexity-adaptive},
the HSS construction instead has cost $\mathcal O(N^{1.5}\log N)$. This conditional
comparison motivates factorization for increasingly difficult iterative
solves. The tested sizes, however, do not establish an asymptotic
crossover, and a timing comparison must include construction cost and
match the delivered accuracy.

Preconditioning substantially reduces the iteration counts: they are
$2$--$3$ for $\epsilon\le10^{-6}$ and at most $12$ for
$\epsilon=10^{-3}$. In particular, at $N=1024^2$ the coarse
factorization reduces the GMRES count from $785$ to $12$ despite its
direct-solve residual of $0.11$. Tightening the compression tolerance
does not eliminate the accuracy gap: at $\epsilon=10^{-12}$, relres
increases from $4.6\times10^{-12}$ at $N=128^2$ to
$1.5\times10^{-10}$ at $512^2$. At the latter size, storage increases
from $7.8$ GB at $\epsilon=10^{-9}$ to $14$ GB at $10^{-12}$. This
tradeoff illustrates the limitations of controlling a global solve
through local compression tolerances. Specialized high-frequency direct
solvers offer another approach to this
regime~\cite{gillman2015spectrally,gopal2022accelerated}.

Some preconditioned GMRES runs stagnate at relative preconditioned
residuals of $1.9\times10^{-12}$ to $7.0\times10^{-12}$, slightly above
the requested tolerance. For these runs, MATLAB returns a
stagnation flag. The starred count is the first iteration whose residual
norm lies within $1\%$ of the smallest value in the recorded history.
Later stalled iterations are not counted. The starred count is therefore
not a count of iterations to convergence at $10^{-12}$.

\begin{table}
  \centering
  \footnotesize
  \caption{\em Lippmann--Schwinger results for the cavity potential
  (\Cref{app:lippmann}) as $N$ and $\epsilon$ vary. Here $q$ is the
  number of wavelengths per unit length and $|A^{(0)}|$ is the root
  dimension. For the fixed right-hand side and solver tolerance,
  $\tilde n_\text{it}$ is independent of $\epsilon$. Starred counts mark
  stagnated preconditioned runs (see the text).}
  \label{t:helmholtz-2D-step-32-gaussian}

    \setlength{\tabcolsep}{4pt}
    \begin{tabular}{ccccccccccc}
      \toprule
      $\epsilon$ & $N$ & $q$ & $\big|A^{(0)}\big|$ & $n_{\text{mv}}$ & $t_{\text{fact}}$ & $t_{\text{solve}}$ & $\text{relres}$ & $\text{mem}_{\text{MB}}$ & $n_\text{it}$ & $\tilde{n}_\text{it}$ \\
      \midrule
      \multirow{4}{*}{$10^{-3}$}
      & $128^2$ & 4 & 256 & 364 & 5.0e+0 & \num{1.2e-1} & \num{3.2e-4} & 1.1e+2 & 4 & 23 \\
      & $256^2$ & 8 & 256 & 480 & 2.4e+1 & \num{4.8e-1} & \num{1.7e-3} & 5.2e+2 & 5 & 58 \\
      & $512^2$ & 16 & 384 & 640 & 1.2e+2 & 2.5e+0 & \num{8.0e-3} & 2.3e+3 & 7 & 214 \\
      & $1024^2$ & 32 & 640 & 800 & 5.7e+2 & 7.8e+0 & \num{1.1e-1} & 8.7e+3 & 12 & 785 \\
      \midrule
      \multirow{4}{*}{$10^{-6}$}
      & $128^2$ & 4 & 512 & 768 & 1.1e+1 & \num{1.8e-1} & \num{2.7e-7} & 2.2e+2 & 2 & 23 \\
      & $256^2$ & 8 & 768 & 1216 & 6.8e+1 & \num{7.1e-1} & \num{1.5e-6} & 1.1e+3 & 3 & 58 \\
      & $512^2$ & 16 & 1152 & 1952 & 4.4e+2 & 2.5e+0 & \num{1.0e-5} & 5.2e+3 & 3 & 214 \\
      & $1024^2$ & 32 & 2176 & 3360 & 3.3e+3 & 1.1e+1 & \num{1.7e-4} & 2.3e+4 & $3^\ast$ & 785 \\
      \midrule
      \multirow{4}{*}{$10^{-9}$}
      & $128^2$ & 4 & 768 & 1008 & 1.4e+1 & \num{1.5e-1} & \num{5.2e-10} & 2.7e+2 & 2 & 23 \\
      & $256^2$ & 8 & 1280 & 1856 & 1.1e+2 & \num{7.0e-1} & \num{1.3e-9} & 1.5e+3 & 2 & 58 \\
      & $512^2$ & 16 & 2176 & 3488 & 8.9e+2 & 3.4e+0 & \num{1.4e-8} & 7.8e+3 & $2^\ast$ & 214 \\
      & $1024^2$ & 32 & 3968 & 6368 & 8.0e+3 & 1.5e+1 & \num{5.8e-8} & 3.8e+4 & $2^\ast$ & 785 \\
      \midrule
      \multirow{3}{*}{$10^{-12}$}
      & $128^2$ & 4 & 1408 & 1592 & 3.0e+1 & \num{2.1e-1} & \num{4.6e-12} & 4.5e+2 & 2 & 23 \\
      & $256^2$ & 8 & 3840 & 4224 & 3.3e+2 & 1.1e+0 & \num{4.9e-11} & 2.7e+3 & 2 & 58 \\
      & $512^2$ & 16 & 9728 & 10720 & 3.7e+3 & 5.7e+0 & \num{1.5e-10} & 1.4e+4 & $2^\ast$ & 214 \\
      \bottomrule
    \end{tabular}

\end{table}

\section{Conclusions} \label{s:end}

We developed an adaptive HSS construction that combines fully matrix-free access, tolerance-driven local rank selection, $\mathcal{O}(k)$ matvecs, and simultaneous compression and factorization. The algorithm needs only matvecs with $A$ and, in general, with its transpose or adjoint. It queries no matrix entries. It produces the HSS factorization in the same upward pass as the compression. To our knowledge, no previous HSS method combines all four
properties. The adaptive construction extends fixed-rank block nullification:
it continues to use the existing samples and bases, selects local ranks from
randomized Frobenius-norm estimates, and extends each nullspace basis by a
unitary Householder row update. The resulting local dense work is
$\mathcal{O}(k^3)$, which matches the asymptotic cost when the ranks are known in
advance. Pruning the last basis block removes much of the rank overshoot without reprocessing the accumulated basis.

The experiments confirm the predicted behavior where the rank model
holds. On the 2D Poisson Schur complement, the ranks stay bounded, the
matvec count is $288$ or $304$ independently of $N$ up to $N=2^{21}$, and
time and storage scale linearly. On the 3D Poisson Schur complement, the
ranks grow like $\sqrt N$, as the two-dimensional model predicts, and the
method detects this growth without any rank input. On the Helmholtz Schur
complement, the Mat\'{e}rn matrices, and the Lippmann--Schwinger cavity,
the factorization reduces the Krylov iteration counts, which reach the
hundreds or thousands at the larger sizes, to at most twelve.

The experiments also show a central limitation. Local randomized tolerances
do not directly control the final solve residual. Rank growth,
ill-conditioned compressed diagonal blocks, and error accumulation through
the triangular elimination factors can amplify local errors across the tree.
In these cases, the factorization can still provide an effective
preconditioner. Predicting end-to-end accuracy, however, requires a unified
error budget that couples the sampling tolerance $\epsilon$, the tree depth $L$, the local conditioning of the elimination blocks, and the working precision $u$. Developing this error budget is the most important next step toward a
fully accuracy-controlled matrix-free factorization.

\begin{acknowledgements}
\begin{sloppypar}
The authors thank Per-Gunnar Martinsson for helpful discussions.
The authors used generative-AI tools, including OpenAI's ChatGPT and Codex
and Anthropic's Claude, for language editing, manuscript organization, and
assistance with software development and debugging. The authors reviewed
and, where appropriate, revised all AI-assisted text and code, and they
verified all numerical claims and software changes. The authors take full
responsibility for the manuscript and accompanying software.
\end{sloppypar}
\end{acknowledgements}
\section*{Declarations}

\paragraph{Funding.}
This work used Stampede3 at the Texas Advanced Computing Center through
allocation TG-MTH250007 from the Advanced Cyberinfrastructure Coordination
Ecosystem: Services \& Support (ACCESS) program~\cite{access2023}, which is
supported by U.S. National Science Foundation grants \#2138259, \#2138286,
\#2138307, \#2137603, and \#2138296.

\paragraph{Competing interests.}
The authors declare that they have no competing interests.

\paragraph{Data availability.}
No external datasets were used. Every matrix in \Cref{s:result} is generated
procedurally from the discretizations and kernels described there. The
numerical data generated during this study, including the result files
supporting the tables and figures, are available from the corresponding
author upon reasonable request.

\paragraph{Code availability.}
The MATLAB implementation requires no third-party toolboxes and is publicly
available, together with the deterministic drivers and job scripts that
generate every table in \Cref{s:result}, at
\url{https://github.com/SMART-Algebra/Adaptive-Matfree-HSS}; release v1.0.0
produces the results reported here.

\paragraph{Author contributions.}
Chao Chen: Conceptualization, Methodology, Validation, Writing (review and
editing), Supervision. Elly Do: Methodology, Software, Investigation, Writing
(original draft). Under Chao Chen's supervision, Elly Do independently
developed the technical details, implemented the algorithms, conducted the
numerical experiments, and prepared the original manuscript as a doctoral
researcher. Both authors read and approved the final manuscript.

\appendix
\section{Diagnostics and Parameter Choices}\label{app:parameters}

We retain the evidence supporting unregularized local recovery and
last-block rank pruning. The Supplementary Material, appended after the references, contains the sampling
protocol, solve-accuracy and Helmholtz resonance diagnostics,
level-by-level conditioning measurements, and block-size/pruning studies.

\subsection{Conditioning of the Local Least-Squares Problems}
\label{ss:ls-conditioning}

At every node, the recovery in \Cref{ss:x_rr_rs_sr} solves an overdetermined
system with coefficient matrix $\tilde\Omega_i^\top$. We record its aspect
ratio $\rho$ (number of equations divided by number of unknowns) and
condition number $\kappa(\tilde\Omega_i)$. \Cref{tab:ls-conditioning}
summarizes four representative problems; the Supplementary Material,
Section~S4, gives the measurements at each level.

\begin{table}[htbp]
\centering
\small
\caption{\em Local least-squares conditioning at $\epsilon=10^{-6}$,
leaf size $256$, and the block sizes of \Cref{ss:setup}; $N=8192$ in 1D
and $N=16384$ in 2D. Each range spans the reported tree levels; the
$\kappa$ column gives the range of level medians, not extrema over all
nodes. Each problem uses one fixed sample stream.}
\label{tab:ls-conditioning}
\begin{tabular}{lrr}
\toprule
Problem & Aspect ratio $\rho$ & Level-median $\kappa$\\
\midrule
1D Poisson Schur & $1.12$--$10.29$ & $1.8$--$31.4$\\
2D Poisson Schur & $1.27$--$2.75$ & $3.9$--$16.4$\\
2D Mat\'{e}rn & $1.25$--$2.11$ & $5.3$--$17.4$\\
Lippmann--Schwinger & $1.20$--$1.75$ & $6.8$--$20.7$\\
\bottomrule
\end{tabular}
\end{table}

The largest values occur at nearly square 1D leaves, with a median of
$31.4$ and a maximum of $39$ over nodes. These measurements did not
indicate a need for regularization in the tested problems, so we use
unregularized QR recovery. They do not, by themselves, bound the error in
the recovered blocks or the complete factorization.

For comparison, the Marchenko--Pastur spectral edges give the
large-matrix Gaussian benchmark
\begin{equation}\label{eq:mp-kappa}
  \kappa_{\mathrm{MP}}(\rho)
  =\frac{\sqrt\rho+1}{\sqrt\rho-1}.
\end{equation}
The measured level medians track this benchmark in both real and complex
arithmetic, including at coarser levels. This supports the practical reuse
of transformed samples, but does not establish that they remain Gaussian
or independent of the computed bases.

\subsection{Sampling Block Size and Rank Pruning}
\label{ss:block-size}

A larger block size $b$ gives fewer and wider updates and a lower-variance
stopping estimator, at the cost of coarser rank resolution. At a node of
dimension $m$, the implementation uses
\begin{equation}\label{eq:b-eff}
  b_{\mathrm{eff}}=\min\!\bigl(b,\lfloor m/4\rfloor\bigr),
\end{equation}
so the rank search retains at least four resolvable steps. At $N=16384$
and $\epsilon=10^{-9}$, increasing $b$ through powers of two from $4$ to
$128$ reduced local factorization time by factors of $2.4$--$3.9$ on the
two-dimensional problems. Very small blocks were also less accurate in
these tests. These observations motivate the defaults $b=16$ in 1D and
$b=32$ in 2D; the full timing, storage, and accuracy diagnostics are in
the Supplementary Material, Section~S5. The timing reductions exclude matvecs
and therefore need not translate into comparable end-to-end speedups.

Rank pruning reduces the sensitivity of storage to this choice.
\Cref{tab:prune-stability} compares three pruning scopes over the same
sweep. Pruning at every level gives the smallest storage variation on all
four problems. In the bounded-rank Poisson case, the max/min storage ratio
falls from $1.397$ without pruning to $1.003$ with pruning at every level.
Thus, pruning permits wider updates without a comparable storage penalty.

\begin{table}[htbp]
\centering
\small
\caption{\em Max/min factor-storage ratio over $b\in\{4,\dots,128\}$ for
three pruning scopes at $N=16384$, $\epsilon=10^{-9}$, and leaf size
$256$, using medians over four matched sample streams. Leaf ranks are
medians at the default block size ($b=16$ in 1D and $b=32$ in 2D).}
\label{tab:prune-stability}
\begin{tabular}{lcccc}
\toprule
Problem & leaf rank $k$ & no pruning & leaves only & every level\\
\midrule
1D Poisson            & $19$  & $1.397$ & $1.010$ & $\mathbf{1.003}$\\
2D Poisson            & $216$ & $1.091$ & $1.088$ & $\mathbf{1.068}$\\
2D Mat\'{e}rn         & $176$ & $1.313$ & $1.215$ & $\mathbf{1.198}$\\
Lippmann--Schwinger   & $119$ & $1.300$ & $1.194$ & $\mathbf{1.058}$\\
\bottomrule
\end{tabular}
\end{table}

\bibliographystyle{spmpsci}
\bibliography{biblio}

\clearpage
\includepdf[pages=-]{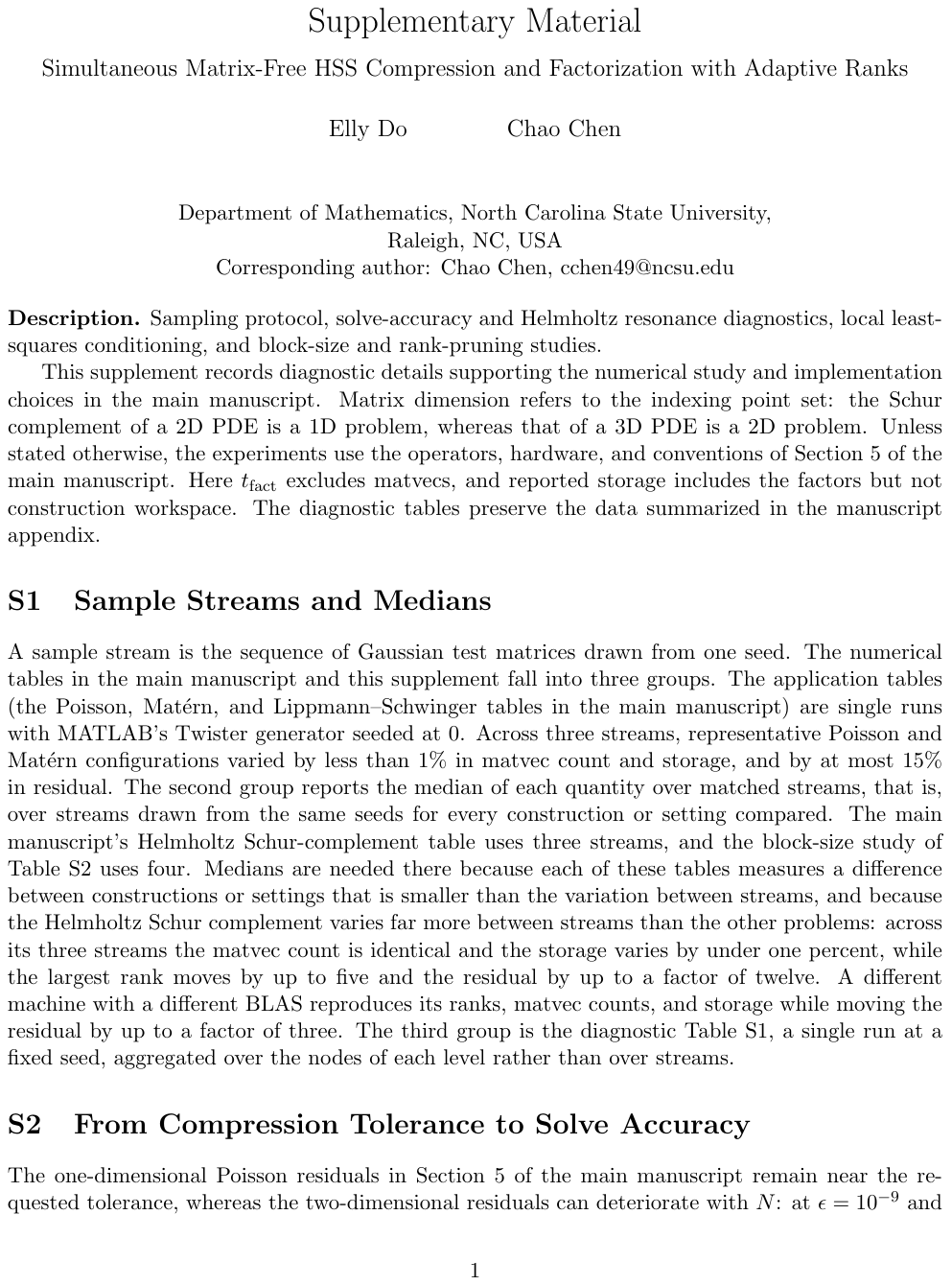}

\end{document}